\documentclass{amsart}
\usepackage{amssymb}
\usepackage{color}
\usepackage{amscd}
\usepackage{verbatim}
\usepackage{epsfig}
\usepackage{hyperref}

\DeclareMathOperator{\ind}{Index}

\DeclareMathOperator{\Vol}{Vol}

\begin{document}
\newcommand\lprime{l}
\newcommand\llprime{l'}
\newcommand\Mand{\ \text{and}\ }
\newcommand\Mor{\ \text{or}\ }
\newcommand\Mfor{\ \text{for}\ }
\newcommand\Real{\mathbb{R}}
\newcommand\RR{\mathbb{R}}
\newcommand\im{\operatorname{Im}}
\newcommand\re{\operatorname{Re}}
\newcommand\sign{\operatorname{sign}}
\newcommand\sphere{\mathbb{S}}
\newcommand\BB{\mathbb{B}}
\newcommand\HH{\mathbb{H}}
\newcommand\CH{\Cx\HH}
\newcommand\dS{\mathrm{dS}}
\newcommand\ZZ{\mathbb{Z}}
\newcommand\codim{\operatorname{codim}}
\newcommand\Sym{\operatorname{Sym}}
\newcommand\Ann{\operatorname{Ann}}
\newcommand\End{\operatorname{End}}
\newcommand\Span{\operatorname{span}}
\newcommand\Ran{\operatorname{Ran}}
\newcommand\ep{\epsilon}
\newcommand\Cinf{\cC^\infty}
\newcommand\dCinf{\dot \cC^\infty}
\newcommand\CI{\cC^\infty}
\newcommand\dCI{\dot \cC^\infty}
\newcommand\Cx{\mathbb{C}}
\newcommand\Nat{\mathbb{N}}
\newcommand\dist{\cC^{-\infty}}
\newcommand\ddist{\dot \cC^{-\infty}}
\newcommand\pa{\partial}
\newcommand\Card{\mathrm{Card}}
\renewcommand\Box{{\square}}
\newcommand\Ell{\mathrm{Ell}}
\newcommand\WF{\mathrm{WF}}
\newcommand\WFh{\mathrm{WF}_\semi}
\newcommand\WFb{\mathrm{WF}_\bl}
\newcommand\Vf{\mathcal{V}}
\newcommand\Vb{\mathcal{V}_\bl}
\newcommand\Vsc{\mathcal{V}_\scl}
\newcommand\Vz{\mathcal{V}_0}
\newcommand\Lb{L_{\bl}}
\newcommand\Hb{H_{\bl}}
\newcommand\Hbpm{H_{\bl,+-}}
\newcommand\Hbpmpm{H_{\bl,\pm\pm}}
\newcommand\Hbpp{H_{\bl,++}}
\newcommand\Hbmm{H_{\bl,--}}
\newcommand\Ker{\mathrm{Ker}}
\newcommand\Coker{\mathrm{Coker}}
\newcommand\Range{\mathrm{Ran}}
\newcommand\Hom{\mathrm{Hom}}
\newcommand\Id{\mathrm{Id}}
\newcommand\sgn{\operatorname{sgn}}
\newcommand\ff{\mathrm{ff}}
\newcommand\tf{\mathrm{tf}}
\newcommand\esssupp{\operatorname{esssupp}}
\newcommand\supp{\operatorname{supp}}
\newcommand\vol{\mathrm{vol}}
\newcommand\Diff{\mathrm{Diff}}
\newcommand\Diffsc{\mathrm{Diff}_\scl}
\newcommand\Diffb{\mathrm{Diff}_\bl}
\newcommand\DiffbI{\mathrm{Diff}_{\bl,I}}
\newcommand\Diffz{\mathrm{Diff}_0}
\newcommand\Psih{\Psi_{\semi}}
\newcommand\Psihcl{\Psi_{\semi,\cl}}
\newcommand\Psisc{\Psi_\scl}
\newcommand\Psib{\Psi_\bl}
\newcommand\Psibh{\Psi_{\bl,\semi}}
\newcommand\Psibc{\Psi_{\mathrm{bc}}}
\newcommand\TbC{{}^{\bl,\Cx} T}\
\newcommand\Tb{{}^{\bl} T}
\newcommand\Tbc{{}^{\bl} \overline{T}}
\newcommand\Sb{{}^{\bl} S}
\newcommand\Tsc{{}^{\scl} T}
\newcommand\Ssc{{}^{\scl} S}
\newcommand\Nb{{}^{\bl} N}
\newcommand\SNb{{}^{\bl} SN}
\newcommand\Lambdab{{}^{\bl} \Lambda}
\newcommand\zT{{}^{0} T}
\newcommand\Tz{{}^{0} T}
\newcommand\zS{{}^{0} S}
\newcommand\dom{\mathcal{D}}
\newcommand\cA{\mathcal{A}}
\newcommand\cB{\mathcal{B}}
\newcommand\cE{\mathcal{E}}
\newcommand\cG{\mathcal{G}}
\newcommand\cH{\mathcal{H}}
\newcommand\cU{\mathcal{U}}
\newcommand\cO{\mathcal{O}}
\newcommand\cF{\mathcal{F}}
\newcommand\cM{\mathcal{M}}
\newcommand\cQ{\mathcal{Q}}
\newcommand\cR{\mathcal{R}}
\newcommand\cI{\mathcal{I}}
\newcommand\cL{\mathcal{L}}
\newcommand\cK{\mathcal{K}}
\newcommand\cC{\mathcal{C}}
\newcommand\cX{\mathcal{X}}
\newcommand\cY{\mathcal{Y}}
\newcommand\cP{\mathcal{P}}
\newcommand\cS{\mathcal{S}}
\newcommand\cZ{\mathcal{Z}}
\newcommand\cW{\mathcal{W}}
\newcommand\Ptil{\tilde P}
\newcommand\ptil{\tilde p}
\newcommand\chit{\tilde \chi}
\newcommand\yt{\tilde y}
\newcommand\zetat{\tilde \zeta}
\newcommand\xit{\tilde \xi}
\newcommand\taut{{\tilde \tau}}
\newcommand\phit{{\tilde \phi}}
\newcommand\mut{{\tilde \mu}}
\newcommand\sigmat{{\tilde \sigma}}
\newcommand\sigmah{\hat\sigma}
\newcommand\zetah{\hat\zeta}
\newcommand\etah{\hat\eta}
\newcommand\loc{\mathrm{loc}}
\newcommand\compl{\mathrm{comp}}
\newcommand\reg{\mathrm{reg}}
\newcommand\bl{{\mathrm b}}
\newcommand\scl{{\mathrm{sc}}}
\newcommand{\sH}{\mathsf{H}}
\newcommand{\cte}{\digamma}
\newcommand\cl{\operatorname{cl}}
\newcommand\Div{\operatorname{div}}
\newcommand\hilbert{\mathfrak{X}}

\newcommand\Hh{H_{\semi}}

\newcommand\bM{\bar M}

\newcommand\xib{{\underline{\xi}}}
\newcommand\etab{{\underline{\eta}}}
\newcommand\zetab{{\underline{\zeta}}}

\newcommand\xibh{{\underline{\hat \xi}}}
\newcommand\etabh{{\underline{\hat \eta}}}
\newcommand\zetabh{{\underline{\hat \zeta}}}

\newcommand\psit{\tilde\psi}
\newcommand\rhot{{\tilde\rho}}

\newcommand\hM{\hat M}

\newcommand\Op{\operatorname{Op}}
\newcommand\Oph{\operatorname{Op_{\semi}}}

\newcommand\innr{{\mathrm{inner}}}
\newcommand\outr{{\mathrm{outer}}}
\newcommand\full{{\mathrm{full}}}
\newcommand\semi{\hbar}

\newcommand\elliptic{\mathrm{ell}}
\newcommand\even{\mathrm{even}}

\newcommand\past{\mathrm{past}}
\newcommand\future{\mathrm{future}}

\newcommand{\sfs}{\mathsf{s}}
\newcommand{\sC}{\mathsf{C}}
\newcommand{\sK}{\mathsf{K}}
\newcommand{\Ham}{\mathsf{H}}
\newcommand{\Hamb}{\mathsf{H}_\bl}

\newcommand \p {\partial}
\newcommand \absv [1]{\left \lvert #1 \right \rvert }
\newcommand \wt [1]{\widetilde{#1}}
\newcommand{\lp}{\left(}
\newcommand{\rp}{\right)}
\newcommand{\la}{\langle}
\newcommand{\ra}{\rangle}
\newcommand{\norm}[2][]{\left \| #2 \right \|_{#1} }
\newcommand \lra {\longrightarrow}
\newcommand{\set}[1]{\left\{ #1 \right\} }
\newcommand \specb {spec_\bl}
\newcommand \ov {\overline{u}}
\newcommand \notg {h}
\newcommand\cXpm{\mathcal{X}_{+-}}
\newcommand\cYpm{\mathcal{Y}_{+-}}
\newcommand\cXmp{\mathcal{X}_{-+}}
\newcommand\cYmp{\mathcal{Y}_{-+}}
\setcounter{secnumdepth}{3}
\newtheorem{lemma}{Lemma}[section]
\newtheorem{prop}[lemma]{Proposition}
\newtheorem{thm}[lemma]{Theorem}
\newtheorem{cor}[lemma]{Corollary}
\newtheorem{result}[lemma]{Result}
\newtheorem*{thm*}{Theorem}
\newtheorem*{prop*}{Proposition}
\newtheorem*{cor*}{Corollary}
\newtheorem*{conj*}{Conjecture}
\numberwithin{equation}{section}
\theoremstyle{remark}
\newtheorem{rem}[lemma]{Remark}
\newtheorem*{rem*}{Remark}
\theoremstyle{definition}
\newtheorem{Def}[lemma]{Definition}
\newtheorem*{Def*}{Definition}

\newcommand{\mar}[1]{{\marginpar{\sffamily{\scriptsize #1}}}}
\newcommand\av[1]{\mar{AV:#1}}
\newcommand{\jgr}[1]{{\mar{JGR:#1}}}
\newcommand{\nh}[1]{{\mar{NH:#1}}}

\renewcommand{\theenumi}{\roman{enumi}}
\renewcommand{\labelenumi}{(\theenumi)}

\title{The Feynman propagator on perturbations of Minkowski space}
\author{Jesse Gell-Redman}
\address{Department of Mathematics, Johns Hopkins University, MD
  21218}
\email{jgell@math.jhu.edu}
\author{Nick Haber}
\address{MSRI, Berkeley and McGill University, Montreal.}
\email{nhaber@stanford.edu}
\author[Andras Vasy]{Andr\'as Vasy}
\address{Department of Mathematics, Stanford University, CA 94305-2125, USA}
\email{andras@math.stanford.edu}


\thanks{The third author gratefully
  acknowledges partial support from the NSF under grant number  
  DMS-1068742 and DMS-1361432.}

\begin{abstract}
In this paper we analyze the Feynman wave equation on Lorentzian
scattering spaces.  We prove that the Feynman propagator exists as a
map between certain Banach spaces defined by decay and microlocal
Sobolev regularity properties.  We go on to show that certain nonlinear
wave equations arising in QFT are well-posed for small data in the
Feynman setting.
\end{abstract}

\maketitle

\section{Introduction}
In this paper we use the method introduced in
\cite{VD2013}, extended in
\cite{Baskin-Vasy-Wunsch:Radiation} and \cite{HVsemi},
to analyze the Feynman propagator on spaces $(M,g)$, called spaces
with non-trapping Lorentzian
scattering metrics (a notion recalled in detail in Section~\ref{sec:geometry}), that at infinity resemble
Minkowski space in an appropriate manner. As the Feynman propagator is
of fundamental importance in quantum field theory, we expect that our
result and methods will be useful in a
systematic treatment of QFT on curved, non-static, Lorentzian backgrounds.

Here the Feynman propagator is defined as the inverse of the wave
operator acting as a map between appropriate function spaces that
generalizes the behavior of the standard Feynman propagator on exact
Minkowski space. Concretely, the distinguishing feature of propagators
from the perspective of function spaces is in terms of the differential order of
the Sobolev spaces at the two halves of the (b-)conormal bundle of the
`light cone at infinity' $S_\pm$; components at which the differential
order is higher, resp.\ lower, than a threshold value determine the inverse one obtains.
Thus, {\em we set up function spaces which are
weighted microlocal Sobolev spaces of variable order of an appropriate kind such that
the wave operator for any non-trapping Lorentzian scattering metric is Fredholm for all but a discrete set of
weights -- See Theorem~\ref{thm:fredholm}} for a precise
statement. Indeed, the same statement holds for more general
perturbations of Lorentzian scattering metrics in the sense of smooth
sections of $\Sym^2\Tsc^*M$, defined below in
Section~\ref{sec:geometry}. Further, we prove Theorem
    \ref{thm:invertibility} below, which we state roughly now.

\begin{thm*}[See Theorem~\ref{thm:invertibility}]
For perturbations of Minkowski space, in
the sense of smooth
    sections of $\Sym^2\Tsc^*M$, the
    Feynman wave operator, described above, is invertible for a
    suitable range of weights (rates of decay or growth of functions
    in the domain).  That is to say, its inverse, the Feynman
    propagator, exists for these space-times.
\end{thm*}

In order to give a rough idea for what {\em the} Feynman propagator is
we recall that in their groundbreaking paper \cite{FIOII} Duistermaat and H\"ormander constructed distinguished parametrices
for wave equations, i.e.\ distinguished solution operators for $\Box
u=f$ modulo $\CI(M^\circ)$. Recall that by H\"ormander's theorem \cite{Hormander:Existence},
singularities of solutions of wave equations propagate along
bicharacteristics inside the characteristic set in phase space, i.e.\ $T^*M^\circ$; the projections of
these to the base space are null-geodesics. Here a bicharacteristic is
an integral curve of the Hamilton vector field of the principal symbol
of the wave operator, which is the dual metric function on
$T^*M^\circ$. For the inhomogeneous wave equation, $\Box u=f$, if,
say, $f$ has wave front set (i.e.\ is singular) at only one point in
$T^*M^\circ$, the different distinguished parametrices produce
solutions with different wave front sets, namely either the forward or
the backward bicharacteristic through the point in question. Here
forward and backward are measured relative to the vector field whose
integral curves they are, i.e.\ the Hamilton vector field. Note,
however, that there is a different notion of forward and backward,
which one may call future- or past-orientedness, namely whether the
underlying time function is increasing or decreasing along the
flow. The relative sign between these notions is the opposite in the
two halves of the characteristic set of the wave operator over
each point. We point out that from the perspective of microlocal
analysis the natural direction of propagation {\em is} given by the Hamilton flow.

As explained by Duistermaat and H\"ormander, a distinguished parametrix is obtained by choosing a
direction of propagation (of singularities, or estimates) {\em in each
  connected component of the characteristic set of the wave
  operator}. Here the direction of propagation is relative to the
Hamilton flow, as above. If the underlying manifold is connected, as one may assume, the characteristic
set has two connected components, and there are $2^2=4$ choices:
propagation forward relative to the Hamilton flow everywhere,
propagation backward along the Hamilton flow everywhere (these are the
Feynman and anti-Feynman propagators), resp.\ propagation in the
future direction everywhere (the retarded propagator) and in the past
direction everywhere (the advanced propagator).
A parametrix, however, is only an approximate
inverse, modulo smoothing --- smoothing operators are not even compact
on such a manifold; for actual applications (such as any computations in physics)
one would need an actual inverse, and most importantly a {\em notion
  of an inverse}.  This is exactly what we provide in Theorem
\ref{thm:invertibility} below.

The historically usual setup for
wave equations, and more generally evolution equations, is that of
Cauchy problems: one specifies initial data at a time slice, and then
one studies local or global solvability. In this sense wave equations
are always locally well-posed due to the finite speed of propagation,
which in turn is proved by energy estimates.  Global well-posedness follows if the local solutions 
can be pieced together well: global hyperbolicity is a notion that 
allows one to do so. If one turns this into a
setup of inhomogeneous wave equations, $\Box u=f$, by cutting $f$ into
two pieces, located in the future, resp.\ the past, of a Cauchy
surface, the choice one is making is that the {\em support} of $u$ be
in the future, resp.\ the past, of that of $f$. This necessarily
implies, indeed is substantially stronger than, the statement that singularities of solutions are accordingly propagated, so
two of the Duistermaat-H\"ormander parametrices correspond to
these. Thus, due to the energy estimates, even when one considers
global solutions, the Cauchy problem, or equivalently the future (or
past) oriented problem, for the wave equation is essentially local in
character, though, as discussed in \cite{VD2013,HVsemi}, in order to
understand the global behavior of solutions, it is extremely useful to
work directly in a global framework in any case.

What we achieve here is to give an analogous well-posedness framework
for the Feynman problems (as opposed to the Cauchy problems). These
problems are {\em necessarily} global in character, very much unlike
the Cauchy problems. Thus, they behave  similarly, in a certain sense,
to elliptic PDE. Indeed, from our perspective, it is an accident
(happening for
good reasons) that the future/past oriented wave equations are local;
one should not normally expect this for any PDE. To be more precise,
singularities of solutions behave just as predicted by the
Duistermaat-H\"ormander construction, but this has no content for
$\CI$ solutions --- the $\CI$ `part' of solutions is globally
determined.

There has been extensive work in the mathematical physics literature
on such QFT problems, often from the perspective of trying to make
sense of division by functions with zeros on the characteristic set:
for Minkowski space, the Fourier transform gives rise to a multiplier
$\xi_n^2-(\xi_1^2+\ldots+\xi_{n-1}^2)$; in a $\pm i0$ sense division
by this is well-behaved away from the origin, but at the origin
delicate questions arise. This is usually thought of as a degree of
freedom in defining propagators: precisely how one extends the
distribution to $0$ even in this constant coefficient setting. (See
\cite[Section~5]{Brunetti-Fredenhagen:Microlocal} for a discussion of
this in the QFT context, and
\cite{Viet:Renormalization} for a recent treatment of renormalization
as such extensions.) From our perspective, this is due to
translational invariance of the problem being emphasized at the
expense of its homogeneity; Mellin transforming in the radial variable
gives rise to a much better behaved problem. Indeed, a generalization
of this is what Melrose's framework of b-analysis \cite{Melrose:Atiyah} relies on; we
further explore it here in the non-elliptic setting following
\cite{VD2013,HVsemi}. When the microlocal structure of the function
spaces on which the wave
operator acts corresponds to the above propagation statements
(propagation in the direction of the Hamilton flow in the Feynman
case, and in the opposite direction in the anti-Feynman case), the remaining choice is
that of a weight: in the case of Minkowski space it turns out that
weights $l$ with $|l|<\frac{n-2}{2}$ give rise to invertibility, while
outside this range the index of the operator changes, with jumps at weight values
corresponding to resonances of the Mellin transformed wave operator
family, which in turn correspond to eigenvalues of the Laplacian on
the sphere $\Delta_{\sphere^{n-1}}$ as we show by a complex scaling
(Wick rotation) argument in Section \ref{sec:Wick}.

For QFT on curved space-times, the work of Duistermaat and H\"ormander was
used to introduce a microlocal characterization of Hadamard states,
which are considered as physical states of non-interacting QFT, by
Radzikowski \cite{Radzikowski:Microlocal}. (Indeed, part of the paper
of Duistermaat and H\"ormander was motivated by QFT questions.) This in turn was then extended by
Brunetti, Fredenhagen and K\"ohler
\cite{Brunetti-Fredenhagen-Kohler:Microlocal,
  Brunetti-Fredenhagen:Microlocal}. G\'erard and Wrochna gave a new
pseudodifferential construction of Hadamard states
\cite{Gerard-Wrochna:Construction, Gerard-Wrochna:Yang-Mills}. In a different direction, Finster and Strohmaier extended the general theory to Maxwell fields \cite{Finster-Strohmaier:Gupta}. However, in all these cases, there
is no way of fixing a preferred state: one is always working modulo
smoothing operators. Our framework on the other hand gives exactly
such a preferred choice.  Note also that the Feynman propagator we
construct relates to an Hadamard-type condition; see Remark
\ref{rem:hadamard} below.

In the settings with extra structure, involving time-like Killing
vector fields, one can construct Feynman propagators in terms of
elliptic operators, e.g.\ via Cauchy data. Other constructions
(such as extensions across null-infinity) in similar
settings are investigated by Dappiaggi, Moretti and Pinamonti
\cite{Dappiaggi-Moretti-Pinamonti:Rigorous, Moretti:Quantum,
  Dappiaggi-Moretti-Pinamonti:Cosmological}. In fact, these latter results
bear the closest connections to ours in that a canonical state is
constructed using the structure on null-infinity. Our results deal
directly with the `bulk', thanks to the Fredholm formulation, with the
linear results having considerable perturbation stability in
particular. (It is due to the module structure required in
Section~\ref{sec:semi} that the non-linear problem is more restrictive.)

Along with setting up such a Fredholm framework, we also study
semilinear wave equations, following the general scheme of
\cite{HVsemi}; we think of these as a first step towards
interacting QFT in this setting. However, being fully microlocal, the
necessary framework requires more sophisticated function spaces than
those discussed in \cite{HVsemi}.  We prove small data
well-posedness results in the Feynman setting for certain semilinear
wave equations in Theorems \ref{thm:semilinear} and
\ref{thm:semilinearthree} below.  In particular, Theorem
\ref{thm:semilinearthree} can be summarized as follows
\begin{thm*}
In $\RR^{3 + 1}$, if $g$ is a perturbation of the Minkowski metric for
which both the
invertibility statements in Theorem \ref{thm:invertibility} and
Theorem \ref{thm:modulereginverse} hold, the problem
 \begin{equation*}
   \Box u + \lambda u^3 = f
 \end{equation*}
is well-posed for small $f$, where $f$ lies in the range, and $u$ in
the domain, of the Feynman wave operator, in particular $u =
\Box^{-1}_{g, fey}(f - \lambda u^3)$ where $\Box^{-1}_{g, fey}$ is the
Feynman propagator mapping as in \eqref{eq:feynmanprop} with $l \ge 0$
sufficiently small.
\end{thm*}

While as far as we are aware non-linear problems have not been
considered in the Feynman context,
for the usual Cauchy problem, i.e.\ the retarded and advanced
propagators, non-linear problems on Minkowski space, as well as
perturbations of Minkowski space (as opposed to the more general
Lorentzian scattering metrics considered in the linear parts of the
paper here), have been very well studied. In particular, even
quasilinear equations are well 
understood due to the work of Christodoulou \cite{Christodoulou:Global} and
Klainerman \cite{Klainerman:Uniform, Klainerman:Null}, with their
book on the global stability of Einstein's equation
\cite{Christodoulou-Klainerman:Global} being one of the main
achievements. Lindblad and Rodnianski
\cite{Lindblad-Rodnianski:Global-existence,
  Lindblad-Rodnianski:Global-Stability} simplified some of their
arguments, and Bieri \cite{Bieri:Extensions, Bieri-Zipser:Extensions}
relaxed some of the decay conditions. We also mention the work of Wang
\cite{Wang:Thesis} obtaining asymptotic expansions, of Lindblad \cite{Lindblad:Global} for results on a
class of quasilinear equations, and of Chru\'sciel and {\L}{\c{e}}ski \cite{Chrusciel-Leski:Polyhomogeneous}
on improvements when there are no derivatives in the non-linearity. H\"ormander's book
\cite{Hormander:Nonlinear} provides further references in the general
area, while the work of Hintz and Vasy \cite{HVsemi} develops the
analogue of the framework we use here in the general Lorentzian
scattering metric setting (but still for the Cauchy problem). Works for the linear problem with implications for
non-linear ones, e.g.\ via Strichartz estimate include the recent work of Metcalfe and Tataru
\cite{Metcalfe-Tataru:Global} where a parametrix construction is
presented in a low
regularity setting.

The structure of the paper is as follows. In
Section~\ref{sec:geometry} we describe the underlying geometry and
study the wave operator microlocally in the sense of smoothness (as
opposed to decay). Estimates modulo compact errors, and thus Fredholm
properties, are established in Section~\ref{sec:fred}. In
Section~\ref{sec:Wick} we show that in Minkowski space, the Feynman
propagator is the limit of the inverses of elliptic problems,
achieved by a `Wick rotation'; this means that from the perspective of
spectral theory the Feynman and anti-Feynman propagators are the
natural replacement for resolvents. This in particular establishes the
invertibility of the Minkowski wave operator on the appropriately
weighted function spaces. Finally, in Section~\ref{sec:semi} we study
semilinear wave equations in the Feynman framework.

The third author is grateful to Jan Derezi\'nski, Christian G\'erard
and Michal Wrochna for very helpful discussions, and the authors are
grateful to Peter Hintz for comments on the manuscript. The authors
are also grateful to the anonymous referees for comments that helped to improve the presentation.

\section{Geometry and the d'Alembertian}\label{sec:geometry}
The basic object of interest
is a manifold $M$ with boundary $\pa M$ equipped with a Lorentzian
metric $g$ (which we take to be signature $(1,n-1)$) in its interior
which has a certain form at the boundary (which is geometrically
infinity) modelled on the Minkowski metric.  In order to define the
precise class of metrics, it is useful to introduce a more general structure. Thus, $\Tsc^{*}M$ is
the scattering cotangent bundle, which we describe presently,
originally defined in \cite{RBMSpec}.  If $\rho$
is a boundary defining function, meaning a function in $C^\infty(M)$
which is non-negative, has $\{ \rho = 0 \} = \p M$, and such that
$d\rho$ is non-vanishing on $\p M$,
smooth sections of $\Tsc^*M$ near the boundary are locally given by $C^{\infty}(M)$ linear
combinations of the differential forms
\begin{equation*}
  \frac{d\rho}{\rho^{2}} ,\  \frac{dw_{i}}{\rho},
\end{equation*}
where $w_{1}, \dots, w_{n-1}$
form local coordinates on $\p M$.
A non-degenerate smooth section of $\Sym^2\Tsc^*M$ of Lorentzian
signature (which we take to be $(1,n-1)$) is called a {\em Lorentzian sc-metric}. The
smooth topology on sc-metrics is the $\CI$ topology on sections of
$\Sym^2\Tsc^*M$, i.e.\ locally in $M$ (which recall is a manifold {\em
  with} boundary, i.e.\ smoothness is {\em up to} the boundary) is induced by the $\CI$ topology
of the coefficients of the basis
\begin{equation}\label{eq:symmetric-tensor-basis}
  \frac{d\rho}{\rho^{2}}\otimes \frac{d\rho}{\rho^{2}} ,\
  \frac{d\rho}{\rho^{2}}\otimes_s\frac{dw_{i}}{\rho},\
  \frac{dw_{i}}{\rho}\otimes_s\frac{dw_{j}}{\rho},\ i,j=1,\ldots,n-1,\
  i\leq j,
\end{equation}
where $\otimes_s$ is the symmetric tensor product.

When $M^{int} = \mathbb{R}^n$, the objects above can be described
along more familiar lines.  Indeed, in this case, the radial
compactification of $\RR^n$ to a ball $\BB^n$ gives the manifold with
boundary $M$, see \cite{RBMSpec}, e.g.\ by using `reciprocal
spherical coordinates' to glue the sphere at infinity $\sphere^{n-1}$
to $\RR^n$. Then if $z_1, \dots, z_n$ are the standard Euclidean
coordinates, setting as usual $r = (z_1^2 + \dots + z_n^2)^{1/2}$, we
can take $\rho = 1/r$ outside the unit ball $B_1(0)$ (recall
that $\rho$ is to be a smooth function on all of $M$), and thus 
\begin{equation}\label{eq:Euclidean-radial-scattering-form}
\frac{d\rho}{\rho^2} = - dr \mbox{ on } \mathbb{R}^n\setminus B_1(0),
\end{equation}
and the $w_i$ can be taken to be some set of $n-1$ angular variables on
$\mathbb{S}^{n-1}$ which give local coordinates on the sphere.  Thus in this
case $dw_i / \rho = r \,d\omega_i$ and one sees that the volume form
$dz = r^{n-1}\,dr \,d\Vol_{\mathbb{S}^{n-1}}$ is a smooth non-vanishing section
of the top degree form bundle, $\wedge^n\Tsc^*M$, up to and including $\p M$.  (Put differently, the volume
form is equal to a wedge product of $d\rho /\rho^2$ and the $dw_i
/\rho$ times a smooth non-vanishing function $a$ which is smooth up to
$r = \infty$.)
Moreover, $\CI(\BB^n)$ consists exactly of the
space of classical (one step polyhomogeneous) symbols of order $0$,
while the standard coordinate differentials $dz_j$ lift to $\BB^n$ to
give a basis, over $\CI(\BB^n)$, of all smooth sections of
$\Tsc^*\BB^n$. In particular, any translation invariant Lorentzian
metric on $\RR^n$ is (after this identification) a sc-metric; and
remains so under perturbations of its coefficients by classical symbols of order
$0$. Moreover, $dz_i\otimes_s dz_j$, $i,j=1,\ldots,n$, $i\leq j$,
forming a basis of $\Sym^2\Tsc^*\BB^n$, the $\CI$ topology on sections
of $\Sym^2\Tsc^*\BB^n$ is simply the $\CI(\BB^n)$ topology on the
$n(n+1)/2$-tuple of coefficients with respect to this basis.

We next recall the
definition of the more refined structure of a {\em Lorentzian scattering space} from
\cite{Baskin-Vasy-Wunsch:Radiation} (see also \cite[Section~5]{HVsemi}), of which the Minkowski metric is
an example via the radial compactification of $\RR^n$, depicted in
Figure \ref{spacetime}. For this, we
assume that there is a $\CI$ function $v$ defined near $\pa M$, with
$v|_{\pa M}$ having a
non-degenerate differential at the zero-set $S=\{v=0,\rho=0\}$ of $v$
in $\pa M$ (which we call the {\em light cone at infinity}); here
$\rho$ is a boundary defining function with the property that the
scattering normal vector field $V=\rho^2\pa_\rho$ modulo $\rho\Vsc(M)$
(it is well-defined in this sense) satisfies that $g(V,V)$ has the
same sign as $v$ at each point in $\pa M$, $g$ has the form
\begin{equation}\label{eq:metric}
g=v\frac{d\rho^2}{\rho^4}-\left(\frac{d\rho}{\rho^2}\otimes\frac{\alpha}{\rho}+\frac{\alpha}{\rho}\otimes\frac{d\rho}{\rho^2}\right)-\frac{\tilde g}{\rho^2},
\end{equation}
where $\tilde g\in\CI(M;\Sym^2 T^*M)$, $\alpha\in\CI(M;T^*M)$,
$\alpha|_S=\frac{1}{2}\,dv$ and
$\tilde g|_{\Ann(d\rho,dv)}$ at $S$ is positive definite, where, for a
set of one forms $\beta_1, \dots, \beta_k$,
$\Ann(\beta_1, \dots, \beta_k)$ is the set of vectors in the
intersection of the kernels of the $\beta_i$.  Apart from the fact
that the Minkowski metric is of this form, the assumption in \eqref{eq:metric} is
natural because it guarantees the structure of the Hamiltonian
dynamics in the cotangent bundle which is required for our analysis.

This is not quite a statement about $g|_{\pa M}$ as a metric on $\Tsc
M$, i.e.\ as a section of $\Sym^2\Tsc^*M$, because of the implied
absence of a
$O(\rho)\frac{d\rho^2}{\rho^4}$ term. Adding such a term results in a {\em
  long-range Lorentzian scattering metric}, the whole theory relevant
to the discussion below goes
through in this setting, as shown in the work of Baskin, Vasy and
Wunsch \cite{Baskin-Vasy-Wunsch:Long-range}; e.g.\ Schwarzschild space-time is of this
form near the boundary of the light cone at infinity. (The difference
is in the precise form of the asymptotics of the linear waves; they
are well-behaved on a logarithmically different blow-up of $M$ at
$S$.) 

Note that a perturbation of a Lorentzian scattering metric in the
sense of sc-metrics (smooth sections of $\Sym^2\Tsc^*M$) is a Lorentzian sc-metric, but it need {\em not}
be (even a long-range) Lorentzian scattering metric, since the above form of the metric
\eqref{eq:metric} need not be preserved. However, the subspace of
sc-metrics of the form \eqref{eq:metric} is a closed subset in the
$\CI$ topology of sc-metrics within the open set of Lorentzian
sc-metrics (in the space of smooth sections of $\Sym^2\Tsc^*M$); by a
perturbation in the sense of Lorentzian scattering metrics we
mean a perturbation within this closed subset.

We remark here that, as
is generally the case, only finite regularity (not being $\CI$) is
relevant in any of the discussion below, though the specific
regularity needed would be a priori rather high. However,
using the low regularity results of Hintz \cite{Hintz:Quasilinear} on
b-pseudodifferential operators one could easily obtain rather precise
low-regularity versions of
the linear results presented here.

For statements beyond Fredholm properties, based on the work in Section~\ref{sec:Wick}, $M$ will be the ball $\mathbb{B}^n$, i.e.\ the
radial compactification of $\mathbb{R}^n$, equipped with a smooth
perturbation of the Minkowski metric,
\begin{equation}
  \label{eq:minkowski}
  g = dz_n^{2} - dz_{1}^{2} - dz_2^2 -   \dots
- dz_{n - 1}^{2},
\end{equation}
 with perturbation understood in the set of sc-metrics. (Later, in
 Section~\ref{sec:semi}, it will be important to have perturbations
 within scattering metrics to preserve the module structure discussed there.)
To see that this takes the form in
\eqref{eq:metric}, following \cite[Sect.\
3.1]{Baskin-Vasy-Wunsch:Radiation}, write $g$ in the coordinates $(\rho,
v, \omega)$ defined by $z_n = \rho^{-1} \cos \theta,
z_{j} = \rho^{-1}\omega_{j} \sin\theta$ for $1 \le j \le n - 1$, where $\rho = |z|^{-1}$,
where $|z| = \sqrt{z_1^2 + \dots + z_n^2}$ and $\omega_{j} = z_{j}/ (|z|^{2} -
z_n^{2})^{1/2}$, and take $v = \cos 2 \theta$.   In this case $\alpha = dv / 2$ identically.

The main object of study here is the wave operator, defined in local
coordinates by
\begin{equation}
  \label{eq:box}
  \Box_{g} := \frac{1}{\sqrt{g}} \p_{i} G^{ij} \sqrt{g} \p_{j},
\end{equation}
where $G$ denotes the inverse of $g$, i.e.\ the dual metric on
$1-$forms defined by $g$.

We further
assume that $g$ is \emph{non-trapping}, which is to say we assume that $S=S_+\cup S_-$
(each $S_\pm$ being the disjoint union of possibly several connected
components),
$$
\{\rho=0,\
v>0\}=C_+\cup C_-,
$$
$C_\pm$ open, $\pa C_\pm=S_\pm$, and such that the
null-geodesics of $g$ tend to $S_+$ as the parameter goes to
$+\infty$, $S_-$ as the parameter goes to $-\infty$, or vice versa. We
also let
$$
C_0=\{\rho=0,\ v<0\}.
$$
We then consider $\Box_g$, on functions (or in the future differential forms or various other
squares of Dirac-type operators), and we wish to analyze the
invertibility of the Feynman propagator.

For this purpose it is convenient, as we explain further in the next paragraph, to consider
\begin{equation}\label{eq:Ldef}
L=\rho^{-(n-2)/2}\rho^{-2}\Box_g\rho^{(n-2)/2};
\end{equation}
then $L\in\Diffb^2(M)$, the space of b-differential operators, meaning
that locally near $\p M$, using
coordinates $(\rho, w_{1}, \dots, w_{n - 1})$ where $\rho$ is the boundary defining function
from \eqref{eq:metric} and $w_i$ are any coordinates on $\p M$, there
are smooth functions $a_{i, \alpha} \in C^\infty(M)$, such that
\begin{equation}\label{eq:b-diff-op-form}
L = \sum_{j + |\alpha| \le 2} a_{j, \alpha}(\rho \p_\rho)^j \p^\alpha_{w}.
\end{equation}
Its principal symbol is the dual metric $\hat G$ of the Lorentzian
b-metric 
\begin{equation}\label{eq:bmetric}
\hat g=\rho^2 g.
\end{equation}
In general, $\Diffb^*(M)$ is the algebra of differential operators generated by
\begin{equation}\label{eq:b-vector-fields}
\Vb := C^\infty(M; \Tb(M)),
\end{equation}
which more
concretely is the
$C^{\infty}(M)$ span of the vector fields
\begin{equation}
  \label{eq:bfields}
  \rho \p_{\rho}, \quad \p_{w_{i}},
\end{equation}
That $L$ is indeed in $\Diffb^2(M)$ can be checked directly from \eqref{eq:metric} and \eqref{eq:box}.
In the definition of $L$ in \eqref{eq:Ldef}, $\rho^{(n-2)/2}$ is introduced to
make $L$ formally self-adjoint with respect to the b-metric $\hat g$. The conformal factor $\rho$ merely
reparameterizes null-bicharacteristics, so our assumption is
equivalent to the statement that null-bicharacteristics of $L$ tend to
$S_\pm$.

While we could consider $\Box_g$ or in fact $\Box_g+\lambda$,
$\lambda\in\Cx$, as a scattering differential operator, corresponding
to the sc-structure, we instead work with $L$ because $\Box_g$ is
rather degenerate as a sc-operator due to the quadratic vanishing of
its principal symbol at the zero section at $\pa M$ (i.e.\
infinity). Note that the option of the b-framework is not available if a non-zero spectral
parameter $\lambda$ is added (it would result in a singular term), but
on the other hand in these cases there is no degeneracy in the
operator in the sc-framework!  The latter phenomenon appears also in
analogous work in the elliptic setting, namely in analysis of
$\Delta_g + \lambda$ for $g$ a \emph{Riemannian} scattering metric, in
particular in Melrose's work on scattering manifolds \cite{RBMSpec},
which incidentally proves and uses radial points estimates.  See
also \cite{GH2008}.  

One of the main features of our analysis, parallel to the recent work
\cite{HVquasi, HVsemi} as well as much other work on analysis on non-compact
spaces going back to Melrose \cite{Melrose:Atiyah}, is that we use an
extension of the vector bundle $T^*(M^{int})$ up to the boundary which
is better suited
to the analysis than $T^*M$, and for which in particular the beginnings
and ends of null-bicharacteristics become tractable objects.
Concretely, we use the b-cotangent bundle, $\Tb^*M$, the dual bundle of
the b-tangent bundle $\Tb M$, whose local sections near the boundary
are $C^\infty(M)$ linear combinations of 
\begin{equation*}
  \frac{d\rho}{\rho} ,\qquad  dw_{i},
\end{equation*}
with
coordinates as above.  Notation as in the paragraph containing
\eqref{eq:Euclidean-radial-scattering-form}, in $\mathbb{R}^n$ we can
assume that near infinity the
differential form $d\rho / \rho = - dr /r$.  (The $w_i$ remain
coordinates on the sphere.)

For an operator $P\in\Diffb^m(M)$, the (b-)principal symbol
$\sigma_{\bl,m}(P)$ is a smooth function on $\Tb^*M$ which is a
homogeneous polynomial of degree $m$ on the fibers of
$\Tb^*M$, extending the standard principal symbol from $T^*M^{int}$ to $\Tb^*M$.
Notice that a vector in $\Tb_q M$, $q\in M$, defines a linear function on $\Tb^*_qM$;
the principal symbol of a vector field is given by $i$ times this fiber-linear
function; it is extended to differential operators by making it
multiplicative, in the process keeping only the leading ($m$th order)
terms of the operator. Concretely, writing b-covectors as
$\sigma\,\frac{d\rho}{\rho}+\sum_j\zeta_j\,dw_j$,
$(\rho,w,\sigma,\zeta)$ are local coordinates on $\Tb^*M$ (global in
the fibers), the b-principal symbol of an operator of the form
\eqref{eq:b-diff-op-form} is
$$
\sum_{j+ |\alpha| = 2} a_{j, \alpha}(\rho,w)(i\sigma)^j (i\zeta)^\alpha,
$$
since the symbol of $\rho \p_\rho$ is $i \sigma$ in the coordinates
$(\rho,w,\sigma,\zeta)$, because $\rho \p_\rho$ acting on $\sigma\,\frac{d\rho}{\rho}+\sum_j\zeta_j\,dw_j$
covector is $\sigma$.

We describe the structure of the
null-bicharacteristics at the boundary in detail now.  The Hamilton
flow on null-bicharacteristics
corresponding to $L$ descends from a flow on $T^* (M^{int})$ to a flow
on the spherical cotangent bundle $S^*(M^{int})$.  On can think of
$S^*(M^{int})$ as either the quotient $(T^* (M^{int}) - o) / \mathbb{R}_+$,
where $o$ denotes the zero section and the action of $\mathbb{R}_+$ is
the standard dilation action on the fibers, or as the bundle obtained
by radially compactifying the fibers of $T^*(M^{int})$ to obtain a
ball bundle $\overline{T}^* (M^{int})$ and taking the corresponding sphere
bundle whose fibers are the boundary of fibers of
$\overline{T}^*(M^{int})$.  The latter process can just as well be done on
on $\Tb^* M$ to obtain $\Tbc^*M$ and taking the boundaries of the
fibers gives the spherical b-conormal bundle $\Sb^*M$.  

For the convenience of the reader, we will give a brief summary of the
properties of the Hamilton flow used in the propagation estimates,
which are discussed in detail in \cite[Section
3]{Baskin-Vasy-Wunsch:Radiation}.  The null-bicharacteristic flow of
$\Box_g$ is the flow of the Hamilton
vector field $\Ham = (\p_\zeta p) \p _z - (\p_z  p) \p_\zeta$
restricted to $p = 0$, where
$z$ are the coordinates on $\mathbb{R}^{n}$, $\zeta$ is dual
to $z$, and $p$ is the principal symbol of $\Box_g$, which is in fact
just the dual metric function $g^{-1} \colon T^* M \lra \mathbb{R}$,
$g^{-1}(z, \zeta) = | \zeta|^2_{g^{-1}(z)}$.  (Thus on Minkowski space $p(\zeta)
= \zeta_n^2 - \zeta_1^2 - \dots - \zeta_{n-1}^2$ and the
null-bicharacteristic flow is a straight line flow in phase space
on the space of null vectors, $\zeta$ with $p(\zeta) = 0$, keeping
$\zeta$ fixed and evolving $z$ affinely.)
The b-principal symbol of $L$, $\sigma_b(L) = \lambda$, can be understood as the
extension of the principal symbol of $L$ in the standard sense to the
b-cotangent bundle $\Tb^*M$, thus as discussed
$\lambda$ is equal to the dual metric function of $\hat g =
\rho^2g$, which extends smoothly to all $\Tb^*M$.
Concretely, in the variables
$\rho, v, y$, writing forms as 
$$
\sigma \frac{d\rho}{\rho} + \gamma\, dv + \eta\, dy,
$$ 
$\lambda$ satisfies
\begin{equation}\begin{aligned}\label{eq:b-princ-symbol}
  \lambda=\sigma_{\bl} (L) &= g^{\rho\rho}\sigma^{2} - (4v - \beta v^{2} +
  O(\rho v) + O(\rho^{2})) \gamma^{2} - 2(2-\alpha v+ O(\rho))\sigma
  \gamma \\
  &\quad + 2g^{\rho y}\cdot \eta \sigma + \big(2 v \Upsilon + O(\rho)
  \big) \cdot \eta \gamma + g^{y_{i}y_{j}}\eta_{i}\eta_{j},
\end{aligned}\end{equation}
where all the $O(.)$ terms are smooth, and $\beta,\Upsilon$ are smooth
as well as are $g^{\rho\rho}$, etc., which are the dual metric
components (in the b-basis, for $\hat g$, or equivalently in the
sc-basis for $g$); see
\cite[Equation~(3.18)]{Baskin-Vasy-Wunsch:Radiation}. (Thus, this
formula specifies the structure of certain dual metric components for
$\hat g$,
such as $\pa_v^2$ and $\pa_v(\rho\pa_\rho)$.)

Consider the
b-Hamilton vector field, i.e.\ the Hamilton vector field of $L$
thought of as a vector field on $\Tb^* M$. (More precisely, this is
the (unique) smooth extension of the Hamilton vector field of $L$ from
a vector field on $T^*M^{int}$ to $\Tb^*M$.) 
Almost as in \cite{Baskin-Vasy-Wunsch:Radiation} (which used $\xi$ in
place of $\sigma$),
this is given in the variables $(\rho, v, y, \sigma, \gamma, \eta)$ by
\begin{equation}
  \begin{aligned}\label{eq:hamilton}
    \Hamb := (\p_\sigma \lambda) \rho \p_\rho + (\p_{\gamma} \lambda)
    \p_v + (\p_\eta \lambda) \p_y - (\rho \p_\rho\lambda) \p_\sigma -
    (\p_v \lambda)\p_{\gamma} - (\p_y \lambda) \p_\eta
  \end{aligned}
\end{equation}
where $\lambda$ is the dual metric function of $\hat g$ on
$\Tb^* M$; see \cite[Equation~(3.20)]{Baskin-Vasy-Wunsch:Radiation}. Note that $\Hamb$ is automatically a vector field tangent to $\Tb^*_{\pa
  M}M$, i.e.\ a b-vector field on $\Tb^*M$ (thus a section of
$\Tb\Tb^*M$).
Since taking the Hamilton vector field is a derivation on 
functions on the cotangent bundle, the flow of $\Hamb$ (this being the
Hamilton vector field of $\hat g=\rho^2 g$) restricted to 
the set $\Sigma = \{ \lambda = 0 \} $ over $M^{int}$ (where $\rho>0$) is a rescaling of the Hamilton flow of $\Box_g$ restricted to 
$\Sigma$; the rescaling becomes singular at $\rho=0$. In our case,
\begin{equation}\label{eq:bHamvf1}\begin{aligned}
\Hamb=&\big(2g^{\rho\rho}\sigma+2g^{\rho y}\eta-2\gamma(2-\alpha
v+O(\rho))\big)(\rho\pa_\rho)\\
&-2\big((4v-\beta v^2+O(\rho
v)+O(\rho^2))\gamma\\
&\qquad\qquad\qquad+(2-\alpha
v+O(\rho))\sigma+(v\Upsilon+O(\rho))\eta\big)\pa_v\\
&+2\big(g^{\rho y}\sigma+(v\Upsilon+O(\rho))\gamma+g^{y_iy_j}\eta_j\big)\pa_y\\
&-(\rho\pa_\rho \lambda)\pa_\sigma-(\pa_v \lambda) \pa_\gamma-(\pa_y \lambda) \pa_\eta;
\end{aligned}\end{equation}
see \cite[Equation~(3.21)]{Baskin-Vasy-Wunsch:Radiation}.


To describe how the null-bicharacteristic flow acts at infinity we
must define and analyze the b-conormal bundle of the submanifold $S$.
There is a
natural map of the b-tangent space $\Tb M \lra TM$
defined on sections, i.e.\ elements of $\Vb$, by considering a b-vector
field as a standard vector field.  (Thus the map is not
surjective over the boundary; $\rho \p_\rho$ vanishes there.)  We can use the dual map $T^* M \lra
\Tb^* M$ to define the b-conormal bundle of submanifolds; specifically,
for our submanifold $S$, the conormal bundle $\Nb^*S$ equal to the
image in $\Tb^* M$ of covectors in $T^*M$ annihilating the image of $T S \subset T_S
M$ in $\Tb M$.  It turns out that the null-bicharacteristics of $L$
(see Figure \ref{choices}) terminate both at $S_+$ and $S_-$ at the
spherical b-conormal bundle 
$$
\Sb N^* S_{\pm}  = (\Nb^* S_{\pm}  \setminus o) / \mathbb{R}_+.
$$
Before we describe this in more detail, we point out that $\Nb^* S$ in
fact has one dimensional fibers, since in coordinates $\rho, v, y$
with $\rho, v$ (so $S = \{ \rho = 0 = v \}$) as above and $y$ local
coordinates on $S$, so vectors in $TS$ are multiples of $\p_{y}$, are annihilated by
forms $a \,dv  + b \,d\rho$ in $T^*M$, which map to forms $a \,dv + b \rho
(\rho^{-1} \,d\rho)$ in $\Tb^* M$ and thus restrict to $a\,dv $ since $S$ lies in the boundary $\rho = 0$.
More concretely, the b-conormal bundle of $S$ is generated by $dv$,
i.e.\ in the coordinates above is given by the vanishing of
$\rho,v,\sigma,\eta$, and thus $y,\gamma$ are (local) coordinates along it.
This means that at each point $p \in S$,
\begin{equation}\label{eq:spherical-conormal-bundle}
\Sb N^*_p S  = (\Nb^*_p S \setminus \{ 0 \} )/ \mathbb{R}_+  = \{
\gamma \, dv
: \gamma
\neq 0 \} / \mathbb{R}_+, 
\end{equation}
so in fact $\Nb^* S$ is a line bundle over $S$ and $\Sb N^* S$ is an
$\mathbb{S}^0$ (two points) bundle generated by the images of $dv$ and
$-dv$ in $\Tb^*_S M$.

The flow on null-bicharacteristics, in view of the structure of the
operator at $S_\pm$, as shown in \cite[Section~3]{Baskin-Vasy-Wunsch:Radiation},
see also \cite[Section 5]{HVsemi}, makes the two halves of
the spherical b-conormal bundle of $S$, $\SNb^*
S=\SNb^*_+S\cup\SNb^*_-S$, into a family of
sources ($-$) or sinks ($+$) for the Hamilton flow,  meaning that the null-bicharacteristics approach
$\SNb_+^*S_+$ as their parameter goes to $+\infty$ and $\SNb^*_-S_-$
as the parameter goes to $-\infty$, or $\SNb_+^*S_-$ as their parameter goes to $+\infty$ and $\SNb^*_-S_+$
as the parameter goes to $-\infty$. Correspondingly, the
characteristic set $\Sigma\subset\Tb^*M\setminus o$, which we also
identify as a subset of $\Sb^*M$, of $L$ globally splits into the
disjoint union $\Sigma_+\cup\Sigma_-$, with the first class of
bicharacteristics contained in $\Sigma_+$, the second in $\Sigma_-$.
One computes, cf.\ the discussion after \cite[Equation~(3.22)]{Baskin-Vasy-Wunsch:Radiation},
that the b-Hamilton vector field in \eqref{eq:bHamvf1}  at
$\Nb^*S$, modulo terms vanishing there quadratically, is
satisfies
$$
- 4 \gamma (\rho \p_\rho) - (8v \gamma + 4\sigma) \p_v +2\mu_i\p_{y_i} + 4 \gamma^2 \p_{\gamma}.
$$
with $\mu_i$ vanishing on $\Nb^*S$. Thus, b-Hamilton vector field is
indeed radial at $\Nb^* S$ (it is a multiple of $\pa_\gamma$ as
$\sigma,\rho,v,\mu_i$ vanish there by \eqref{eq:spherical-conormal-bundle});
furthermore within $\Nb^* S$ one sees that for $\gamma>0$ fiber
infinity is a sink (the flow tends towards it), while for $\gamma<0$ a source.
One also sees that in fact $\Sb^*NS$ is a source/sink bundle depending
on the sign of $\gamma$, i.e.\ in the normal directions to $\Nb^*S$ the flow
behaves the same way as within $\Nb^*S$, by checking the eigenvalues
of the linearization, see \cite[Equation~(3.23)]{Baskin-Vasy-Wunsch:Radiation}.

 \begin{figure}\label{spacetime}
    \centering
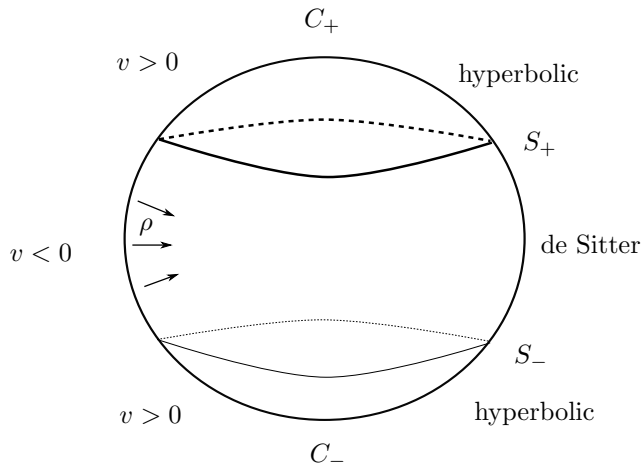
    \caption{$\rho$ equals zero exactly on the boundary and has
      non-vanishing differential there.}
      \end{figure}

Recall that the basic result for elliptic problems on compact
manifolds without boundary is elliptic
regularity estimates, which in turn imply Fredholm properties.
Indeed, if $P$ is an elliptic operator of order $k$ on a compact manifold without
boundary $X$, then for any $m'<m+k$ one has the estimate
\begin{equation}
  \label{eq:ellipticestimates}
  \norm[H^{m + k}(X)]{u} \le C ( \norm[H^{m}(X)]{P u} + \norm[H^{m'}(X)]{u} ).
\end{equation}
That $P$ is a Fredholm map from $H^{m + k}(X)$ to $H^{m}(X)$ is an
immediate consequence of this estimate and the fact that
$H^{m + k}(X)$ is a compact subspace of $H^{m'}(X)$, together with the
fact that $P^*$, the formal adjoint of $P$, is then also elliptic, so
analogous estimates hold for $P^*$.

Here we
have real principal type points over $M^\circ$ as $\Box_g$ is
non-elliptic, as well as radial points at $\SNb^*_\pm
S_\pm$.  Recall that real principal type estimates simply propagate
regularity along null-bicharacteristics, i.e.\ given that the estimate
holds at a point, one gets it elsewhere as well. The basic result at
radial points which are sources or sinks, see
\cite[Proposition~4.4]{Baskin-Vasy-Wunsch:Radiation}, \cite[Proposition~5.1]{HVsemi} and indeed
\cite{Haber-Vasy:Radial} for a precursor in the boundaryless setting
(in turn based on \cite{VD2013}, which further goes
back to \cite{RBMSpec}), in terms of b-Sobolev spaces, which we
proceed to describe in detail, 
is that subject to restrictions on the decay and regularity orders, in
the high regularity regime, one has a real principal type estimate but
without an assumption that one has the regularity anywhere, provided
one has at least a minimum amount of a priori regularity at the point in
question. To clarify, \textit{away} from radial points, real principal
type propagation estimates control the norm in $H^s$ (microlocally)
near a point on a null bicharacteristic in terms of the $H^{s -1}$
norm of $Lu$ and the $H^{s}$ norm of $u$ \textit{elsewhere on the
  bicharacteristic} (together with arbitrarily low regularity norms of
$u$); for high regularity radial points estimates, the
``elsewhere...'' can be removed, as one gets an estimate for the $H^s$
norm near the radial point in terms of the $H^{s-1}$ norm of $Lu$
(provided one knows the wavefront set does not intersect the radial
set at the point in question.)
 On the other hand, in the low regularity setting, one can
propagate estimates into the radial points, much as in the case of
real principal type estimates.  See Theorem \ref{thm:propofsing}. 

To describe this concretely, we must first say what we mean precisely
by regularity and vanishing order.  For any manifold with boundary $M$, fix a non-vanishing b-density $\mu$, i.e. a
non-vanishing smooth section of the density bundle of $\Tb M$, which
necessarily takes the form $\rho^{-1}\wt{\mu}$ for a non-vanishing
density $\wt{\mu}$ on the manifold with boundary $M$ (so for $M =
\overline{\mathbb{R}^n}$, the radial compactification of
$\mathbb{R}^n$, notation as in the paragraph containing
\eqref{eq:Euclidean-radial-scattering-form}, $\mu = a \rho^{-1}
|\,d\rho\, dw|$ where $|dw| = |\prod_{i = 1}^{n-1} dw_i|$ and $a$ is a
smooth non-vanishing function up to $\rho = 0$.) The density $\mu$ is
natural here as it can be taken to be the absolute value of the volume
form of a b-metric, e.g.\ on $\mathbb{R}^n$, $|\rho^2 \,dz^2|$ is such a
b-metric, and
we define
$L^2_b$ to be the Hilbert space induced by $\mu$, so
\begin{equation}
  \label{eq:l2bdef}
  \la u, v \ra_{L^{2}_{b}} = \int_{M} u \,\overline{v} \, \mu.
\end{equation}
We define the weighted b-Sobolev spaces, first for integer orders $k \in
\mathbb{N}$ by letting $u \in \Hb^{k}(M)$ if and only if $V^1 \dots
V^{k'} u \in L^2_b$ for every $k'-$tuple of b-vector fields $V_i \in
\mathcal{V}_b$ with $k' \le k$.  (Recall that $\Vb$ is the space of
b-vector fields discussed in \eqref{eq:b-vector-fields}; thus to $u
\in \Hb^k(M)$ one can apply in particular $k'$-fold derivatives of the
form $\rho \p_\rho$ and $\p_{w_i}$ and remain in $L^2_b$.)  For $m\geq
0$ real we have
\begin{equation}\label{eq:bsobolevdef-pos}
  \begin{split}
    \Hb^{m}(M) &= \{u \in C^{-\infty}(M) \ | \ A u \in \Lb^2(M) \
    \forall A \in \Psib^{m}(M) \},\\
    \Hb^{m,l}(M) &= \rho^{l}\Hb^{m}(M)
  \end{split}
\end{equation}
where $\Psib^{m}(M)=\Psib^{m,0}(M)$ is the space of b-pseudodifferential operators,
described in Section \ref{sec:fred}. For $m<0$ this can be extended by
duality, or instead for all real $m$, demanding that $u$ lie in $\Diffb^N(M)\Lb^2(M)$
for some $N$ (i.e.\ be a finite sum of at most $N$th b-derivatives of
elements of $\Lb^2(M)$) and satisfy the regularity under application of
b-ps.d.o's:
\begin{equation}\label{eq:bsobolevdef}
  \begin{split}
    \Hb^{m}(M) &= \{u \in \Diffb^*(M)\Lb^2(M) \ | \ A u \in \Lb^2(M) \
    \forall A \in \Psib^{m}(M) \},\\
    \Hb^{m,l}(M) &= \rho^{l}\Hb^{m}(M)
  \end{split}
\end{equation}
In general, we will allow a variable $m \in
C^\infty(\Sb^*M;\RR)$, in which case the same definition can be applied,
namely one simply takes variable order ps.d.o's (see
\cite[Appendix]{Baskin-Vasy-Wunsch:Radiation}). (Such variable order
spaces have a long history, starting with Unterberger
\cite{Unterberger:Resolution} and Duistermaat
\cite{Duistermaat:Carleman}; see the work of Faure and Sj\"ostrand \cite{Faure-Sjostrand:Upper} and
Dyatlov and Zworski \cite{Dyatlov-Zworski:Dynamical} for other recent applications.) Taking $M =
\overline{\mathbb{R}^n}$, note that we
may choose the measure $\mu$ in \eqref{eq:l2bdef} so that $\Hb^{0,
  n/2} = L^2$ where $L^2$ here and below denotes the standard Hilbert
space on $\RR^n$, indeed we can take $\mu = \rho^{-n} |dz|$ there; we
remark that the equality as Banach spaces up to equivalence of norms
(which is what matters mostly) is automatic.  Note that the $L^2_b$ pairing gives an isomorphism
\begin{equation}
  \label{eq:l2bduality}
  (\Hb^{m, l})^* \simeq \Hb^{-m, -l}.
\end{equation}

For $s \in
\mathbb{R}$, the weighted b-Sobolev wavefront sets of a
distribution $u$, denoted $\WFb^{s, l}(u)$ are the directions in phase
space in which $u$ fails to be in $\Hb^{s,l}(M)$.  A concrete
definition using explicit b-pseudodifferential operators is given in
\eqref{eq:goodwavefrontdef} below, but for the moment we state that it is defined for $u \in
\Hb^{-N, l}$ by
\begin{equation}
  \label{eq:wfdef}
  \WFb^{s, l}(u) = \bigcap \set{\Sigma(A) \subset \Tb M : A u \in \Hb^{s, l}(M)},
\end{equation}
where the intersection is taken over all $A \in \Psib^{0, 0}(M)$, i.e.\ $A$ is a $(0, 0)$ order
b-pseudodifferential operator (again, see Section \ref{sec:fred}) and
$\Sigma(A)$ is the characteristic
set (vanishing set of the principal symbol) of $A$. Equivalently, a
point $(p,\xi)\notin\WFb^{s,l}(u)$ (where $\xi \in \Tb^*_p M\setminus o$) if there exists
$A\in\Psib^{0,0}(M)$ which is elliptic at $(p,\xi)$ such that $Au\in \Hb^{m,l}(M)$.
We say that $u$
is in $\Hb^{s, l}$ microlocally if $(p, \xi) \not \in \WFb^{s, l}(u)$
where $\xi \in \Tb^*_p M$.
There is a completely analogous definition of $\WFb^{m,l}$ for varying $m \in
C^\infty(^bS^*M)$ and for $ l \in \mathbb{R}$.

We have the following result, which is essentially \cite[Proposition~5.1]{HVsemi}. For the
following statement, let $\mathcal R$ be any of the above discussed
connected components of radial sets $\SNb_\pm^*S_\pm$.
\begin{prop}\label{thm:propofsing}
Let $(M, g)$ be a Lorentzian scattering space as in
\eqref{eq:metric}.  Let $L$ be as above and $u \in \Hb^{-\infty, l}(M)$.

If $m + l < \frac{1}{2}$ and $m$ is nonincreasing along the Hamilton
flow in the direction that approaches $\mathcal R$, then $\mathcal R$
is disjoint from $\WFb^{m, l}(u)$ provided that $\mathcal R \cap
\WFb^{m-1, l}(Lu) = \varnothing$ and a punctured neighborhood in
$\Sigma \cap \Sb^*M$ of $\mathcal R$ (i.e.\ a neighborhood of
$\mathcal{R}$ with $\mathcal R$ removed) is
disjoint from $\WFb^{m, l}(u)$.

On the other hand, suppose that $m' + l > \frac{1}{2}, m \geq m'$ and $m$ is nonincreasing along the Hamilton flow in the direction that leaves $\mathcal R$. Then if $\WFb^{m', l}(u)$ and $\WFb^{m - 1, l}(Lu)$ are both disjoint from $\mathcal R$, then $\WFb^{m, l}(u)$ is disjoint from $\mathcal R$.
\end{prop}

For elliptic regularity, the
variable order $m$ is completely arbitrary, but for real principal type
estimates it has to be non-increasing in the direction along the
Hamilton flow in which we wish to propagate the estimates.

So now fixing $l$ and taking $m$ satisfying $m+l>1/2$ at exactly one
of $\SNb^*_+S_+$ or $\SNb^*_-S_-$ and $m + l < 1/2$ at the other
(e.g. if $> 1/2$ at $\SNb^*_+S_+$ then $< 1/2$ at $\SNb^*_-S_-$), and
similarly $m+l>1/2$ at either $\SNb^*_-S_+$ or $\SNb^*_+S_-$ and $m +
l < 1/2$ at the other,
we
obtain estimates for $u$ in $\Hb^{m,l}$ in terms of the  $\Hb^{m-1,l}$
norm of $Lu$ plus the $\Hb^{m',l}$ norm of $u$ for
$m'<m$ (a weaker norm).  To make this precise we work with varying order Sobolev
spaces $\Hb^{m,l}(M)$.  These are discussed in detail in
\cite[Appendix]{Baskin-Vasy-Wunsch:Radiation} in the setting of
standard Sobolev spaces (i.e.\ without the ``b''), but since the
development is nearly identical we discuss them only briefly.
Specifically, given a function $m \in C^{\infty}(\Sb^{*}M)$ that is
monotonic along the Hamilton flow, 
$u \in \Hb^{m, l}(M)$ if and only if $A u \in L^{2}_\bl(M)$ for any $A
\in \Psib^{m,l}=\rho^l\Psib^m$, where for $l=0$ membership of
$\Psib^{m,l}$ means that $A$ is the
quantization of a symbol $a \in C^{\infty}(\Tb^{*}M)$ satisfying
(among other standard symbol conditions elaborated in
\cite[Appendix]{Baskin-Vasy-Wunsch:Radiation}) that
$\absv{a(\rho, w, \sigma, \omega) } \le C (1 + \sigma^{2} +
\absv{\omega}^{2})^{m/2}$; here $\rho$ is again a boundary defining
function, and coordinates on $\Tb^{*}M$ are obtained by parametrizing
b-covectors as
\begin{equation*}
  \sigma \frac{d \rho}{\rho} + \omega^{i} dw_{i}.
\end{equation*}
For $l \in \mathbb{R}$, we thus have $\Hb^{m,l}(M) := \rho^{l}\Hb^{m,
  0}(M)$; the norm on these spaces is given by any elliptic
$A\in\Psib^{m,l}$ together with the $\Hb^{m',l}$ norm, where $m'<\inf
m$. (This is only defined up to equivalence of norms, but
that is all we need.) Thus, given any $s, r$ with $s$ monotone along the Hamilton flow and
$r \in \mathbb{R}$, consider the spaces
\begin{equation}
\label{eq:spaces}
\cY^{s,r}=\Hb^{s,r}(M),\ \cX^{s,r}=\{u\in\Hb^{s,r}(M):\ Lu\in\Hb^{s-1,r}(M)\}.
\end{equation}
Then $\cX^{s,r}$ is a Hilbert space with norm
$$
  \norm[\cX^{s,r}]{u} =  \norm[\Hb^{s, r} (X)]{u} + \norm[\Hb^{s
    - 1, r} (X)]{Lu} .
$$
With $m, l$ and $m'$ as above (in particular $m$ is a function),
we have the estimates
\begin{equation}
  \label{eq:propagationestimates}
  \norm[\Hb^{m, l}(X)]{u} \le C ( \norm[\Hb^{m-1, l} (X)]{L u} + \norm[\Hb^{m', l} (X)]{u} ).
\end{equation}
(Here $m' < m$ can be taken to be a function, but this is not important.
It can, for instance, be taken to be an integer $N < \inf m$.)  

Note that the `end' of the bicharacteristics at which
$m+l<1/2$ is the direction in which the estimates are propagated, thus
the choices
\begin{equation}\label{eq:choices}
  \begin{split}
    \pm(m+l-1/2)&<0 \mbox{ at } \SNb^*_+S_+\ \text{(sinks)},\ \mbox{ and }\\
    \pm(m+l-1/2)&<0 \mbox{ at }   \SNb^*_-S_+\ \text{(sources)}
  \end{split}
\end{equation}
determine what (if any) type of inverse we get for $L$; we denote $L$ on
the corresponding spaces by $L_{\pm\pm}$ with the two $\pm$
corresponding to the two $\pm$ as in \eqref{eq:choices}, i.e.\ the first to the
direction of propagation in $\Sigma_+$, the second to that in
$\Sigma_-$, with the signs being positive if the propagation is
towards $S_+$ and negative if the propagation is towards $S_-$. Notice
that by our requirements of $m+l>1/2$ at exactly one end of each
bicharacteristic and $m+l<1/2$ at the other, \eqref{eq:choices} is
{\em equivalent} to
\begin{equation}\label{eq:choices-S-}
  \begin{split}
    \mp(m+l-1/2)&<0 \mbox{ at } \SNb^*_-S_-\ \text{(sources)}\ \mbox{ and }\\
    \mp(m+l-1/2)&<0 \mbox{ at }   \SNb^*_+S_-\ \text{(sinks)}.
  \end{split}
\end{equation}
That is to say, 
\begin{equation*}
L_{\pm \pm} \mbox{ denotes \textit{any} map }  L \colon \cX^{m,l} \lra \cY^{m-1,l}
\end{equation*}
for which the pair $(m, l)$ satisfy \eqref{eq:choices} with the given 
$\pm, \pm$ combination (the first sign in the first inequality and
the second in the second).  (To be clear, the fact that we write the
choices in \eqref{eq:choices} as taking place at $S_+$ is arbitrary,
as we could just as easily make the choices at $S_-$.  Whichever signs
are chosen at $S_+$, the opposite sign is chosen on the other end of
the flow at $S_-$.  For
example, for $L_{+-}$, the condition on $m$ at $S_-$ is that $
-(m+l-1/2)<0$ at $\SNb^*_-S_-$ and $m+l-1/2 <0$ at $\SNb^*_+S_-$.)
Strictly
speaking, $L_{\pm \pm}$ depends on $m$, but in fact we will see that
the choice of $m$ satisfying a particular version of
\eqref{eq:choices} is irrelevant.  Thus we use the notation
\begin{equation}
  \label{eq:spacenotation}
  \cX^{m, l}_{\pm\pm} = \cX^{m,l},\ \cY^{m,l}_{\pm\pm}=\cY^{m,l} \mbox{ for \emph{any} $(m, l)$ satisfying \eqref{eq:choices} }
\end{equation}
with the given $\pm$, $\pm$ combination.  See Figure~\ref{choices}.

 \begin{figure}\label{choices}
    \centering
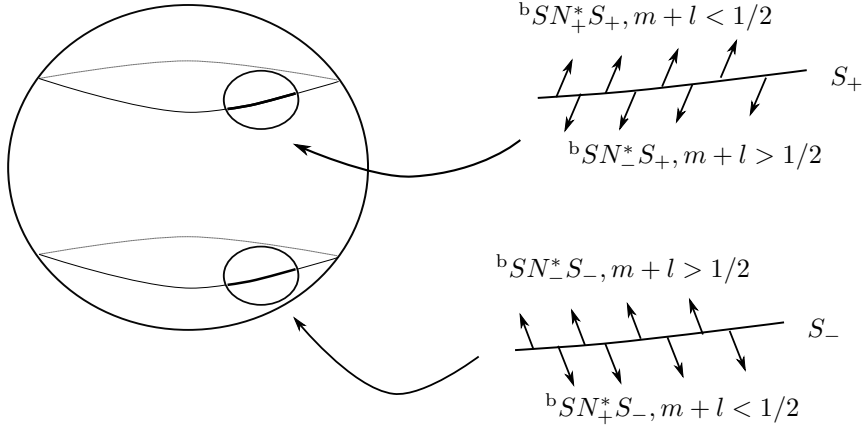
    \caption{For the operator $L_{+-}$, corresponding to the forward Feynman
      problem, high regularity is imposed at the `beginning' (near
      $\SNb^*_- S$) of each
      null bicharacteristic, whether they begin at $S_{+}$ or $S_{-}$.}
      \end{figure}
We call $L_{++}$ the forward wave operator
(corresponding to the forward solution), $L_{--}$ the backward wave
operator, $L_{+-}$ and $L_{-+}$ the Feynman and anti-Feynman wave
operators (sometimes we simply call them both Feynman), with
$L_{+-}$ propagating forward along the Hamilton flow (i.e.\ to the sinks), and $L_{-+}$
backward along the Hamilton flow (i.e.\ to the sources) in {\em both} $\Sigma_+$ and $\Sigma_-$. Here we point out that either of the forward
and backward wave operators propagate estimates in the {\em opposite}
directions relative to the Hamilton flow in $\Sigma_+$, resp.\
$\Sigma_-$; the propagation is in the same direction relative to a
time function in the underlying space $M$.

 \begin{figure}\label{smooth}
    \centering
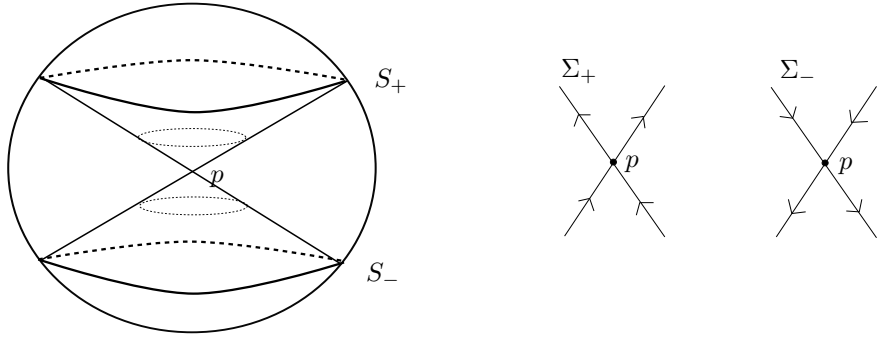
    \caption{The traces (i.e.\ projections from the cotangent space) of
      the light rays passing through an arbitrary
      point $p$.  In the cotangent space these separate into the
      forward and backward pointing null-bicharacteristics, depicted
      heuristically at right.  The operator $L_{+-}$ corresponds to propagation of
      singularities along the flow, and corresponds to the choice of
      $+$ in the first and $-$ in the second inequality in \eqref{eq:choices}}
      \end{figure}

\section{Mapping properties of the Feynman propagator}\label{sec:fred}
The main result of this section is Theorem  \ref{thm:fredholm} below,
which asserts that $L_{\pm \pm}$ are Fredholm maps between appropriate
Hilbert spaces.  As mentioned, the estimates in \eqref{eq:propagationestimates} are {\em
  not} sufficient to conclude that $L$ is Fredholm, 
since the weaker norm does not possess additional decay.  Thus the
main technical result of this section is the following.  For
$(m, l)$ chosen as in \eqref{eq:choices} for any choice of signs
$\pm \pm$, and  \textit{for certain choices of} $l$ (see the theorem),
we have
\begin{equation}
  \label{eq:propagationestimatesimproved}
  \begin{split}
    \norm[\Hb^{m, l}]{u} &\le C ( \norm[\Hb^{m-1, l} ]{L u} +
    \norm[\Hb^{m', l'} ]{u} ) \\
    \norm[\Hb^{1 - m, -l}]{v} &\le C ( \norm[\Hb^{-m, -l} ]{L v} +
    \norm[\Hb^{1 -m'', -l''} ]{v} )
  \end{split}
\end{equation}
where $m' < m < m''$ and $l' < l < l''$.  As explained in the proof of
Theorem \ref{thm:fredholm}, it is
then a simple exercise using the fact that $\Hb^{m, l} \subset
\Hb^{m', l'} $ is compact provided $m' < m$ and $l' < l$ to show that
$L_{\pm \pm}$ is Fredholm on the spaces in the theorem.  

To obtain the improved estimates in
\eqref{eq:propagationestimatesimproved}, as in elliptic problems, we
also need to consider the Mellin
transformed normal operator $\hat N(L)(\sigma)$ of $L$, which is a
family of differential operators on $\pa M$, parameterized by
$\sigma\in\Cx$.  Given an arbitrary $P \in \Diffb^*(M)$ of order $k$, 
\begin{equation}
  \label{eq:boparbitrary}
  P = \sum_{i + \absv{\alpha} \le k} a_{i, \alpha}(\rho, x) (\rho
  \p_{\rho})^{i} \p_{x}^{\alpha},
\end{equation}
the normal operator is locally given by
\begin{equation}
  \label{eq:normal}
  N(P) := \sum_{i + \absv{\alpha} \le k} a_{i, \alpha}(0, x)
  (\rho \p_{\rho})^{i} \p_{x}^{\alpha} \in \Diffb^{k}([0,
  \infty)_{\rho} \times \p M).
\end{equation}
The Mellin transform is defined, initially on compactly supported smooth
functions $u \in C^{\infty}(\mathbb{R}^{+}; \mathbb{C})$, by
\begin{equation*}
  \cM(u)(\sigma) = \hat u(\sigma) = \int_{0}^{\infty} \rho^{-i\sigma} u(\rho) \frac{d\rho}{\rho}.
\end{equation*}
Note that $\cM u(\sigma) = \mathcal{F} v (\sigma)$ where $\mathcal{F}$
is the Fourier transform and $v(x) = u(e^{x})$.  Writing complex
numbers $\sigma = \xi + i \eta$, it extends to a unitary isomorphism
\begin{equation}\label{eq:mellindef}
  \cM \colon \rho^{l}L^{2}(\mathbb{R}^{+}, d\rho/\rho) \lra L^{2}(\set{\im
    \sigma = - l } , d\xi).
\end{equation}
The inverse map of \eqref{eq:mellindef} is given by
\begin{equation}
  \label{eq:1}
  \cM_{l}^{-1} f(\rho) = \frac{1}{2\pi} \int_{\set{\im \sigma = -l}} \rho^{i \sigma}
  f(\rho) d\sigma.
\end{equation}
Moreover, conjugating $N(P)$ by the Mellin transform in $\rho$ gives
\begin{equation}
  \label{eq:reducednormal}
  \hat{N}(P)(\sigma) = \sum_{i + \absv{\alpha} \le k} a_{i, \alpha}(0, x)
  \sigma^{i} \p_{x}^{\alpha}
\end{equation}

We digress briefly to describe following typical example of a
b-pseudodifferential operator which is elliptic at a point $p \in \Tb
M$ lying over the boundary, and how it relates to the b-wavefront
set discussed above. If $p \in \Tb^*M$ lies over the boundary, then in coordinates $(\rho, y, \xi, \eta)$
on $\Tb M$ where $\rho$ is a boundary defining function, $\xi$ is dual to
$\rho$ and $\eta$ to $y$, we have $p = (0, y_{0}, \xi_{0}, \eta_{0})$.  We
obtain a b-pseudodifferential operator that is elliptic at $p$ by
choosing a cutoff function $\chi(\rho, y)$ with
$\chi(0,y_0) \neq 0$ and such that $\chi$ is supported in $\set{\rho
  < \epsilon}$ for small $\epsilon$, in particular small enough so
that $\set{\rho < \epsilon} \simeq \p M \times [0, \epsilon)$. Let
$\phi(\xi,\eta)$ be a symbol, homogeneous near infinity, non-zero in
the cone given by positive multiples of
$(\xi_0,\eta_0)$. With $\mathcal{F}$ the Fourier transform in the $y$
variables, we define
\begin{equation}
  \label{eq:boperator}
  \begin{split}
  A u &:= \mathcal{M}^{-1}_{0} \mathcal{F}^{-1} \phi \mathcal{F}\mathcal{M} (\chi u).
\end{split}
\end{equation}
Then $A \in \Psib^{0, 0}(M)$, and the b-principal symbol of $A$ at
order and weight $(m, l) = (0, 0)$ is:
\begin{equation}
  \label{eq:bsymbol}
  \sigma_{0, 0}(A) \colon \Tb^* M \lra \mathbb{C}, \quad \sigma_{0,
    0}(A) = \chi \phi
\end{equation}
where we think of $\chi \phi= \chi(\rho, y) \phi(\xi,\eta)$ as a function on $\Tb^* M$, which
near the boundary and with our coordinates is diffeomorphic to $\{
\rho < \epsilon, \xi \} \times T^* \p M$, supported on the
neighborhood of $(0, y_0, \xi_0, \eta_0)$ under consideration.
In fact, such operators can be used to neatly describe the b-wavefront
sets of distributions.  Given a distribution $u \in \Hb^{-N, l}(M)$,
then for $m, l \in \mathbb{R}$,
\begin{equation}\label{eq:goodwavefrontdef}
(0, y_0, \xi_0, \eta_0) \not \in \WFb^{m, l}(u) \iff \exists \chi, \phi
\mbox{ with } A \rho^{- l} u  \in \Hb^{m, 0}(M),
\end{equation}
where $A$ is formed from $\chi$ and $\phi$ as in \eqref{eq:boperator}.
(The $\rho^{-l}$ in the front is there so that the inverse Mellin
transform $\mathcal{M}_0^{-1}$ of the resulting object is well defined.)

The structure and properties of $\hat{N}(L)(\sigma)$ are discussed at length in
\cite{HVsemi}.  To briefly summarize, for each $\sigma$,
$\hat{N}(L)(\sigma)$ is a second order differential operator which is
elliptic in the interior of the regions $C_{\pm}$, and hyperbolic on
their complement $\p M \setminus (C_+ \cup C_-)$ whose
characteristic set splits into two components $\wt{\Sigma}_{\pm}$,
each of which contains a Lagrangian submanifold of radial points lying
over $S = \p C_+ \cup \p C_-$,  and which split the conormal bundle
$N^*S$ (in $\p M$) into four components
$N^*_{\pm} S_{\pm}$ which are sources ($N^*_-S$) and sinks
($N^*_+S$) for the Hamilton flow.

The estimates corresponding to those of the previous section allow one
to conclude that $\hat N(L)(\sigma)$ is Fredholm
for each $\sigma$ on the induced Sobolev spaces, where $\im\sigma=-l$,
i.e.\ provided 
\begin{equation}\label{eq:boundarychoices}
  \begin{split}
    \pm(m- \im \sigma-1/2)&<0  \mbox{ at } N^*_+S_{+} , \mbox{ and }\\
    \pm(m- \im \sigma-1/2)&<0 \mbox{ at }    N^*_-S_{+}.
  \end{split}
\end{equation}
More precisely here $m$ is replaced by $m|_{S^* \pa M}$, which is a
well-defined subbundle of $\Sb^*_{\pa M} M$. (Indeed, the map $\Tb M
\lra TM$, restricts to a surjection $\Tb_{\p M} M
\lra T\p M$, and the dual injection $T^*\p M \lra \Tb_{\p M}^* M$
gives the desired inclusion of $S^* \p M$ after modding out by the
action of $\mathbb{R}_+$).   Thinking of $\sigma$ as
the b-dual variable of $\rho$ (which thus depends on the choice of
$\frac{d\rho}{\rho}$ in a neighborhood of $\p M$ but is invariant at $\pa M$), covectors are of the form
$\beta+\sigma\frac{d\rho}{\rho}$, $\beta\in T^*M$, thus (identifying
functions on $\Sb^*M$ with homogeneous degree $0$ functions on
$\Tb^*M\setminus o$, where $o$ denotes the zero section) for each
$\sigma\neq 0$ one actually has a function on $T^* \pa M$. One thus
obtains a family of large parameter norms (as described in the theorem
just below), analogous to the usual
semiclassical norms: for $\sigma$ in a compact set, the norms are
uniformly equivalent to each other, but as $\sigma\to\infty$ this
ceases to be the case.  In fact, we have the following applications of
\cite[Proposition
5.2]{Baskin-Vasy-Wunsch:Radiation}, \cite[Theorem 2.14]{VD2013}.

\begin{thm}
  \label{thm:indexzeroboundary}(Statement of \cite[Theorem
  2.14]{VD2013} in our setting.)
In strips in which $\mathrm{Im} \,\sigma$ is bounded, $\hat
N(L)(\sigma)^{-1}$ has finitely many poles.  
\end{thm}

\begin{proof}
Since this is the statement of \cite[Theorem
  2.14]{VD2013}, with the underlying analysis being carried out in
  \cite[Section~2]{VD2013}, we only give a brief sketch.

As we will see momentarily, our family $\hat N(L)(\sigma)$ forms an analytic
  Fredholm family
  \begin{equation}
    \label{eq:normalfamily}
    \hat N(L)(\sigma) \colon \cX^{m}(\p M) \lra \cY^{m-1}(\p M),
  \end{equation}
  where $\cX^{m}(\p M) = \{ \phi : \phi \in H^{m}(\p M), \hat
    N(L)(\sigma) \phi \in H^{m-1}(\p M)\}$, and $\cY^{m}(\p M) = \set{
    \phi : \phi \in H^{m}(\p M)}$ provided that $\sigma$ and $m$ are
  related as in \eqref{eq:boundarychoices}, whose inverse is thus
  meromorphic if $\hat N(L)(\sigma)$ is invertible for at least one
  $\sigma=\sigma_0$. For bounded $\im \sigma$, we can see that
  $\hat N(L)(\sigma)$ is invertible for sufficiently large
  $\re \sigma$. This follows exactly as in
  \cite[Proposition 5.2]{Baskin-Vasy-Wunsch:Radiation}, which in turn
  follows directly from \cite[Theorem
  2.14]{VD2013}. The key to this is to consider the semiclassical problem
  obtained by letting $h = |\sigma|^{-1}$ and $z =
  \frac{\sigma}{|\sigma|}$, and, letting $P_\sigma = \hat N
  (L)(\sigma)$, studying
$$P_{h, z} := h^2 P_{h^{-1} z} \in \Psi^2_h(\partial M),$$
where $\Psi^{2}_{h}(\partial M)$ denotes the space of semiclassical
pseudodifferential operators of order $2$ on $\partial M$.
This semiclassical family on $\partial M$ has Lagrangian submanifolds of radial
points (coming from the b-radial points of $L$), and, as described in
\cite[Section~2.8]{VD2013}, the standard positive commutator
proof of propagation of singularities around Lagrangian submanifolds
of radial points carries over to the semiclassical regime without
difficulty. This allows us to obtain estimates
\begin{equation*}
  \begin{split}
    \| u \|_{H^m_h} &\leq C (h^{-1} \|P_{h, z} u \|_{H^{m - 1}_h} +
    h\|u\|_{H^{-N}_h}) \\
    \| v \|_{H^{1 - m}_h} &\leq C (h^{-1} \|P^*_{h, z} v \|_{H^{-m}_h} +
    h\|u\|_{H^{-N}_h})
  \end{split}
\end{equation*}
for arbitrarily large $N$, within strips of
bounded $\im \sigma$. 
As described at the beginning of the proof of Theorem
\ref{thm:fredholm} below, these estimates imply that $N(L)(\sigma)$
mapping in \eqref{eq:normalfamily} is Fredholm.  
Hence, for sufficiently small $h$, the
$-N$ norm can be absorbed into the left hand side, giving
by the first inequality injectivity and by the second surjectivity.
(This point is also elaborated in Theorem \ref{thm:fredholm}.)
Note that the statement of \cite[Proposition
5.2]{Baskin-Vasy-Wunsch:Radiation} is for only the
forward and backward propagators, as the results come from microlocal
positive commutator estimates which are sufficiently microlocal, the
conclusion, with the same proof, also holds for the Feynman operators.\end{proof}

\begin{rem}
  We point out that analogues of
the estimates used so far go through if $L$ has sufficiently weak
trapping with slight modifications: so-called b-normally hyperbolic
trapping, as introduced in \cite{Hintz-Vasy:Normally}, gives essentially the same estimates for
$\sigma$ real and large. (However, we do not study this here.)
\end{rem}

Following \cite{HVsemi} we will prove the following Fredholm mapping
result for $L$, mapping between spaces $\cX^{m,l}$ and $\cY^{m-1,l}$
which satisfy not only the threshold properties in \eqref{eq:choices},
but furthermore that in the high regularity regime we assume a full
order more Sobolev regularity.
Specifically, we will assume that $(m, l)$ are chosen as in \eqref{eq:choices} for any
choices $\pm, \pm$, with the additional property that when the $-$ sign is valid on the
left hand side, i.e.\ $-(m+l-1/2)<0$, then in fact $-(m+l-3/2)<0$ as
well, and thus the
complete set of options for $m, l$ shall be
\begin{equation}
    \label{eq:full-reg-list}
\begin{tabular}{|c||c|c|c|c|} \hline 
Region & Feynman & Anti-Feynman & Retarded & Advanced \\ \hline\hline
$\SNb^*_+S_+$ & 
	$m + l < 1/2$ & 
	$m + l > 3/2$ & $m + l < 1/2$ & $m + l > 3/2$ \\ \hline
$\SNb^*_-S_+$ & 
	$m + l > 3/2$ &
	$m + l < 1/2$ & $m + l < 1/2$ & $m + l > 3/2$\\ \hline
$\SNb^*_+S_-$ &
	$m + l < 1/2$ &
	$m + l > 3/2$ & $m + l > 3/2$ & $m + l < 1/2$\\ \hline
$\SNb^*_-S_-$ &
	$m + l > 3/2$ & 
	$m + l < 1/2$  & $m + l > 3/2$ & $m + l < 1/2$\\ \hline
\end{tabular}
\end{equation}
\begin{thm}\label{thm:fredholm}
Assume that $(m, l)$ satisfy \eqref{eq:choices}, and in addition the
properties of the previous paragraph, i.e.\ one of the four columns in
\eqref{eq:full-reg-list}.  Moreover, assume that, subject to this choice,
there are no poles of $\hat N(L)(\sigma)^{-1}$ on the line
$\im\sigma=-l$, (where $\hat N(L)$ maps as in \eqref{eq:normalfamily}
with $m = m \rvert_{T^* \p M}$.)  Then $L$ is Fredholm as a map
\begin{equation}\label{eq;fredholmmap}
L:\cX^{m,l}\to\cY^{m-1,l}.
\end{equation}
\end{thm}

\begin{rem}
Note that the Fredholm property is stable under b-perturbations of 
$L$, in $\Psib^{2,0}$, meaning for operators of the form $L + P$ where
$P \in \Psib^{2, 0}$ is small. For $P$ in $\Diffb^2$ this means that
$P$ also has an expression as in \eqref{eq:b-diff-op-form} but with
coefficient functions $a_{j, \alpha}$ which are small in $C^\infty$.  In particular, any perturbation of a Lorentzian 
scattering metric in the sense of sc-metrics gives rise to a similarly 
Fredholm problem.  Indeed, for such a perturbation, the poles of the Feynman resolvent family for the normal operator of the
perturbation of $L$ are themselves perturbed in a continuous fashion from the
poles of those of the normal operator of $L$.
Since obtaining a Fredholm problem at a weight corresponds to having
a line $\{ \im \zeta = c \}$ with no poles, the result follows.
\end{rem}

\begin{rem}\label{rem:hadamard}
Microlocal elliptic regularity states that
$\WFb^{m_0,l}(u)\setminus\Sigma\subset\WFb^{m_0-2,l}(f)$ if $Lu=f$ and
$u\in\Hb^{\tilde m,l}$ for some $\tilde m$ (i.e.\ $u\in\Hb^{-\infty,l}$).
  Propagation of singularities, in the sense of $\WFb$, implies that if
  $L u = f$, where $u\in\Hb^{m,l}$, $f\in\Hb^{m-1,l}$ for some $m$, $l$ satisfying
  \eqref{eq:choices} for the $+-$ signs, then a point
  $\alpha\in\Sigma\setminus(\Sb N_+^*S_+\cup\Sb N_+^*S_-)$ (i.e.\
  $\alpha$ is not at the radial sink, at which the function spaces have
  low regularity) is not in $\WFb^{m_0,l}(u)$ provided that
  the backward bicharacteristic through $\alpha$ is disjoint from
  $\WFb^{m_0-1,l}(f)$ and provided $\WFb^{m_0-1,l}(f)$ is disjoint
  from $\Sb N^*_- S_+\cup\Sb N^*_-S_-$, i.e.\ the radial sources at
  which high regularity is imposed. In particular, if
$f$ is compactly supported in $M^\circ$, then
$\WFb^{m_0,l}(u)\setminus (\Sb N_+^*S_+\cup\Sb N_+^*S_-)$
  contained in the order $m_0-1$ wavefront set of $f$ together with the flowout of
  the intersection of the wavefront set of $f$ with the characteristic
  set $\Sigma$ of $L$. Restricted to the interior, where $\WFb$ is
  just the standard wave front set $\WF(u)\subset S^* M^\circ$, this states that
  \begin{equation}
    \label{eq:hadamard-for-single-space}
    \WF^{m_0}(u)\subset \WF^{m_0-1}(f) \bigcup (\cup_{t \ge 0} \Phi_t(\WF^{m_0-1}(f) \cap \Sigma))
  \end{equation}
where $\Phi_t$ is the time $t$ Hamilton flow (on the cosphere bundle).
In particular this applies to $u=L_{+-}^{-1}f$ when $L_{+-}$ is
actually invertible, so within the characteristic set the wave front set of $u$ is a subset of the
forward flowout of that of $f$.

There are analogous conclusions for the other choices of signs in
\eqref{eq:choices} with the wavefront sets of solutions
contained in the direction of the Hamilton flowout
of the wavefront set of $f$ corresponding to the choice of direction on each
component of the characteristic set. In particular, for the $-+$ sign, $\cup_{t \ge 0} \Phi_t(\WF^{m_0-1}(f) \cap
\Sigma)$ is replaced by $\cup_{t \le 0} \Phi_t(\WF^{m_0-1}(f) \cap
\Sigma)$.

Further, it is not hard to show that, provided $L_{\pm\pm}^{-1}$ exists, the
Schwartz kernel $K_{\pm\pm}$ of $L_{\pm\pm}^{-1}$ satisfies a corresponding wave
front set conclusion in $M^\circ\times M^\circ$. For instance, for $L_{+-}^{-1}$,
$\WF(K_{+-})\setminus N^*\mathrm{diag}$ is contained in the
forward flowout of $N^*\mathrm{diag}$, the conormal
bundle of the diagonal, with respect to the Hamilton vector field in the left factor.
\end{rem}

\begin{proof}
We wish to obtain the improvements to
  \eqref{eq:propagationestimates} in the
  estimates in
  \eqref{eq:propagationestimatesimproved}.  These estimates imply that
  the map in \eqref{eq;fredholmmap} is Fredholm.  Indeed, using the
  fact that the containment $\Hb^{m, l} \subset
\Hb^{m', l'} $ is compact provided $m' < m$ and $l' < l$, the first
  estimate in \eqref{eq:propagationestimatesimproved} shows
that the map has closed range and finite dimensional kernel.  Assuming
that $v$ lies in $(\mbox{image}(L:\cX^{m,l}\to\cY^{m-1,l}))^\perp,$
where the orthogonal complement is taken with respect to the $L^2_b$
pairing (see \eqref{eq:l2bdef}) between $\Hb^{1 - m, -l}$ and $\cY^{m - 1, l} = \Hb^{m - 1,
  l}$, it follows that $Lv = 0$ and thus the second estimate in
\eqref{eq:propagationestimatesimproved} shows that the space of such
$v$ is finite dimensional.

Thus we need only obtain the improved estimate in
\eqref{eq:propagationestimatesimproved}.    The proof is essentially the proof of \cite[Proposition
  2.3]{HVsemi}, and we recall it briefly for the convenience of the
  reader.  The condition on
  $\hat N(L)(\sigma)^{-1}$ on the line $\im(\sigma) = -l$ implies by
  taking the inverse Mellin transform that the map
  \begin{equation}
    \label{eq:15}
    N(L) \colon \cX^{m',l}(\p M \times \mathbb{R}_+) \lra \cY^{m' -1,
      l}(\p M \times \mathbb{R}_+)
  \end{equation}
is bounded and invertible, \textit{where $m' \colon T^* \p M \lra
  \mathbb{R}$ is any function satisfying the constraints in
  \eqref{eq:choices} that $m$ satisfies}.  Thus there is a $C$ such that 
\begin{equation*}
  \| u \|_{\Hb^{m', l}(\p M \times \mathbb{R}_+)} \le C   \| N(L) u
  \|_{\Hb^{m' - 1, l}(\p M \times \mathbb{R}_+)},
\end{equation*}
and we may furthermore choose $m'$ so that it satisfies the constraint
\textit{and} that $m' < m$.  Choosing a cutoff function $\chi$ that is supported near $\p M$ and
equal to $1$ in a neighborhood thereof, we have (with a constant whose
value changes from line to line)
\begin{equation*}
  \begin{split}
    \| u \|_{\Hb^{m, l}(M)} &\le C (\| L u
    \|_{\Hb^{m - 1, l}(M)} + \| u \|_{\Hb^{m', l}(M)}) \\
 &\le C (\| L u
    \|_{\Hb^{m - 1, l}(M)} + \| \chi u \|_{\Hb^{m', l}(M)} + \| (1 -
    \chi) u \|_{\Hb^{m', l}(M)}) \\
 &\le C (\| L u
    \|_{\Hb^{m - 1, l}(M)} + \| N(L) \chi u \|_{\Hb^{m' - 1, l}(M)} + \| u
    \|_{\Hb^{m', l'}(M)}).
  \end{split}
\end{equation*}
Now, writing $N(L) \chi = [N(L), \chi] + \chi(N(L) - L) + \chi L$, and
using $N(L) - L = \rho P$ where $P \in \Diffb^2(M)$, and
$[N(L), \chi] = \rho P'$ where $P' \in \Diffb^1(M)$, note that
\begin{equation*}
      \| u \|_{\Hb^{m, l}(M)}  \le C (\| L u
    \|_{\Hb^{m - 1, l}(M)} + \|  u \|_{\Hb^{m' + 1, l'}(M)}),
\end{equation*}
so to obtain the improved estimate in
\eqref{eq:propagationestimatesimproved} we need only make sure that
$m' + 1 < m$ which can be done due to the $-(m+l-3/2)<0$ assumption at
appropriate radial sets.  
\end{proof}

It is important to remark here that $L_{\pm\pm}$ are rather different
operators for different choices of $\pm\pm$, on the other hand the
choice of $m,l$ {\em satisfying the constraints} corresponding to a
given $\pm$ (i.e.\ a given one of the two constraints) matter much
less.  For example, the invertibility of the normal operator
$\hat N(L)(\sigma)$ is {\em independent} of these additional choices, 
so long as the $m$ satisfies that $m-\im\sigma-1/2$ has the correct
sign at the
relevant locations and has the correct monotonicity.  In the Feynman
case see Proposition
\ref{thm:boundaryscaling} below; the regularity theory shows that the
potential kernel of the operator,
as well as of the adjoint, is indeed independent of these choices. The
choice of $l$ does affect the index of $L$, however as a Fredholm
operator, as we show for the Feynman operator in Theorem
\ref{thm:constantindex}.


We also note that the adjoint of $L_{++}$ is $L_{--}$, while
that of $L_{+-}$ is $L_{-+}$, so one should {\em not} think of $L$ as
a self-adjoint operator even though it is of course formally self-adjoint.

The standard setting in which $\Box_g$ is considered is that of
evolutionary problems, in which the forward or backward propagator
$L_{++}^{-1}$ and $L_{--}^{-1}$ are considered. On the other hand, the
Feynman propagator arises for instance by Wick-rotating suitable Riemannian
problems. Here we are interested in the Feynman propagator, but we
first explain the more studied forward and backward problems in order
to be able to contrast these.

For the forward or backward problems the usual tools of evolutionary
problems, namely standard energy estimates, can be used to compute the
index in some cases, as discussed in \cite[Theorem~5.2]{HVsemi}. Since
in this paper we focus on the Feynman propagator, we shall be brief. Thus, recall first from
\cite[Section~3.2.1]{Baskin-Vasy-Wunsch:Radiation} that the Lorentzian
scattering metric $g$ in fact induces (even) asymptotically hyperbolic
metrics $k_+$, resp.\ $k_-$, on $C_+$, resp.\ $C_-$, with $S_+$, resp.\ $S_-$ being conformal
infinity for these. Similarly, an asymptotically de Sitter Lorentzian
metric is induced on $C_0$, for which $S_+$ is future and $S_-$ is
past conformal infinity. One thus can consider the spectral family
$\Delta_{C_\pm}-\frac{(n-2)^2}{4}-\sigma^2$, as well as its inverse
$\cR_{C_\pm}(\sigma)$ for $\im\sigma\gg 0$, which continues
meromorphically to the complex plane (in $\sigma$).
Then, as discussed in \cite[Section~7]{Baskin-Vasy-Wunsch:Radiation}
and more systematically in \cite{Vasy:Resolvents-AH-dS},
for
the forward problem, the poles of $\hat N(L)(\sigma)^{-1}$ consist of
the poles of the meromorphically continued resolvents
$\cR_{C_+}(\sigma)$ (i.e.\ its resonances) and
$\cR_{C_-}(-\sigma)$ on the asymptotically hyperbolic caps $C_\pm$, as
well as possibly a subset of $i\ZZ\setminus \{0\}$. (The latter correspond to
possible differentiated delta distributional resonant states, which
exist e.g.\ in even dimensional Minkowski space and which are
responsible for the strong Huygens principle on the one hand and for
the absence of poles of the meromorphically continued resolvent on odd dimensional
hyperbolic spaces on the other hand.) Further, the resonant states
and dual states have a certain support structure (this corresponds to
$C_0$ being a hyperbolic region), namely for $\phi$ supported in
$C_0\cup\overline{C_+}$, $\hat N(L)(\sigma)^{-1}\phi$ can only have
poles if $\sigma$ is either a pole of $\cR_{C_+}(\sigma)$ or is in
$-i\Nat^+$, see \cite{Baskin-Vasy-Wunsch:Radiation,
  Vasy:Resolvents-AH-dS}. Thus, see \cite{HVsemi}, suppose that $|l|<1$ (one could take
$l$ larger if one also excludes the possible imaginary integer poles of
$\hat N(L)(\sigma)^{-1}$), and $\cR_{C_\pm}(\sigma)$ have no poles in
$\im\sigma\geq -|l|$, and that
there is a boundary-defining function $\rho$ which is globally
time-like (in the sense that $\frac{d\rho}{\rho}$ is such with respect
to $\hat g$) near $\overline{C_+}\cup\overline{C_-}$. (These
assumptions hold e.g.\ on perturbations of Minkowski space.) Then any element
of $\Ker L$ would be vanishing to infinite order at $\overline{C_-}$
(and the same for $\Ker L^*$, where $L^*$ is the adjoint of $L$ with
respect to the $\Lb^2 = L^2(\RR^n, \mu)$ pairing in \eqref{eq:l2bdef}, with $C_-$ replaced by $C_+$)
by the first hypothesis and vanishing in a neighborhood of
$\overline{C_-}$ by the second. Finally, a result of
Geroch's \cite{Geroch:Domain} (relying on a construction of Hawking's)
shows that $M$ is globally hyperbolic (there is a Cauchy surface for
which every timelike curve intersects it exactly one time) under these assumptions, and in particular $L_{++}$
and $L_{--}$ are invertible since any element of $\Ker L_{++}$ would
vanish globally, and similarly for elements of $\Ker L_{++}^*$. One can then use the relative index
theorem of Melrose \cite[Chapter~6]{Melrose:Atiyah} to compute the index of $L$ on other weighted spaces.

For the Feynman propagator there is no simple direct identification of
the poles of $\hat N(L)(\sigma)$. However, in Minkowski space, one can
compute these exactly by virtue of a Wick rotation (Proposition
\ref{thm:boundaryscaling}), and further even
show the invertibility of $L$ on appropriate weighted spaces (Theorem \ref{thm:invertibility}). Namely,
the poles of $\hat N(L)(\sigma)$ are exactly those values of $\sigma$
for which the operator $\Delta_{\sphere^{n-1}}+(n-2)^2/4+\sigma^2$ is {\em not}
invertible, i.e.\ $\sigma$ is of the form $\pm
i\sqrt{\lambda+(n-2)^2/4}$, $\lambda$ an eigenvalue of
$\Delta_{\sphere^{n-1}}$, i.e.\ $\lambda=k(k+n-2)$, $k\in\Nat$, so
$\lambda+(n-2)^2/4=(k+(n-2)/2)^2$, and thus $\sigma=\pm
i(\frac{n-2}{2}+k)$. For future reference, we define
\begin{equation}
  \label{eq:indicialroots}
  \Lambda = \set{ \pm
(\frac{n-2}{2}+k) : k \in \mathbb{N}_0
}
\end{equation}
This gives a gap between
the two strings of poles with positive and negative imaginary parts,
and for $|l|<\frac{n-2}{2}$, $L_{+-}$ and $L_{-+}$ are invertible (and
are adjoints of each other on dual spaces). Since the framework we set
up is stable under {\em general b-ps.d.o.\ perturbations}, we conclude
that for general sc-metric perturbations $g$ of the Minkowski metric
$g_0$, $L_{g,+-}$ and $L_{g,-+}$ have the same properties, provided
the $|l|$ is taken slightly smaller:

\begin{thm}\label{thm:invertibility}
Let $\delta\in (0,\frac{n-2}{2})$. Then there exists a neighborhood $\cU$ of the Minkowski
metric $g_0$ in $\CI(M;\Sym^2\Tsc^*M)$ (i.e.\ in the sense of
sc-metrics, see \eqref{eq:symmetric-tensor-basis} and the paragraph
following it) such that for $g\in\cU$,
\begin{equation}\label{eq:invertiblemap}
L_{g,+-} \colon \cXpm^{m, l} \lra \cYpm^{m - 1,l},
\end{equation}
with $\cXpm^{m, l},\cYpm^{m - 1,l}$ as in \eqref{eq:spacenotation},
is invertible for
$|l|<\frac{n-2}{2}-\delta$ and $m$ satisfying the forward Feynman condition
for $+-$ in \eqref{eq:choices}, strengthened as in Theorem~\ref{thm:fredholm}, and where $\cXpm^{m,l}$ is the domain
of the Feynman wave operator defined in \eqref{eq:spacenotation}.  The same is true for $L_{g,-+}$ with
$+-$ replaced by $-+$ in all the spaces.
\end{thm}
\begin{proof}
  For the actual Minkowski metric $g_0$, the invertibility is a restatement of
  Theorem \ref{thm:scalinglemma} below.  Since the estimates in
  \eqref{eq:propagationestimatesimproved} hold uniformly on a
  sufficiently small neighborhood $\cU'$ of $g_0$, $L_{g, +-}$ defines a
  continuous bounded family mapping as in \eqref{eq:invertiblemap},
  and thus is invertible on a possibly smaller neighborhood $\cU$.
\end{proof}

Taking into account the construction of $L$ (see \eqref{eq:Ldef}), for
metrics $g$ in the neighborhood $\cU$ in the theorem, we deduce that 
\begin{equation}\label{eq:realinvertiblemap}
\Box_{g,+-}:\cXpm^{m,l+\frac{n-2}{2}}\to \cYpm^{m,l+\frac{n-2}{2}+2}
\end{equation}
is invertible for $|l|<\frac{n-2}{2}-\delta$.  Its inverse is indeed
the forward Feynman propagator (which
is well defined on space $\cYpm^{m,l+\frac{n-2}{2}+2}$ with weight
$l$ in the stated range,
\begin{equation}\label{eq:feynmanprop}
\Box^{-1}_{g,fey}:\cYpm^{m,l+\frac{n-2}{2}}\to \cXpm^{m,l+\frac{n-2}{2}+2}.
\end{equation} 
The same for $+-$ replaced by $-+$ and ``forward'' replaced by ``backward''.

\begin{rem}
The class of perturbations we consider {\em does not} preserve the
radial point structure at $\SNb^*S_\pm$. Nonetheless, the {\em
  estimates} the radial point structure implies for $L$ and $L^*$ are
preserved, much as discussed for Kerr-de Sitter spaces in \cite{VD2013}.
\end{rem}

\section{Wick rotation (complex scaling)}\label{sec:Wick}

In this section we work only with the Minkowski metric, which we
continue to denote by $g$.  We now explain Wick rotations in Minkowski
space, where it amounts to
replacing $\Box_{g}=D_{z_n}^2-D_{z_1}^2-\ldots-D_{z_{n-1}}^2$ by
\begin{equation}\label{eq:rotatedbox}
\Box_{g,\theta}=e^{-2\theta}D_{z_n}^2-D_{z_1}^2-\ldots-D_{z_{n-1}}^2
\end{equation}
where $\theta$ is a complex parameter. While it may seem that we are
using rather sophisticated techniques for a simple operator, this is
in some ways necessary since we need invertibility on our variable
order function spaces, which would not be so easy to show using very
simple techniques!

Concretely, consider complex
scaling, corresponding to pull-back by the diffeomorphism
$\Phi_\theta(z)=(z_1,\ldots,z_{n-1},e^{\theta}z_n)$ for
$\theta\in\RR$, i.e.\ considering
$U_\theta^*\Box(U_\theta^{-1})^*$, where $U_\theta = (\det D
\Phi_\theta)^{1/2} \Phi_\theta^* f$, extending the result to an
analytic family of operators in
$\theta\in\Cx$ (near the reals). This gives rise to the family
$\Box_{g,\theta}$.  Letting
\begin{equation}\label{eq:Ldeftheta}
L_\theta=\rho^{-(n-2)/2}\rho^{-2}\Box_{g,\theta}\rho^{(n-2)/2},
\end{equation}
as
soon as $\im\theta\in(-\pi,\pi)\setminus\{0\}$, $L_\theta$ 
is a an elliptic b-differential operator; when $\theta=\pm i \pi/2$,
one obtains the Euclidean
Laplacian $\Box_{g, \pm i \pi/2} = \Delta_{\RR^n} $. In the elliptic region the corresponding operator $L_\theta$ satisfies the
Fredholm estimates uniformly for $L_{\theta,+-}$ (and its adjoint, for which
the imaginary part switches sign, but one propagates estimates
backwards)
when $\im\theta\geq 0$, and for $L_{\theta,-+}$ when $\im\theta\leq
0$. 

The main analytic property that we will use below for the operators
$L_\theta$ is that for regularity functions $m$ chosen to satisfy say
the forward ($+-$) Feynman condition, the corresponding operators $L_{\theta,
  +-}$ satisfy estimates
  \begin{equation}
    \label{eq:uniformfredholm}
    \norm[\Hb^{m, l}]{u} \le C ( \norm[\Hb^{m-1, l} ]{L_{\theta} u} +
    \norm[\Hb^{m', l'} ]{u} ).
  \end{equation}
  uniformly in $\theta$ for $m, l$ corresponding to $+-$ and $m'
  < m$, $l' < l$, meaning precisely that there is a constant $C$ such that for $|\theta| < \delta_0, \im
  \theta \ge 0$
  for $u \in \Hb^{m, l}$,
  \eqref{eq:uniformfredholm} holds provided $m, l$ satisfy the $+-$
  Feynman condition and $- l \not \in \Lambda$.  For $|\theta| <
  \delta_0, \im \theta \le 0$ they hold provided $m, l$ satisfy the $-+$
  Feynman condition and $l \not \in \Lambda$.  (Note that $\Lambda = -
  \Lambda$ so actually the conditions on $l$ are the same.) The
  reason for the uniformity is that all of the ingredients are
  uniform; this is standard for elliptic estimates. On the other hand, it holds for real
  principal type estimates where the imaginary part of the principal
  symbol amounts to complex absorption, provided one propagates
  estimates in the {\em forward} direction of the Hamilton flow if the
  imaginary part of the principal symbol is $\leq 0$ (which is the
  case for $\im\theta\geq 0$, $\theta$ small) and backwards
  along the Hamilton flow if the imaginary part of the principal
  symbol is $\geq 0$, as shown by Nonnenmacher and
  Zworski \cite{Nonnenmacher-Zworski:Quantum} and Datchev and Vasy
  \cite{Datchev-Vasy:Gluing-prop} in the semiclassical microlocal setting
  and, as is directly relevant here, extended to the general b-setting by Hintz and Vasy
  \cite[Section~2.1.2]{HVsemi}. Moreover, at radial points in the
  standard microlocal setting this was shown by Haber and Vasy
  \cite{Haber-Vasy:Radial}, and the proof of
  Proposition~\ref{thm:propofsing} can be easily modified in the same
  manner so that non-real principal symbol is also allowed at the
  b-radial points. Finally, the normal operator constructions are also
  uniform since they rely on estimates for the Mellin transformed
  family which are uniform as we stated; the resonances (poles) of the inverse
  of this family thus a priori vary continuously, so in particular
  near an invertible weight for $\theta=0$ one has uniform estimates.
  (In fact we will show in Proposition \ref{thm:boundaryscaling}
  below that the poles of the complex scaled normal families are
  constant, i.e.\ do not vary with $\theta$.)

Note that the estimates in \eqref{eq:uniformfredholm} are \emph{not}
the standard elliptic estimates.  Indeed, the term on the left hand
side is in a space of differentiability order one lower than
ellipticity provides.  The point is that the estimates in
\eqref{eq:uniformfredholm} are exactly those which are uniform down to
$\im \theta = 0$.

The family of operators $L_\theta$ defines a family of Mellin
transformed normal operators on the boundary, $\hat
N(L_\theta)(\sigma)$ as above, and we have, still for $g$ equal to the
Minkowski metric, that
\begin{equation}
  \label{eq:8}
  \hat{N}(L_{\pm i \pi / 2})(\sigma) = \Delta_{\sphere^{n-1}}+(n-2)^2/4+\sigma^2.
\end{equation}

We recall the theorem of Melrose describing the behavior of the
elliptic operators $L_\theta$ for $\im \theta \neq 0$, which is a
special case of our more general framework in that elliptic operators
are also Fredholm on variable order Sobolev spaces in view of our results.

\begin{thm}[Melrose \cite{Melrose:Atiyah}, with Theorem~\ref{thm:fredholm} here
  giving the variable order version] \label{thm:bopfred}
  Let $P$ be an elliptic b-differential operator of order $k$ on a manifold with
  boundary $M$, and assume that $\hat N(P)^{-1}(\sigma)$ has no poles on
  the line $\im \sigma = -l$.  Then the operator $P$ satisfies
  \begin{equation}
    \label{eq:Fredholmelliptic}
  \| u \|_{\Hb^{s + k, l}} \le C(\| P u \|_{\Hb^{s , l}} + \| u
  \|_{\Hb^{- N, l'}}),
  \end{equation}
for any $N > 0$ and some $l' < l$.  In particular, 
$$
P \colon \Hb^{s + k, l} \lra \Hb^{s, l}
$$
is Fredholm.
\end{thm}

Thus the set $\Lambda$ in \eqref{eq:indicialroots} gives the set of
weights $l$ for which 
$$
\Delta_{\RR^n} \colon \Hb^{m + 1, l + (n - 2)/2} \lra \Hb^{m - 1, l +
  (n - 2)/2 + 2}
$$
is Fredholm; indeed by the definition of $L_\theta$ in \eqref{eq:Ldeftheta},
we see that
$$
L_{i\pi/2}=\rho^{-(n-2)/2}\rho^{-2}\Delta_{\RR^n}\rho^{(n-2)/2} \colon \Hb^{m + 1, l} \lra \Hb^{m - 1, l}
$$
is Fredholm exactly when $- l \not \in \Lambda$.
Consider the elliptic operators $L_{\theta}$  as maps between forward
Feynman b-Sobolev spaces
\begin{equation}\label{eq:ellipticmaps}
  L_{\theta, +-} \colon \Hb^{m + 1, l} \lra \Hb^{m-1,l}.
\end{equation}
In Section \ref{sec:normalfamily} below, we will prove in Proposition
\ref{thm:boundaryscaling} that for the
$\theta-$dependent family of  Mellin transformed normal operators of the
complex scaled Feynman operators,
$\hat{N}(L_{\theta, +-})(\sigma)$, \textit{the inverse families have equal
  poles.}  Thus the set $\Lambda$ in \eqref{eq:indicialroots} is in
fact the set of all poles of the inverse families $\hat{N}(L_{\theta, +-})$ in the forward Feynman setting.  The same holds for $-+$.
As a corollary to Theorem \ref{thm:bopfred} and the fact that the
index of a continuous family of Fredholm operators is constant (see
\cite{TaylorI}), we obtain the following:
\begin{lemma}\label{thm:constantindexelliptic}
For $\Lambda$ as in \eqref{eq:indicialroots} and $-l \not \in
\Lambda$, the maps in
  \eqref{eq:ellipticmaps} form a continuous
  Fredholm family and thus have constant index for $\theta \in (0, \pi/2]$.
\end{lemma}

 \begin{figure}\label{scalingspace}
    \centering
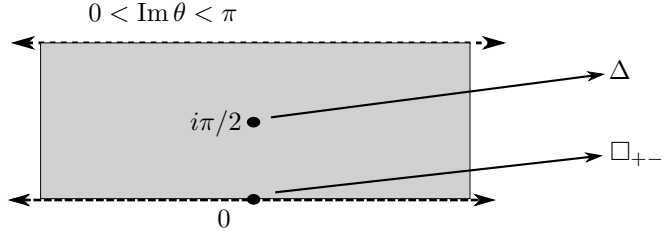
    \caption{The index of $L_{\theta}$ is constant in the grey region
      by the standard continuity of the index for Fredholm families.
      In Theorem \ref{thm:constantindex} we show the continuity of the
    index extends up to the dashed line at bottom, i.e.\ to $L_{+-}$.}
      \end{figure}

\subsection{Index of $L_{\theta, +-}$}

To prove the invertibility theorem, we will first establish the
following 
\begin{thm}\label{thm:constantindex}
For fixed $m, l$ satisfying the forward Feynman condition in $L_{+-}$, and such that $l
\not \in \Lambda$, for $\mathrm{Im} \ \theta \in (0, \pi/2)$,
  \begin{equation}
    \label{eq:equalindices}
    \ind(L_{\theta} \colon \Hb^{m + 1, l} \lra \Hb^{m-1,l}) = \ind(L_{+-}
    \colon \cX^{m, l} \lra \cY^{m-1,l}).
  \end{equation}
\end{thm}
\begin{proof}
  This follows from the mere fact that the estimates in
  \eqref{eq:uniformfredholm} hold uniformly in $\theta$ for $m, l$
  corresponding to $+-$ and $m'  < m$, $l' < l$.

  Assume first that $L_{+-}$ on the right hand side of
  \eqref{eq:equalindices} is invertible.  Then one can drop the compact
  error terms, and thus then the estimates take the form
  \begin{equation}\label{eq:invest}
\|u\|_{\Hb^{m,l}}\leq C\|Lu\|_{\Hb^{m-1,l}},\ \|v\|_{\Hb^{1-m,-l}}\leq
C\|L^* v\|_{\Hb^{-m,-l}},
\end{equation}
where again $L^*$ is the adjoint of $L$ with respect to the $\Lb^2$
pairing (see \eqref{eq:l2bdef}).  To see that the estimate on the
right follows,
since $\Hb^{-m, -l}$ is dual
to $\Hb^{m, l}$ with respect to
the $\Lb^2$ pairing, using the surjectivity of $L$ to go to
the second line we have
  \begin{gather*}
    \| L^*v \|_{\Hb^{-m, -l}} = \sup_{ \| w \|_{\Hb^{m,l}} = 1} \la L^* v,
    w \ra_{\Lb^2} \ge \sup_{ \| w \|_{\cX^{m,l}} = 1} \la v, L w
    \ra_{\Lb^2} \\ \ge \frac 1C \sup_{ \| g \|_{\Hb^{m - 1,l}}} \la v, g
    \ra_{\Lb^2} = \frac 1C \| v \|_{\Hb^{1 - m , -l}}.
 \end{gather*}

We claim that the estimates in \eqref{eq:invest} imply the
analogous estimates also hold for $L_\theta$, $\im\theta$ small
with $\im \theta > 0$, namely that
\begin{equation}\label{eq:Lthetainj}
\|u\|_{\Hb^{m,l}}\leq C\|L_\theta u\|_{\Hb^{m-1,l}},\ \|v\|_{\Hb^{1-m,-l}}\leq
C\|L_\theta^* v\|_{\Hb^{-m,-l}}.
\end{equation}
Otherwise, for example for the first estimate, we would have a
sequence $\theta_j \to 0$ with $\im \theta_j > 0$ and $u_j$ with 
$$
\|u_j \|_{\Hb^{m, l}} = 1 \mbox{ and } L_{\theta_j} u_j \to 0  \mbox{
  in } \Hb^{m - 1,
  l}.
$$
  Extracting a strongly convergent subsequence of the $u_j$ in $\Hb^{m', l'}$
for $m' < m$ and $l' < l$, by the uniform estimates in
\eqref{eq:uniformfredholm} we would obtain a limit $\wt{u}$ with
$\wt{u} \neq 0$ and $L \wt{u} = 0$, a contradiction.  A similar
argument shows that the second estimate also holds for small $\theta$
with $\im \theta > 0$.

Now as soon as $\im\theta\neq 0$, these give
improved estimates by elliptic regularity, namely
$$
\|u\|_{\Hb^{m+1,l}}\leq C\|L_\theta u\|_{\Hb^{m-1,l}},\
\|v\|_{\Hb^{2-m,-l}}\leq C\|L_\theta^* v\|_{\Hb^{-m,-l}}.
$$
Indeed these follow since for $\im \theta > 0$, $L_\theta$ and
$L^*_\theta$ are Fredholm
maps from $\Hb^{m' + 1, l}$ to $\Hb^{m' - 1, l}$ for any $m'$ and by
\eqref{eq:Lthetainj} are injective for the given $m$ and $l$ and thus
for any $m'$ by elliptic regularity.
Thus, for example taking $m = s$ to be constant in the first inequality
and $m = -s + 1$ in the second inequality gives that $L_\theta$ is
injective and surjective with domain $\Hb^{m, l}$  (which again by
elliptic regularity means that $L_\theta$ is an isomorphism for any
$m$ and the given $l$).  This establishes the theorem in the case that
$L_{+-}$ is invertible on the spaces under consideration.

If $L_{+-}$ is not invertible but is Fredholm, one can get back to the
same setting by adding finite dimensional function spaces to the
domain and target as usual, showing that the index is stable under
this deformation. Concretely, let
\begin{equation*}
  \begin{split}
    V &:= \Ker(L_{+-} \colon \cX^{m, l} \lra \Hb^{m-1, l}) \\
    W &:= \Coker(L_{+-} \colon \cX^{m, l} \lra \Hb^{m-1, l}),
  \end{split}
\end{equation*}
where by definition the cokernel in the second line is the orthogonal
complement of the range with respect to some (fixed) inner product.  The map $\wt{L}_\theta$ from $W \oplus
\cX^{m, l} = W \oplus V \oplus V^\perp $ to $V \oplus \Hb^{m
  -1, l} = V \oplus W \oplus  W^\perp$ which takes $w + v + v'$ to $v +
w + L_\theta v'$ is an isomorphism for $\theta = 0$, and by the above
analysis is also an isomorphism for $\theta$ small with $\im \theta > 0$.
Therefore the Fredholm index of the Feynman
propagators for Minkowski space is the same as that of
$\Delta_{\RR^n}$ acting on a weighted b-space with the same
weight. \end{proof}

We can then use the relative index
theorem of Melrose \cite[Chapter~6]{Melrose:Atiyah}, which expresses
the difference in the index of a b-differential operator at different
weights as the sum of the residues of the normal operator at
appropriate indicial roots.  This can be
extended from the elliptic setting considered there to ours without
any difficulties, to compute the index of $L$ on other weighted
spaces; here in fact because of the Wick rotation we can use the
elliptic result directly.
\begin{cor}\label{thm:explicitindex}
Under assumptions as in Lemma \ref{thm:constantindexelliptic}, 
  \begin{equation}
    \label{eq:12}
    \ind(L_{\theta, +-} \colon \cX^{m, l} \lra \cY^{m - 1, l}) =
    - \sgn(l) N(\Delta_{\mathbb{S}^{n -1} }  + (n-2)^2/4 ; l),
  \end{equation}
where $N(\Delta_{\mathbb{S}^{n -1} }  + (n-2)^2/4; l)$ is the number of
eigenvalues $\lambda$ of $\Delta_{\mathbb{S}^{n -1} }  + (n-2)^2/4$
with $\lambda < l^2$.  In particular, 
\begin{equation*}
|l| < (n - 2)/2 \implies  \ind(L_{g,+-} \colon \cX^{m, l} \lra \cY^{m - 1, l}) = 0.
\end{equation*}
\end{cor}
\begin{proof}
  By Theorem \ref{thm:constantindex}, we have that
  \begin{equation*}
    \begin{split}
     & \ind(L_{\theta, +-} \colon \cX^{m, l} \lra \cY^{m - 1, l}) \\
     & \qquad =
      \ind(L_{i\pi/2} \colon \Hb^{m + 1, l} \lra
      \Hb^{m - 1, l}) \\
      & \qquad =
      \ind(\Delta_{\RR^{n}} \colon \Hb^{m + 1, l + (n - 2)/2} \lra
      \Hb^{m - 1, l + (n - 2)/2 + 2}),
    \end{split}
  \end{equation*}
and the latter was computed by Melrose, see \cite[Section
6.2]{Melrose:Atiyah}, or the interpretation in \cite[Theorem
2.1]{GH2008} where it is shown to be exactly the right hand side of \eqref{eq:12}.
\end{proof}

\subsection{Invertibility of the Feynman problem for $\Box_{g,
    \theta}$ down to $\theta = 0$.}

It follows from Theorem \ref{thm:bopfred} and \eqref{eq:8}, together
with the spectral theory of the sphere discussed above, that
$$\Delta_{\RR^n} \colon \Hb^{m+1, l + (n-2)/2} \lra \Hb^{m-1, l +
  (n-2)/2 + 2}$$ is Fredholm as long as $-l \not \in \Lambda$ where
$\Lambda$ is defined in \eqref{eq:indicialroots}.  In fact, we have
\begin{thm}\label{thm:laplaceinverse}
The map $\Delta_{\RR^n} \colon \Hb^{m+1, l + (n-2)/2} \lra \Hb^{m-1, l +
 (n-2)/2 + 2}$ is invertible provided $|l| < (n-2)/2,$ $m\in\CI(\Sb^*\overline{\RR^n})$.
\end{thm}

\begin{proof}
This is shown in the proof of \cite[Lemma 3.2]{CCH2006}, for
$m\in\RR$.  Indeed, they show
using the maximum principle and elliptic regularity that there can be
no nullspace of $\Delta$ in $\Hb^{m, l}$ for any $l > 0$ (and the same must be
true for the formal adjoint), from which the result follows since the
operator is Fredholm.  Our results give the
general Fredholm statement for arbitrary
$m\in\CI(S^*\overline{\RR^n})$, and elliptic regularity then gives
that any element of the kernel is in $\Hb^{\infty,l+(n-2)/2}$, with an
analogous statement for the cokernel, and these are trivial in turn by the
constant $m$ result.
\end{proof}

Consider the map
\begin{equation}
  \label{eq:scaledboxmap}
  \Box_{g, \theta} \colon \cX^{m,(n-2)/2 + l}_{+-}(\theta) \lra
\Hb^{m - 1,(n-2)/2 + l + 2 }
\end{equation}
where $\cX^{m, l}_{+-}(\theta) = \{  u \in
  \Hb^{m, l} :  \Box_{g, \theta} u \in \Hb^{m - 1, l +  2} \}$ is a
  $\theta$-dependent space with the graph norm,
$$
\|u\|^2_{\cX^{m,l}_{+,-}(\theta)}=\|u\|^2_{\Hb^{m,l}}+\|\Box_{g, \theta} u\|^2_{\Hb^{m - 1, l +  2}},
$$
so  by the elliptic
  estimates discussed above, 
$$
\cX^{m, l}_{+-}(\theta) = \left\{
    \begin{array}{ccc}
      \cX^{m, l + (n - 2)/2}_{+-} & \mbox{if} & \theta \in \mathbb{R}
      \\
      \Hb^{m + 1, l + (n - 2)/2} & \mbox{if} & \im \theta \in (0, \pi) 
    \end{array}
\right.,
$$
with the equivalence of norms uniform for $\theta$ in compact
subsets of $\RR\times(0,\pi)$.
(Here the $+-$ is just to remind us that $m + l$ satisfies the
conditions corresponding to $L_{g, +-}$, although this makes no
difference in the elliptic region.)  We will now study the set 
$$
\mathcal{D}_l = \{\theta : \im(\theta) \in [0, \pi/2] \mbox{ and }
\Box_{g, \theta} \mbox{ mapping as in \eqref{eq:scaledboxmap} is
  invertible}.
\}
$$
We see that for $|l| < (n-2)/2$, $\mathcal{D}_l$ contains $i \pi /2 $
and is thus non-empty.
\begin{thm}\label{thm:scalinglemma}
  Let $|l| < (n-2)/2$.  The set $\mathcal{D}_l$ contains the entire
  closed strip $\{ \im \theta \in
  [0, \pi/2] \} $.  In particular $\Box_{g, +-}$ mapping as in
  \eqref{eq:realinvertiblemap} is invertible for $g$
  equal to the Minkowski metric and $|l| < (n - 2)/2$. 
\end{thm}

We will prove Theorem \ref{thm:scalinglemma} by arguing along lines
similar to those in \cite{MV2004, MV2005}, which in turn follow the
development in \cite{HS1996}.  

\begin{proof}[Proof of Theorem \ref{thm:scalinglemma}]

We will define a subspace $\mathcal{A} \subset L^2 = L^2(\RR^n)$ of so-called
analytic vectors and a family of maps
\begin{equation}
  \label{eq:9}
  U_\theta \colon \mathcal{A} \lra L^2,
\end{equation}
for $\theta$ in an open neighborhood $\mathcal{D} \subset \mathbb{C}$
of $0$ with the following properties:
\begin{enumerate}
\item For $\theta \in \mathbb{R}$, $U_\theta$ is unitary on $L^2$.
\item For $f \in \mathcal{A}$ and $\theta \in\mathcal{D}$, $$U_\theta
  \Box_{g, \theta_0} U_\theta^{-1} f = \Box_{g, \theta + \theta_0} f.$$
In particular, $U_\theta$ is injective and $\mathcal{A}$ is in the
range of $U_\theta$ for $\theta \in \mathcal{D}$.
\item $U_\theta \mathcal{A}$ is dense in $\Hb^{m, l}$ for all $\theta
  \in \mathcal{D}$ and any $m \colon \Sb^* M \lra \mathbb{R}$, $l \in \mathbb{R}$.
\end{enumerate}

We will then leverage the properties of $\mathcal{A}$ and $U_\theta$
to prove Theorem \ref{thm:scalinglemma} as follows.  Recall that, by
Theorem \ref{thm:constantindex},
$\Box_{g, \theta}$ as in \eqref{eq:scaledboxmap} is a Fredholm map of
index zero.
Since it is invertible for $\theta' = i \pi/2$, it is invertible for
$\theta$ near $\theta'$.  It follows by the analytic Fredholm theorem that 
\begin{equation}
  \label{eq:11}
  \Box_{g, \theta}^{-1} \colon \Hbpm^{m - 1,l} \lra \Hbpm^{m + 1,l}
\end{equation}
extends to a meromorphic family of operators in the strip
$\{0 < \im \theta < \pi \}$ with finite rank poles.  In particular, if
$\wt{\theta}$ is a putative pole, then for $\theta$ near $0$,
\begin{equation}
  \label{eq:16}
  \Box^{-1}_{g, \wt{\theta} + \theta} = \sum_{j =-N}^{-1} A_j \theta^j + M_\theta, \mbox{ where } M_\theta \mbox{ is holomorphic.}
\end{equation}
Thus if $\wt{\theta}$ is indeed a pole, by the density of
$\mathcal{A}$ we may choose $f, h$ such that, e.g.\ $\la f, A_1 h\ra
\neq 0$ and thus $\la f, \Box^{-1}_{g,
  \wt{\theta} + \theta} h\ra$ has a pole at $\theta = \wt{\theta}$.
On the other hand the matrix elements satisfy
\begin{equation}\label{eq:crux}
\la f, \Box^{-1}_{g,\wt{\theta} +  \theta} \notg \ra_{L^2}  = \la U_{\overline{\theta}} f ,
\Box^{-1}_{g, \wt{\theta}} U^{-1}_{\theta} \notg\ra.
\end{equation}
We will see that for $h \in \mathcal{A}$, both $U_\theta h$ and
$U_\theta^{-1} h$ are analytic for $\theta \in \mathcal{D}$, so the
matrix elements of $\Box_{g,\wt{\theta} + \theta}^{-1}$ are analytic
functions for
$\theta \in \mathcal{D}$, and thus
$\Box_{g,\wt{\theta} +  \theta}^{-1}$ has no poles in $0 < \im \theta \le
\pi/2.$

We have proven that $\Box_{g, \theta}$ is invertible only for those $\theta$ with $0 < \im
\theta < \pi$.  Since $\Box_{g,
  \theta}$ is not strictly speaking an analytic Fredholm family on an open set
containing $\theta = 0$ (the domain changes according to whether
$\im \theta = 0$ or not and $\theta = 0$ lies on the boundary of $0
\le \im \theta \le \pi /2$) we need a different argument there.  Pick
$\theta_0$ close enough to $0$ so that the density statements
for $U_\theta \mathcal{A}$ hold on an open set including $\theta= -  \theta_0$.  
Assuming for contradiction that $\Box_{g, +-}$ is not invertible for
$l$ in the given range, it will suffice to construct
elements $f, \notg \in \mathcal{A}$ such that 
\begin{equation}\label{eq:moredivergence}
\la f, \Box^{-1}_{g, \theta_0 +
  \theta} \notg\ra_{L^2} \mbox{ diverges as } \theta_0 + \theta \to 0 \mbox{
  in } \im(\theta
+ \theta_0) > 0,
\end{equation}
since then by
\eqref{eq:crux} with $\wt{\theta} = \theta_0$ we will have a contradiction.
Note that by this assumption there exists $u_0$ lying in $\Hbpm^{m - 1, l +
  (n-2)/2 + 2}$ such that
$$
u_0 \not \in \Ran(\Box_{g,+-} \colon \cX_{+-}^{m,
  l + (n-2)/2} \lra \Hb^{m - 1,   l + (n-2)/2 + 2})
$$ 
since by Theorem \ref{thm:constantindex} the map has index zero.  Since
\eqref{eq:scaledboxmap} is
Fredholm and $\mathcal{A}$ is dense, we may instead choose $\notg \in
\mathcal{A}$ such that also $\notg \not \in \Ran(\Box_{g,
  0})$.  Using the invertibility proved above, for $\im (\theta_0 +
\theta) > 0$, we consider
$\Box^{-1}_{g, \theta_0 + \theta} \notg \in \Hb^{m + 1, l} \subset
\cX^{m, l+(n-2)/2}_b$.  We claim that 
$$
\| \Box^{-1}_{g, \theta_0 + \theta}  \notg \|_{\cX^{m, l + (n-2)/2}_{+-}}
\mbox{ diverges as } \theta + \theta_0 \to 0
$$ 
Indeed, otherwise $\Box^{-1}_{g, \theta_0 + \theta}  \notg$ converges subsequentially
to some $u \in \cX^{m, l}_b$ weakly, and by a standard
argument we must have $L_0 u = \notg$, which is impossible by
assumption.

Note that this does not guarantee that
\eqref{eq:moredivergence} holds for any $f \in \mathcal{A}$; this
requires a further argument.  To see this, we use the uniform Fredholm
estimates in \eqref{eq:uniformfredholm}, which in terms of $\Box_{g,
  \theta_0 + \theta}$ and applied to $\Box_{g,
  \theta_0 + \theta}^{-1} \notg$ take the form
  \begin{equation}
    \label{eq:uniformfredholmbox}
    \begin{split}
      \|\Box_{g, \theta_0 +
          \theta}^{-1} \notg \|_{\Hb^{m, l + (n- 2)/2}} &\le C ( \norm[\Hb^{m-1, l + (n - 2)/2 +
        2} ]{\notg} \\
     &\qquad + \|\Box_{g, \theta_0 + \theta}^{-1} \notg \|_{\Hb^{m', l' + (n - 2)/2 } } ),
    \end{split}
  \end{equation}
where $m' < m$ and $l' < l$.  Letting $\theta_j$ be a sequence with
$\theta_j \to -\theta_0$, let $c_j = \|\Box_{g, \theta_0 +
          \theta_j}^{-1} \notg \|_{\Hb^{m, l + (n- 2)/2}} $, and
        let 
$$
u_j = c_j^{-1}\Box_{g, \theta_0 +
          \theta_j}^{-1} \notg.
$$ 
 By \eqref{eq:uniformfredholmbox} and
        the compact containment of $\Hb^{s, \ell} \subset \Hb^{s',
          \ell'}$ when $s' < s$ and $\ell' < \ell$,
        the $u_j$ converge subsequentially (dropped from the notation) to a non-zero element $u \in \Hb^{m, l +
          (n- 2)/2}$.  It follows that 
 $$
\la \Box_{g, \theta_0 + \theta_j}^{-1} \notg, u
      \ra_{\Hb^{m, l +
          (n- 2)/2}} = c_j(1 + o(1)),
$$
where $o(1) \to 0$ as $j \to \infty$.
We claim that there is a $\delta_0 > 0$ such that for any $\wt{f}$ with
        $\|\wt{f} - u \|_{\Hb^{m, l + (n- 2)/2}} < \delta_0$,
that $\la \Box_{g, \theta_0 + \theta_j}^{-1} \notg, \tilde f
      \ra_{\Hb^{m, l +
          (n- 2)/2}}$ is also divergent.
        Indeed, 
        \begin{equation}\label{eq:closetou}
          \begin{split}
            \la \wt{f}, \Box_{g, \theta_0 + \theta_j}^{-1} \notg\ra_{\Hb^{m, l +
          (n- 2)/2}} &= \la u,  \Box_{g, \theta_0 + \theta_j}^{-1} \notg
      \ra_{\Hb^{m, l +
          (n- 2)/2}}   \\
&\qquad  + \la \wt{f} - u, \Box_{g, \theta_0 + \theta_j}^{-1} \notg
      \ra_{\Hb^{m, l +
          (n- 2)/2}} \\
      &\ge c_j(1 + o(1)) -  C c_j \delta_0 \ge  \frac 12 c_j,
          \end{split}
        \end{equation}
for $C \delta_0 < 1/3$  and $j$ large.  This is not exactly the desired
divergence in \eqref{eq:moredivergence} since the inner product is not
$L^2$.  Define
\begin{equation}
  \label{eq:japanese}
  \la z \ra = (z_1^2 + \dots + z_n^2 + 1)^{1/2},
\end{equation}
and let $P \in \Psib^{m}(\mathbb{R}^n)$ be elliptic and
self-adjoint.  Then (since by the paragraph below \eqref{eq:l2bdef} we have
$\Hb^{0, n/2} = L^2(\RR^n)$)
$$
\la z \ra^{l - 1} (P + i) \colon \Hb^{m, l + (n-2)/2} \lra L^2(\RR^n).
$$
and we may take the $\Hb^{m, l + (n - 2)/2}$ inner product to be 
$$
\la u, v \ra_{\Hb^{m, l + (n - 2)/2}} = \la \la z \ra^{l - 1} (P + i) u, \la z \ra^{l - 1} (P + i) v\ra_{L^2}.
$$ 
Using the density of $\mathcal{A}$ in all weighted b-Sobolev spaces,
we choose $$\wt{f} = (P - i)^{-1}\la z \ra^{-2l + 2} (P + i)^{-1} f$$
for some $f \in \mathcal{A}$ such that $\wt{f}$ within $\delta_0$ of
$u$ in $\Hb^{m, l + (n - 2)/2}$.  Thus
$$
 \la \wt{f}, \Box_{g, \theta_0 + \theta_j}^{-1} \notg\ra_{\Hb^{m, l +
          (n- 2)/2}} =  \la f, \Box_{g, \theta_0 + \theta_j}^{-1} \notg\ra_{L^2},
$$
and \eqref{eq:moredivergence} is established, which means that up to
the construction of $\cA$ and $U_\theta$ and showing that the
properties claimed for them hold, the proof is complete.


It remains to define $\mathcal{A}$ and $U_\theta$ and prove that they
have the properties i)-iii) stated above.  Following \cite{MV2004}, we define $\mathcal{A}$ to be the space of $f
\in C^\infty(\RR^{n-1} \times \RR)$ such that, writing $z = (z'',
z_n)$ with $z'' \in \RR^{n-1}$, we have that $f(z'', z_n)$ is the
restriction to $\zeta \in \RR$ of an entire function $f(z'', \zeta)$
which satisfies
\begin{equation}\label{eq:A1}
  \sup_{|\re{\zeta}| < C |\im{\zeta}|} |f(z'', \zeta)| \la \zeta
  \ra^{N} < + \infty,
\end{equation}
for any $C, N > 0$ where $\la \zeta \ra = (1 + |\zeta|^2)^{1/2}$, and
also assume that
\begin{equation}
  \label{eq:A2}
  \supp f(z'', \zeta) \subset K\times \mathbb{C},
\end{equation}
where $K \subset \RR^{n-1}$ is compact.
Finally, for $f \in \mathcal{A}$ let
\begin{equation}
  \label{eq:uthetadef}
  U_\theta (f)(z'', z_n) : = e^{\theta}f(z'',e^{\theta} z_n).
\end{equation}
By the proof of \cite[Proposition 3.6]{MV2004}, for $|\im \theta| < \pi
/ 4$, $U_\theta
\mathcal{A}$ is dense in $L^2 = L^2(\RR^{n}, |dz|),$ where $|dz|$ denotes
Lebesgue measure.  Indeed, given $f \in C^\infty_c(\RR^{n})$, let 
\begin{equation*}
  f_t(z'', z_n) := \frac{1}{(4 \pi t)^{1/2}}\int_\RR
    e^{-(z_n -e^{\theta} y)^2/4t} e^{\theta} f(z'', y) dy.
\end{equation*}
Then the reference shows that $f_t \in \mathcal{A}$ and $U_\theta f_t
\to f$ in $L^2$ as $t \to \infty$.  Thus $U_\theta \mathcal{A}$ is
dense in $L^2 = \Hb^{0, n/2}$.  To see that $U_\theta \mathcal{A}$
is dense in $\Hb^{M, L}$, for any $f \in C^\infty_c$ take a sequence
$U_\theta \wt{f}_i$ with 
\begin{equation*}
  U_\theta \wt{f}_i \to (e^{2\theta} z_n^2 + |z''|^2 + i)^{-L - 2M + n/2}
  (\Box_{i\pi/2 + \theta}+ i)^{M} f
  \in  \Hb^{0, n/2}, 
\end{equation*}
and set $F_i := (e^{2\theta} z_n^2 + |z''|^2 + i)^{L + 2M - n/2}
  (\Box_{i\pi/2 + \theta}+ i)^{-M}
\wt{f}_i$.  Then in fact $F_i = U_\theta f_i$, where $f_i = (z_n^2 +
|z''|^2 + i)^{L + 2M -  n/2}
  (\Delta+ i)^{-M} \wt{f}_i$ where again $\Delta = \Box_{i\pi/2}$ is
  the Laplacian on $\RR^n$. Since $F_i = U_\theta f_i \to f$ in $H^{M,
    L}_b$, and since $\Hb^{M, L}$ is dense in $\Hb^{m, l}$ provided $M
  \ge L$ and $L \ge l$, the desired density is established.

\end{proof}

\subsection{Complex scaling for $N(L_\theta)$}\label{sec:normalfamily}

In this section we will apply another complex scaling to the normal
operators corresponding to the $L_\theta$.  Namely, let $m, l$ be
chosen for the forward Feynman problem $L_{+-}$, and consider
  the operators $L_{\theta, +-}$ defined in \eqref{eq:Ldeftheta}.  Let
  $H^m(\p M)$ denote the variable order Sobolev spaces obtained by
  restricting $m$ to $T^* \p M$ as described above.  Consider the
  operators
  \begin{equation}
    \label{eq:normaltheta}
    \hat{N}(L_{\theta, +-})(\sigma)\colon \cX^{m}_{\theta}(\p M) \lra H^{m-1}(\p M),
  \end{equation}
where $\cX^{m}_{\theta}(\p M) = \{ u \in H^{m}(\p M) :
\hat{N}(L_\theta)(\sigma) u \in H^{m - 1}(\p M) \}$.
\begin{prop}\label{thm:boundaryscaling}
The poles of the inverse family $\hat{N}(L_{\theta, +-})(\sigma)^{-1}$ are
independent of $\theta$ for
$\im \theta \in [0, \pi / 2]$.
\end{prop}
\begin{proof}
  As in the previous section, we wish to define a set of analytic
  vectors $\wt{\mathcal{A}} \subset L^2(\p M)$, and a family of maps
  $\wt{U}_\theta \colon \wt{\mathcal{A}} \lra L^2(\p M)$ defined for
  $\theta$ in an open set which we also call $\mathcal{D} \subset
  \mathbb{C}$, and such that conditions i), ii), and iii) below
  \eqref{eq:9} above hold.  The Proposition then follows exactly as in
  the proof of Theorem \ref{thm:scalinglemma} above.

  Consider homogeneous degree zero functions on $\mathbb{R}^n$ of the
  form
  \begin{equation}
    \label{eq:homozero}
    F = \frac{p_l(z_1, \dots, z_n)}{(z_1^2 + z_2^2 + \dots + e^{2\omega}z_n^2)^{l/2}} ,
  \end{equation}
  where $p_l$ is a homogeneous polynomial of degree $l$ and $\omega
  \in \mathbb{C}$ with $|\im(\omega)| < \pi/4$ .
  \begin{equation}
    \label{eq:analboundary}
    \wt{\mathcal{A}} = \set{
      f \in C^\infty(\p M)  :  f = \sum_{i = 1}^k F_i \rvert_{\absv{z} = 1}
    },
  \end{equation}
  or in words, $\wt{\mathcal{A}}$ consists of all finite sums of
  restrictions of homogeneous degree zero functions as in
  \eqref{eq:homozero} to the sphere.  Note that $\wt{\mathcal{A}}$ is
  dense in every Sobolev space; indeed, $\wt{\mathcal{A}}$ contains the
  spherical harmonics, which are restrictions to the sphere of
  harmonic polynomials, and which form a basis of every Sobolev space
  by Fourier series \cite{taylor:vol2}.
  For $\theta \in \mathbb{R}$, we define
  \begin{equation}
    \label{eq:13}
    V_\theta F = (\det D \Phi_\theta)^{1/2} \Phi_\theta^* F,
  \end{equation}
  with $\Phi_\theta$ as above, i.e.\ $\Phi_\theta(z_1, \dots, z_{n -
    1}, z_n) = (z_1, \dots, z_{n - 1}, e^\theta z_n)$, and thereby
  define, for $f = \sum_{i = 1}^k F_i \rvert_{\absv{z} = 1}$,
  \begin{equation}
    \label{eq:14}
    \wt{U}_\theta f =  \sum_{i = 1}^k (V_\theta F_i) \rvert_{\absv{z} = 1}.
  \end{equation}
To see that condition iii) holds, note that given $0 < \delta < \pi/4$ if we define 
$\wt{\cA}_{\delta} \subset \wt{\cA}$ to be those elements where the
functions $F_i$ in the definition of $\wt{\cA}$ have parameter $\omega$
(in \eqref{eq:homozero}) with $|\omega| \le \delta$, then
$\wt{\cA}_\delta$ remains dense in $L^2$.
satisfying $|\omega| < \delta$, then $\wt{\cA}_\delta$ remains dense
since it still contains the spherical harmonics.  On the other hand,
for $|\im \theta| < \pi / 4 $ and $\delta$ sufficiently small
$\wt{\cA}_\delta \subset \wt{U}_\theta \wt{\cA}$, so the density
result holds also for $\wt{U}_\theta \wt{\cA}$. 

Property ii) in the complex scaling process in this context simply
says that for $f \in \wt{\mathcal{A}}$ and $\theta \in\mathcal{D}$, $$\wt{U}_\theta
  N(L_{g, \theta_0}) \wt{U}_\theta^{-1} f = N(L_{g, \theta + \theta_0})
  f,$$ which follows directly from the corresponding statement for
  $\Box_{g}$.  Thus all three properties hold and the argument proceed
  exactly as in the previous section.

\end{proof}

\section{Module regularity and semilinear problems}\label{sec:semi}

Despite being the {\em simplest} natural spaces between which the operator $\Box_g$
extends to a Fredholm map, the $\cX_{\pm\pm}^{m, l}$ and
$\cY_{\pm\pm}^{m, l}$ have the following critical
shortcoming when one wants to analyze semilinear equations.  If one wishes for example to do a
Picard iteration using the inverse of $\Box_g$ on some space (as we do
below e.g.\ in the proof of Proposition \ref{thm:semilinear})
invertibility restricts the range of weights $l$.  Furthermore,
negative weights $l$ are dangerous when the non-linear term involves
positive powers, as these will decrease the weight.
A lower bound
on $l$ produces an upper bound on $m$ in the low regularity region $m +
l < 1/2$: if for example one wants to take $l > 0$ then $m < 1/2$.
This is exactly the rub: distributions with Sobolev regularity less
than $1/2$ \textit{cannot} be multiplied together in general, and thus
a good weight forces bad regularity from the perspective of semilinear
equations.  In this section we overcome this problem by adding module
derivatives to the $\Hb^{m, l}$ spaces, as we describe now, using
these in the end to prove Theorems \ref{thm:semilinear} and
\ref{thm:semilinearthree} below. We remark that our use of these
spaces is related to the use of the infinitesimal generators of the
Lorentz group by Klainerman \cite{Klainerman:Uniform} though we work
in a purely $L^2$-based setting while Klainerman proceeds in a $L^\infty
L^2$ setting, and more crucially our module regularity is necessarily
pseudodifferential. As mentioned earlier, results analogous to ours
for the Cauchy problem were obtained by Hintz and Vasy
\cite{HVsemi}; in that case one can use module
generators that are vector fields, and thus one has
in a certain sense a geometric generalization of Klainerman's
vector fields. For the Feynman propagators we then further generalize
this to a microlocal setting.

Elaborating on Proposition \ref{thm:propofsing}, one can also have a version between spaces
with additional module regularity, much as in
\cite[Section~5]{HVsemi}. 
The module regularity is with respect to pseudodifferential
operators characteristic 
on the halves of the conormal bundles of $S_{\pm}$ toward which we
propagate regularity, e.g.\ for $L_{+-}$ they are characteristic on
$\SNb^*_+S_+$ and $\SNb^*_+S_-$.
Concretely, consider the $\Psib^{0}$-module $\cM_{\pm\pm}$ in
$\Psib^{1}$ consisting of b-pseudodifferential operators $A$ whose
b-principal symbols $\sigma_{b,1}(A)$ vanish
 on the components of
$\SNb^*_\pm S_\pm$ at which the domain $L_{\pm \pm}$ has \emph{low}
regularity.
Thus, elements in
\begin{gather*}
  \mathcal{M}_{+-} \mbox{ are characteristic at } \SNb^*_+S_+ \cup
  \SNb^*_+S_-, \mbox{ and in } \\
  \mathcal{M}_{++} \mbox{ are characteristic at } \SNb^*S_+,
\end{gather*}
meaning their principal symbols vanish on the indicated sets.
For an integer $k$ we consider spaces
\begin{equation}
  \label{eq:definitionmodulespaces}
  \Hbpmpm^{m, l,
  k} := \{ u \in \Hb^{m,l} : \cM_{\pm\pm}^ku \subset \Hb^{m,l} \}.
\end{equation}
The $\Hbpp^{m,l,k}$ (and the $--$
whose analysis is essentially identical to the $\Hbpp^{m,l,k}$) thus
have module regularity defined by $\cM_{++}$, which consists of first
order b-pseudodifferential operators that are characteristic on the
b-conormal bundle of $S_{+}$ and are allowed to be b-elliptic at
$S_{-}$.  Thus $\cM_{++}$ admits differential local generators in the following
sense; let $\mathcal{V}_{++}$ denote the $\CI(M)$ module of vector fields $V$
which in the coordinates $\rho, v, y$ satisfy the two conditions: i) in a neighborhood
of $S_{-}$, $V$ is in the $\CI(M)$ span of $\rho \p_{\rho}, \p_{v},
\p_{y}$, i.e.\ $V$ is locally a b-vector field there, and ii) near
$S_{+}$, $V$ is in the $\CI(M)$ span of $\rho \p_{\rho}, \rho \p_{v}, v
\p_{v}, \p_{y}$, i.e.\ $V$ is tangent to $S_{+}$.  Then $\cM_{++}^j u
\subset \Hb^{m,l}  \iff (\forall i\leq j)\ \mathcal{V}_{++}^i u
\subset \Hb^{m,l}$, and thus membership of a distribution in
$\cM_{++}$ can be checked by applying differential and not
pseudodifferential operators.   The $\Hbpp^{m,l,k}$  were
studied in \cite[Section~5]{HVsemi}.  Note that if we
localize near $S_{-}$, since elements of $\cM_{++}$ are not required
to be characteristic on $SN^{*}S_{-}$, we have full b-regularity
to order $m + k$ there, which is to say that if $\chi$ is a cutoff
function supported away from a neighborhood of $S_{+}$, then for $u
\in \Hbpp^{m,l,k}$,  $\chi u \in \Hb^{m + k, l}$.

We have the following regularity result which says that if the right
hand side of $L_{\pm \pm}u = f$ has module regularity then the solution $u$ has
the appropriate corresponding module regularity.  As explained
below \cite[Theorem 5.3]{HVsemi}, the following is a consequence of the
extension of \cite[Proposition 4.4]{Baskin-Vasy-Wunsch:Radiation} obtained in
\cite[Theorem 6.3]{Haber-Vasy:Radial} in the interior case (i.e.\ with no
``b'').  (Recall that $\cY^{m-1,l} = \Hb^{m - 1, l}$ and that $\cX^{m,
  l} \subset \Hb^{m, l}$.)

\begin{thm}\label{thm:modulereginverse}(See \cite[Theorem
  5.4]{HVsemi} for the retarded/advanced propagators; the proof is
  identical in the present general case.)
  Let $g$ be a perturbation of the Minkowski metric in the sense of
  Lorentzian scattering metrics (see Section~\ref{sec:geometry}).  Let $m \colon \Sb^*M \lra \mathbb{R}$,
  $l \in \mathbb{R}$ satisfy the assumptions of Theorem~\ref{thm:fredholm},
  corresponding to a
  particular choice of $\pm \pm$, and let $k \in \mathbb{N}_{0}$, and
  assume that $L_{\pm\pm}^{-1}:\cY^{m-1,l}\to\cX^{m,l}$ exists
  (automatic if $|l|<\frac{n-2}{2}$ and the perturbation is small by Theorem~\ref{thm:invertibility}).
  Then $L_{\pm\pm}^{-1}$ restricts to a bounded map
  $$L^{-1}_{\pm \pm} \colon \Hbpmpm^{m - 1,l,k} \lra \Hbpmpm^{m, l, k}.$$
\end{thm}

Thus, $\Hbpm^{m,l,k}$ is the subspace of $\Hb^{m,l}$ consisting of $u$ such that with $\cM_{+-}$
denoting 1st order b-pseudodifferential operators characteristic on
$\SNb^*_+S_+$ and $\SNb^*_+S_-$, $\cM_{+-}^k u\in\Hb^{m,l}$. Since
elements in $\cM_{+-}$ can be elliptic wherever they are not required
to be characteristic, elements in $\Hb^{m, l, k}$ must have $m+k$
b-derivatives at the other halves of the
conormal bundles, i.e.\ at $\SNb^*_-S_+$ and $\SNb^*_-S_-$. Also notice
that $m+l>1/2$ at the radial sets from which we propagate estimates
implies that $m+l+k>1/2$ for $k\in\Nat$, so the requirements for the
propagation estimates are satisfied there; for the radial sets towards
which we propagate the estimates we still need, and have, $m+l<1/2$ as
the module derivatives are `free' in the sense that one can take $k$
arbitrarily high without sacrificing invertibility. (One could
also use a different normalization, so there are no $k$ additional
derivatives present at the other halves, but one has to be careful
then to make the total weight function behave appropriately; for the
present normalization the previous assumptions on $m$ are the
appropriate ones.)

These observations also yield the following
improvement of the theorem:

\begin{thm}\label{thm:modulereginverse-better}
  Let $g$ be a perturbation of the Minkowski metric in the sense of
  Lorentzian scattering metrics (see Section~\ref{sec:geometry}).  Let $m \colon \Sb^*M \lra \mathbb{R}$,
  $l \in \mathbb{R}$, $k \in \mathbb{N}_{0}$ satisfy the following weakening of the assumptions of Theorem~\ref{thm:fredholm}
  corresponding to a
  particular choice of $\pm \pm$: when $m+l-3/2>0$ is assumed in
  Theorem~\ref{thm:fredholm}, assume merely $m+l+k-3/2>0$.
 Assume that $L_{\pm\pm}^{-1}:\Hbpmpm^{m-1,l}\to\Hbpmpm^{m,l}$ exists
  (automatic if $|l|<\frac{n-2}{2}$ and the perturbation is small by Theorem~\ref{thm:invertibility}).
  Then $L_{\pm\pm}^{-1}$ restricts to a bounded map
  $$L^{-1}_{\pm \pm} \colon \Hbpmpm^{m - 1,l,k} \lra \Hbpmpm^{m, l,
    k}.$$
\end{thm}
The assumptions on the regularity function in the theorem thus
translate to
\begin{equation}
    \label{eq:full-reg-list-2}
\begin{tabular}{|c||c|c|} \hline 
Region & Feynman & Retarded \\ \hline\hline
$\SNb^*_+S_+$ & 
	$m + l < 1/2$ & 
$m + l  < 1/2$ \\ \hline
$\SNb^*_-S_+$ & 
	$m + l + k > 3/2$ &
$m + l < 1/2$ \\ \hline
$\SNb^*_+S_-$ &
	$m + l < 1/2$ &
$m + l + k> 3/2$ \\ \hline
$\SNb^*_-S_-$ &
	$m + l + k > 3/2$ &  $m + l + k > 3/2$ \\ \hline
\end{tabular},
\end{equation}
with analogous conditions for the anti-Feynman and advanced propagators.

\begin{rem}
Note that if $\tilde m$ is a function on $\Sb^*M$ such that $m=\tilde
m$ near the radial sets where the regularity is low and $\tilde
m=m+k$ where the regularity is high, $\tilde m$ monotone along the
bicharacteristics, then $\Hbpmpm^{m, l,  k}\subset \Hbpmpm^{\tilde
  m,l}$, explaining the sense of `restricts' for $L_{\pm\pm}^{-1}$.
\end{rem}

\begin{rem}\label{rem:constant-order}
A very useful consequence of the theorem is that given any $l$ one may take {\em
  constant} $m$ with $m+l<1/2$ and $k$ such that $m+k+l>3/2$, i.e.\
one does not need variable order spaces after all. However, note that
the variable order spaces are extremely useful in obtaining this
conclusion: otherwise at the minimum one would have to do estimates on
the dual spaces of spaces with module regularity for the adjoint
operator; these dual spaces are more difficult to work with.
\end{rem}

\begin{proof}
We simply note that with $\tilde m$ as remarked after the statement of
the proposition, for $f\in \Hbpmpm^{m - 1,l,k}$, $u=L_{\pm\pm}^{-1}f\in
\Hb^{\tilde m,l}$ and microlocally away from the low regularity
regions in fact $u$ is in $\Hb^{m+k,l}$ (as $f$ is in $\Hb^{m+k-1,l}$
and $m+k+l>1/2$ there)
by the standard high regularity radial point statement, Proposition~\ref{thm:propofsing}, and the propagation of
singularities. On the other hand, at the low regularity radial sets,
the extension of \cite[Proposition 4.4]{Baskin-Vasy-Wunsch:Radiation} as
explained
below \cite[Theorem 5.3]{HVsemi} applies, giving that
$u$is in $\Hb^{m,l,k}$ microlocally since $f$ is in $\Hb^{m-1,l,k}$
microlocally, $u$ is in $\Hb^{m+k,l}$ in a punctured neighborhood of
these radial sets, and $m+l<1/2$. All these membership statements come
with estimates, proving the theorem.
\end{proof}

One reason one may want to develop this is to solve nonlinear
equations, as we do in Section \ref{sec:semilinear}. To this end we will be forced to restrict the class of
regularity functions $m$ we consider in the spaces $\Hb^{m, l, k}$ so
that we can keep track of the wavefront sets of products of 
distributions therein.  Specifically, we will assume that, writing $m_+(x) = \max_{\Sb^*_xM}
m$, that
\begin{equation}\begin{aligned}\label{eq:reg-function-condition}
(\forall x) (\forall s <m_+(x)) \quad  \{ \xi\in \Tb^*_x M\setminus o
:\  m(x, \xi) \le s \} \mbox{ is a convex cone},\\
x\in S_+\Rightarrow m|_{\Sb^*_x M}\ \text{attains its minimum on}\ \Sb
N_+^*S_+,\\
x\in S_-\Rightarrow m|_{\Sb^*_x M}\ \text{attains its minimum on}\ \Sb
N_+^*S_-.
\end{aligned}\end{equation}
The first of these conditions says that all of the non-trivial sublevel sets are convex cones
within the fibers of $\Tb^*_xM$. The last two are only important
because of the treatment of module derivatives, where our modules are
characteristic on exactly the two above mentioned sets where the minimum is attained.

For this purpose, one then wants to check the following analogue of
\cite[Lemma~5.4]{HVsemi}:

\begin{prop}\label{thm:algebra}
Assume that $m \colon \Sb^* M \lra \RR$ satisfies \eqref{eq:reg-function-condition}. If furthermore,  $m > 1/2$ and $k
\in \Nat$ satisfies $k > (n - 1)/2$. Then
\begin{equation}\label{eq:lossabovehalf1}
\Hbpm^{m,l_1,k}\Hbpm^{m,l_2,k}\subset \Hbpm^{m - 0,l_1+l_2,k},
\end{equation}
where $\Hbpm^{m - 0,l_1+l_2,k} = \cap_{\epsilon > 0} \Hbpm^{m -
  \epsilon,l_1+l_2,k}$.  Thus for any $\epsilon > 0$ and $m, k$
satisfying the stated conditions, there is a constant $C = C(\epsilon,
m, k)$ such that
$$
\| u v \|_{\Hbpm^{m - \epsilon,l_1+l_2,k}} \le C  \| u \|_{\Hbpm^{m,l_1,k}}
\|  v \|_{\Hbpm^{m,l_2,k}}.
$$

Furthermore, if $m>1/2$ is constant (and thus all assumption on it are
satisfied), one can take $\ep=0$, and drop $-0$ in \eqref{eq:lossabovehalf1}.
\end{prop}

The proof of Proposition \ref{thm:algebra} comes at the end of Section
\ref{sec:balgebraprops} below.  The condition that $m > 1/2 $ can (and
will) be relaxed in Section \ref{sec:semilinear}, but for the moment
we use it to simplify arguments below.

To use this proposition for the semilinear Feynman problems, we will
need to apply it to the spaces $\Hbpm$. For any $l<0$, one can find
a constant $m>1/2$ such that $m+l<1/2$ and a positive integer $k>(n-1)/2$ such
that $m+k+l>3/2$ so by
Remark~\ref{rem:constant-order} the Feynman propagator is applicable
and Proposition~\ref{thm:algebra} is also applicable so that
\begin{equation}\label{eq:reallossabovehalf1-const}
\Hbpm^{m,l_1,k}\Hbpm^{m,l_2,k}\subset \Hbpm^{m,l_1+l_2,k}.
\end{equation}
It is
interesting to note that in fact the stronger requirements of
Theorem~\ref{thm:modulereginverse} can also be arranged:

\begin{cor}\label{thm:algebracor}
For every weight $\ell < 0$, there exists a function $m
\colon \Sb^* M \lra \RR$ such that: 1)  $m > 1/2$, 2) $m, \ell$
satisfy the same assumptions as the forward Feynman condition in the
Theorem~\ref{thm:fredholm}, concretely those in \eqref{eq:full-reg-list}, and 3) $m$ satisfies the
property on the sublevel sets and minima in
\eqref{eq:reg-function-condition}.  For such $m, \ell$ and for $k
\in \Nat$ satisfying $k > (n - 1)/2$,
\begin{equation}\label{eq:reallossabovehalf1}
\Hbpm^{m,l_1,k}\Hbpm^{m,l_2,k}\subset \Hbpm^{m - 0,l_1+l_2,k}.
\end{equation}

In particular, under these assumptions the $p$-fold
products satisfy
\begin{equation}\label{eq:reallossabovehalf2}
(\Hbpm^{m,l,k})^p\subset \Hbpm^{m - 0,pl,k}.
\end{equation}
\end{cor}

\begin{proof}[Proof of Corollary \ref{thm:algebracor} assuming
  Proposition \ref{thm:algebra}]
  Fix $\ell < 0$.  The corollary follows from the proposition by construction of a
  regularity function $m$ satisfying the conditions listed in the
  statement of the corollary.
To do so, fix a constant $m_+>1/2-\ell>1/2$; indeed to satisfy the
strengthened form given in Theorem~\ref{thm:fredholm} take $m_+>3/2-\ell>3/2$. The function $m$ will be
arranged to be equal to $m_+$ except on a small neighborhood $U_+$ of $\SNb_+^* S_+$ and $U_-$ of $\SNb_+^* S_-$, which are the low regularity regions,
where it will be arranged to be smaller. We consider $U_+$; $U_-$ is
analogous. Using any (local) defining functions $\rho_i$, $i=1,\ldots,n+1$, of $\SNb_+^* S_+$ in $\Sb^* M$ and letting $f=\sum\rho_i^2$, the Hamilton
derivative of $f$ is monotone in a neighborhood $U_+$ of $\SNb_+^* S_+$ due to the non-degenerate linearization in the normal
direction (with the size of the neighborhood of course depending on
the choice of the $\rho_i$), with the monotonicity being strict in the
punctured neighborhood. We may assume that $U_+$ is disjoint from any
other component of the radial set; note that if one chooses to, one
may always shrink $U_+$ to lie in any pre-specified neighborhood of
$\SNb_+^*S_+$. One then takes a cutoff function $\phi$, with $\phi\equiv
1$ near $0$, $\phi'\leq 0$ on $[0,\infty)$, $\phi$ supported sufficiently close to $0$ so that
$\phi\circ f$ is compactly supported  in $U_+$; for $c>0$,
$m=m_+-c(\phi\circ f)$ satisfies all monotonicity requirements along
the Hamilton flow, and if $m_+-c<1/2-\ell$, i.e.\ $c>m_++\ell-1/2$,
then the radial point part of the Feynman condition is also
achieved. Note that as $1/2-\ell>1/2$, we may arrange in addition that
$m_-=\inf m=m_+-c>1/2$ by choosing $c<m_+-1/2$. This also gives the
minimum attaining conditions in \eqref{eq:reg-function-condition}. To
arrange the convexity, it is useful to be more definite about the
$\rho_i$: the conormal bundle is $\rho=v=0$, $\zeta=0$, $\eta=0$ where
$\zeta$ is b-dual to $\rho$ and $\eta$ is b-dual to the variables $y$
along $S_+$. Thus, with $\xi'$ the b-dual variable to $v$ (which is
thus non-zero on the conormal bundle minus the zero section) a (local)
quadratic defining function $f$ is
$f=\frac{\zeta^2}{(\xi')^2}+\frac{|\eta|^2}{(\xi')^2}+v^2+\rho^2$. The
convexity requirement for the sublevel sets of $m$ then is implied by
one for those of $f$ (only the ones below sufficiently small positive
values matter), which is thus a convexity condition for the sets
$$
\{(\xi',\zeta,\eta):\
\frac{\zeta^2}{(\xi')^2}+\frac{|\eta|^2}{(\xi')^2}\leq \alpha,\ \xi'>0\},
$$
for all sufficiently small $\alpha>0$, which however certainly
holds. This completes the proof of the corollary.
\end{proof}

\subsection{Microlocal multiplicative properties of Sobolev spaces}
Before we turn to multiplicative properties of b-Sobolev spaces as in
Proposition~\ref{thm:algebra} we first study the analogous properties
of standard Sobolev spaces. Here we need to work with variable order
Sobolev spaces because of the microlocal nature of the spaces
$\Hbpm^{m,l_1,k}$, in particular as $m$ is a function in this case,
and as $k$ gives additional regularity for one half of a conormal
bundle only.

Thus, for a distribution $u$ and for $s\in\CI(S^*\RR^n)$ one defines
$\WF^s(u)$ as in the case of the b-wave front set (so the definitions
agree in the interior): $(p,\xi)\notin\WF^s(u)$ if there exists
$A\in\Psi^0(\RR^n)$ elliptic at $(p,\xi)$ such that $Au\in H^s(\RR^n)$, or
equivalently, if there exists
$A\in\Psi^s(\RR^n)$ elliptic at $(p,\xi)$ such that $Au\in L^2(\RR^n)$. We are
then interested in questions of the kind: for which functions $r\geq
s\geq s_0$ on $S^*\RR^n$ does the implication
$$
u\in H^r,\ v\in H^{s_0}\Rightarrow \WF^s(uv)\subset\WF^s(v),
$$
hold?

By a (Fourier) weight function $w  \colon \RR^n \lra \RR$, we mean a smooth,
measurable, positive
function of polynomial growth, meaning $w \le C \la \xi \ra^N$ for
some $C, N > 0$.  The variable order Sobolev space $H^{(w)}$
corresponding to $w$ is then
\begin{equation}
  \label{eq:variablesobolev}
  H^{(w)} = \{ u \in S'(\RR^n) : w \hat{u} \in L^2(\RR^n) \}.
\end{equation}
Thus $w(\xi) = \la \xi \ra^s$ for $s \in \RR$ defines the standard
Sobolev space of order $s$.   The most common weight function we use
below is of the form $w(\xi) = \la \xi \ra^{s(\hat{\xi})}$ where
$$
\hat{\xi} = \xi / |\xi|,
$$
so  $s$
is a function on the unit sphere, and thus we let
\begin{equation}\label{eq:realsobolevdef}
H^s := H^{(w)}, \mbox{ where } w = \la \xi \ra^s,
\end{equation}
that is, for $s$ of the form $s=s(\hat\xi)$ (which is a rather special
case of $s\in\CI(S^*\RR^n)$!), $H^s$ is the special case of $H^{(w)}$
where $w$ has this form.
The reason such special $s$ are sufficient for us is that
multiplication is a local operation on the base space $\RR^n$, so
using the continuity of a general weight in the base variable, up to
arbitrarily small losses in the order, one may assume that all weights
are in fact dependent only on the Fourier dual variable, cf.\ the
proof of Proposition~\ref{thm:algebra} in
Section~\ref{sec:balgebraprops}, explicitly the containment
\eqref{eq:variable-to-const-order}.

Given an interior point $p \in M$ and local coordinates $x$ near $p$, we write the induced coordinates on
the cotangent space $T^*_pM$ with the variable $\xi$.  The map $\xi
\mapsto \hat{\xi}$ then identifies the spherical cotangent bundle at
$p$, $$S^*_pM := (T_p M \setminus o) / \RR^+, $$ with the
unit sphere in $\mathbb{R}^n$.  Here $o$ denotes the zero section, and the $\RR^+$
action is the natural dilation on the fibers.  Given $\hat
\xi \in S^*_pM$ and $s \in \RR$, the Sobolev
wavefront set of order $s$ at $p$ of a distribution $u$, $\WF^s(u)$,
satisfies $(p, \xi_0) \not \in \WF^s(u)$ if and only if there
is a cutoff function $\chi$ on $M$ supported near $p$ so that $w(\xi)
\widehat{\chi u} \in L^2$ for a weight function $w$ satisfying
\begin{equation}\label{eq:hsmicro}
w(\xi) \ge \chi_2(\hat{\xi}) \la \xi \ra^s
\end{equation}
for $\chi_2$ a cutoff function on
the unit sphere $\mathbb{S}^{n - 1}$ with $\chi_2(\hat{\xi}_0) \equiv
1$ on a neighborhood of $\hat{\xi}_0$ in $\mathbb{S}^{n-1}$.  In particular, $(p, \xi_0) \not \in \WF^s(u)$ if and only if, for
some cutoff function $\chi$ and some $\sfs$ with $\sfs =
\sfs(\hat{\xi})$, $\sfs(\hat{\xi}) \equiv s$ on some open set $U
\subset S^{*}_{p}M$ with $\hat{\xi}_{0} \in U$ and $\sfs \ll 0$ off
$U$, then $\chi u \in H^{\sfs}$.

We are interested in properties of the Sobolev wavefront sets of
products of distributions, and to this end we will exploit \cite[Lemma 4.2]{HVsemi}:
\begin{lemma}\label{thm:algcondition}
  Let $w_1, w_2, w$ be weight functions such that one of the
  quantities
  \begin{equation}
    \label{eq:algcondition}
    \begin{split}
      M_+ := \sup_{\xi \in \RR^n} \int \lp \frac{w(\xi)}{w_1(\eta)
        w_2(\xi - \eta)}\rp^2 d\eta \\
      M_- := \sup_{\xi \in \RR^n} \int \lp \frac{w(\xi)}{w_1(\eta)
        w_2(\xi - \eta)}\rp^2 d\xi 
    \end{split}
  \end{equation}
is finite.  Then $H^{(w_1)} \cdot H^{(w_2)} \subset H^{(w)}$.
\end{lemma}
The most well-known algebra property of Sobolev spaces is that $H^s$
is an algebra provided $s > n/2$.  We now ask, for example, under what
assumption on $r,s,s_0$ does one have
\begin{equation}\label{eq:basicwavefrontproduct}
u\in H^r, v\in H^{s_0} \implies
\WF^s(uv)\subset \WF^s(v),
\end{equation}
i.e.\ if $u$ satisfies an a priori high regularity assumption, and $v$
has a priori not too low regularity, can we conclude that
 if $v$ is $H^s$ microlocally, then
$uv$ is also $H^s$ microlocally? We now provide a partial answer using 
Lemma \ref{thm:algcondition}.  (Note that taking $r = s = s_0
> n/2$ in the following lemma gives the standard statement that $H^s$
is an algebra for $s > n/2$.)
\begin{lemma}
  For distributions $u, v$ (which one may assume to be compactly
  supported due to the locality of multiplication) and $r, s,
  s_0\in\RR$,
  \begin{equation}
    \label{eq:basicassumptionsforproduct}
    r \ge s \ge s_0 >  0 \mbox{ and } r - s + s_0 > n/2 \implies \mbox{
      the containment \eqref{eq:basicwavefrontproduct} holds.}
  \end{equation}
\end{lemma}
\begin{proof}
  Indeed, let $u, v$ be compactly supported distributions and let
  $\xi_0 \neq 0$ have $(x_0, \xi_0) \not \in \WF^{s}(v)$.  By
  definition, there is an open cone $\sC \subset \RR^n\setminus o$
  containing $\xi_0$ and a function $\sfs$ with
$$
\sfs = \sfs(\hat{\xi}) \mbox{ such that }
  \sfs \equiv s \mbox{ on } \sC, \sfs \ge s_0,
$$
and a cutoff function
  $\chi$ supported near $x_0$ such that $\chi v \in H^\sfs$, i.e.\
  $\widehat{\chi v} \la \xi \ra^\sfs \in L^2$.  To see that $(x_0,
  \xi_0) \not \in \WF^s(uv)$ we choose a conic subset $\sK \subset\sC$
  with compact cross section and $\xi_{0} \in \sK$, and let $s' =
  s'(\hat{\xi})$ be such that $s'(\hat{\xi}) \equiv s$ for $\hat{\xi}$
  near $\hat{\xi}_0$, $s' \le s$ everywhere, and such that $s'=s_0$ on
  a conic neighborhood of $\sK^c$.  Thus if $\chi uv \in H^{s'}$ for
  some (possibly different) cutoff $\chi$ with $\chi(x_{0}) \neq 0$,
  then $(x_{0}, \xi_{0}) \not \in \WF^{s}(uv)$.  The argument below
  shows that $uv$ is microlocally $H^{s'}$ outside $\sK$, but we
  concentrate on the statement in $\sK$.  We will apply Lemma
  \ref{thm:algcondition} with $w = \la \xi \ra^{s'}, w_1 = \la \xi
  \ra^r$, and $w_2 = \la \xi \ra^{\sfs}$.  That is, we will show that
  \begin{equation}
    \label{eq:basicmodel}
    H^{r} \cdot H^\sfs \subset H^{s'},
  \end{equation}
  for $r, \sfs, s'$ chosen as above.  Writing
  \begin{equation}
    \label{eq:6}
    I_\xi = \int \lp \frac{w(\xi)}{w_1(\eta)  w_2(\xi - \eta)}\rp^2 d\eta,
  \end{equation}
  we want to show that $\sup_\xi I_\xi \le C < \infty$.

  We first note that for $\xi\in\sK$ this is bounded by the analogous
  expression where $w(\xi)$ is replaced by $\langle\xi\rangle^s$.
  Thus, we first show
  \begin{equation}\label{eq:I-xi-1}
    \sup_{\xi \in \sK} I_\xi = \sup_{\xi\in\sK}\int\Big(\frac{\langle\xi\rangle^{s}}{\langle\eta\rangle^{\sfs(\eta)} \langle\xi-\eta\rangle^r}\Big)^2\,d\eta
  \end{equation}
  is finite.  Since $s \ge 0$ (recall that $s$ is a constant), $\langle\xi\rangle^{2s}\lesssim
  \langle\eta\rangle^{2s}+\langle\xi-\eta\rangle^{2s}$, it suffices to
  prove that
  \begin{equation}\label{eq:twointegrals}
    \sup_{\xi\in\sK}\int\frac{1}{\langle\eta\rangle^{2\sfs(\eta)-2s}
      \langle\xi-\eta\rangle^{2r}}\,d\eta \quad 
    \mbox{ and } \quad \sup_{\xi\in\sK}\int\frac{1}{\langle\eta\rangle^{2\sfs(\eta)} \langle\xi-\eta\rangle^{2r-2s}}\,d\eta
  \end{equation}
  are finite.  We start by looking at the second of these. We break up
  the integral into $\langle\eta\rangle\geq\langle\xi-\eta\rangle$ and
  it complement. In the former region the integral is bounded by
$$
\int\frac{1}{\langle\xi-\eta\rangle^{2r-2s+2\sfs(\eta)}}\,d\eta
$$
and since $r \ge s$ in the latter region by
$$
\int\frac{1}{\langle\eta\rangle^{2r-2s+2\sfs(\eta)}}\,d\eta,
$$
both of which are finite under the assumption in
\eqref{eq:basicassumptionsforproduct} as $\sfs\geq s_0$.

Turning to the first integral in \eqref{eq:twointegrals}, since
$\sfs(\eta)$ is not necessarily greater than or equal to $s$, the
argument of the previous paragraph does not go through -- instead, we
break up the integral into one over $\sC$ and one over $\sC^c$.  Now,
using that $\sfs \equiv s$ on $\sC$,
$$
\int_{\sC}\frac{1}{\langle\eta\rangle^{2\sfs(\eta)-2s}
  \langle\xi-\eta\rangle^{2r}}\,d\eta=\int_{\sC}\frac{1}{
  \langle\xi-\eta\rangle^{2r}}\,d\eta\leq\int \frac{1}{
  \langle\xi-\eta\rangle^{2r}}\,d\eta=\int \frac{1}{
  \langle\eta\rangle^{2r}}\,d\eta
$$
is finite, independent of $\xi$, if $r>n/2$ (which is implied by
$r-s+s_0>n/2$ and $s \ge s_0$). For the integral over $\sC^c$, we use
that there is a constant $C_0>0$ such that
$C_0\langle\xi-\eta\rangle\geq\langle\eta\rangle$ for $\xi\in\sK$ and
$\eta\in\sC^c$. Correspondingly, as $r\geq 0$,
$$
\int_{\sC^c}\frac{1}{\langle\eta\rangle^{2\sfs(\eta)-2s}
  \langle\xi-\eta\rangle^{2r}}\,d\eta\leq C_0^{2r}
\int_{\sC^c}\frac{1}{\langle\eta\rangle^{2\sfs(\eta)-2s+2r}}\,d\eta\leq
C_0^{2r} \int\frac{1}{\langle\eta\rangle^{2s_0-2s+2r}}\,d\eta
$$
which is finite if $r-s+s_0>n/2$. This proves \eqref{eq:I-xi-1}.

The bound for $I_{\xi}$ with $\xi \not \in \sK$ proceeds along the
same lines, where now one can replace $w(\xi)$ by
$\langle\xi\rangle^{s_0}$, and $\langle\eta\rangle^{\sfs(\eta)}$ by
$\langle\eta\rangle^{s_0}$, and is left to the reader.

This completes the proof of \eqref{eq:basicmodel} and thus that the
conditions on $r, s, s_{0}$ in \eqref{eq:basicassumptionsforproduct}
imply \eqref{eq:basicwavefrontproduct}.  
\end{proof}

For our applications, i.e.\ to study the module regularity
defining the spaces $\Hbpmpm^{m,l,k}$ we will first study spaces for
which one has extra regularity in certain directions.  To this end,
we write $\RR^n$ as $\RR^{d + (n - d)}$, i.e.\ we decompose into $x = (x',
x'')$ where $x' \in \RR^d$ and $x'' \in \RR^{n - d}$, and for functions $f$, we write the
Fourier side variable $\xi$ as $(\xi', \xi'')$.  Let
\begin{equation}\label{eq:basicmodulereg}
\cY^{m,a}_d(\RR^{d + (n - d)}) = \{ u : \hat{u} \la \xi \ra^m \la \xi^{''}
  \ra^a \in L^2 \},
\end{equation}
so elements have $m$ total derivatives (derivatives in all variables)
and in addition $a$ derivatives in $x''$.

  \begin{lemma}(\cite[Lemma~4.4]{HVsemi}.)\label{thm:basicalgebra}
    Let $m,a \in \mathbb{R}$.  If $m > d /2$ and $a > (n - d)/2$ then $\cY_d^{m, a}$ is an
    algebra.  If $a,b\geq 0$, $a + b > n - d$, then $\cY_d^{m,a} \cdot \cY_d^{m, b}
    \subset H^{m}$.  
  \end{lemma}
  \begin{proof}
    We begin with the second statement.  This is exactly \cite[Equation 4.6]{HVsemi}, namely using $\la \xi \ra^p \lesssim \la \eta
    \ra^p + \la\xi - \eta
    \ra^p$ for $p \geq 0$, we have
    \begin{equation}
      \label{eq:4.6}
      \begin{split}
       & \int \lp \frac{\la \xi \ra^m}{\la \xi - \eta \ra^m \la \xi'' - \eta'' \ra^a \la \eta  \ra^m \la \eta''
          \ra^b} \rp^2 d\eta \\
        &\le         \int \lp \frac{1}{\la \xi - \eta \ra^m \la \xi'' - \eta'' \ra^a \la \eta''
          \ra^b} \rp^2 d\eta  
        +         \int \lp \frac{1}{ \la \xi'' - \eta'' \ra^a \la \eta  \ra^m \la \eta''
          \ra^b} \rp^2 d\eta \\
        & \le         \int \lp \frac{1}{\la \xi' - \eta' \ra^m \la \xi'' - \eta'' \ra^a \la \eta''
          \ra^b} \rp^2 d\eta  
        +         \int \lp \frac{1}{ \la \xi'' - \eta'' \ra^a \la \eta'  \ra^m \la \eta''
          \ra^b} \rp^2 d\eta ,
      \end{split}
    \end{equation}
and integrating first in the primed and then the double primed
variable shows this integral is uniformly bounded.

When $a = b > (n - d)/2$, estimating the numerator in \eqref{eq:4.6}
in the same way, we see that $\cY_d^{m, a}$ is an algebra
since
    \begin{equation}
      \label{eq:firstalgebra}
      \begin{split}
       & \int \lp \frac{\la \xi \ra^m\la \xi'' \ra^a}{\la \xi - \eta \ra^m \la \xi'' - \eta'' \ra^a \la \eta  \ra^m \la \eta''
          \ra^a} \rp^2 d\eta \\
        &\le \sum_{i, j = 1}^2        \int \lp \frac{f_i g_j}{\la \xi - \eta \ra^m \la \xi'' - \eta'' \ra^a \la \eta''
          \ra^a} \rp^2 d\eta , 
      \end{split}
    \end{equation}
where $f_1 = \la \eta \ra^m$, $f_2 = \la \xi -  \eta \ra^m$ and $g_1 =
\la \eta'' \ra^a$, $g_2 = \la \xi'' -  \eta'' \ra^a$.  We estimate the
$f_1 g_1$ term by
    \begin{equation*}
      \begin{split}
       & \int \lp \frac{\la \eta \ra^m\la \eta'' \ra^a}{\la \xi - \eta \ra^m \la \xi'' - \eta'' \ra^a \la \eta  \ra^m \la \eta''
          \ra^a} \rp^2 d\eta \\
        &\le \int \lp \frac{1}{\la \xi - \eta \ra^m \la \xi'' - \eta''
          \ra^a } \rp^2 d\eta \le \int \lp \frac{1}{\la \xi' - \eta' \ra^m \la \xi'' - \eta'' \ra^a } \rp^2 d\eta,
      \end{split}
    \end{equation*}
so integrating separately in the primed and double primed variable
shows this is uniformly bounded.  The other terms are bounded in
exactly the same way.
  \end{proof}

The wavefront set containment in
\eqref{eq:basicwavefrontproduct}--\eqref{eq:basicassumptionsforproduct} can be
``improved'' when one assumes the distributions lie in the $\cY^{m,
  a}_d$, in the sense that less total regularity (i.e.\ smaller $m$)
is required.
\begin{lemma}
  Let $r, a,m \in \RR$, and let $u \in H^r, v \in \cY_d^{m, a}$.
  Then, provided, $m >
  d/2$, and $a > (n - d)/2$  we have that for any $s \in\RR$ with $r \ge s \ge m + a$, that
  \begin{equation}
    \label{eq:2}
    \WF^{s}(uv) \subset \WF^{s}(v).
  \end{equation}
\end{lemma}
\begin{proof}
Note first that the conditions on $r, s, m, $ and $a$ imply that $r, s
> n/2$ and thus give square integrable weight functions.  Let $\xi_0
\neq 0$ and $(x_0, \xi_0) \not \in \WF^s(v)$.  As above, let $\sC
\subset \RR^n$ be an open,
conic set, such that 
$$
\xi_0 \in \sC \mbox{ and } \{x_0\}\times\sC \subset (\WF^s(v))^{c},
$$
There is a function $\sfs = \sfs(\hat{\xi})$ such that
$\sfs \rvert_{\sC} =
s$, and $\chi  v \in H^{(w_1)}$ where
$$
w_1 = \la \xi \ra^m \la \xi'' \ra^a + \la \xi \ra^{\sfs} .
$$
Furthermore, $\chi$ can be chosen such that $\chi u \in H^{(w_2)}$
where $w_2 = \la \xi \ra^r$.  To show that $(x_0, \xi_0) \not \in
\WF^s(uv),$ we choose $\sK$ a conic subset
containing $\xi_0$, and a function $s'(\hat \xi)$ with
\begin{equation}\label{eq:sfsdescrip}
s'(\hat{\xi})  =  s \mbox{ for  } \hat{\xi}_0 \mbox{ near }
\hat{\xi},  \mbox{ while } s'(\hat{\xi}) = \delta \mbox{ on a
  neighborhood of } \sK^c,   
\end{equation}
where $\delta > 0$ is fixed and small and also $s' \le \sfs$
everywhere.  We
apply Lemma \ref{thm:algcondition} with this $w_1, w_2$ and $w = \la
\xi \ra^{s'}$.

Defining $I_\xi$ as in \eqref{eq:6} with the current $w_1, w_2$ and
$w$ we have
\begin{equation}
  \label{eq:basicmoduleregwave}
I_\xi := \int\Big(\frac{\langle\xi\rangle^s}{(\la \eta \ra^m \la
  \eta'' \ra^a + \la \eta \ra^{\sfs} )
  \langle\xi-\eta\rangle^r}\Big)^2\,d\eta ,
\end{equation}
and we want to know that $M_+ = \sup_\xi I_\xi$ is finite.  Again we use
that $\la \xi \ra^p \lesssim \la \eta \ra^p + \la \xi - \eta \ra^p$
for $p > 0$ to write $I(\xi) \le I_1(\xi) + I_2(\xi)$ where
\begin{equation}
  \label{eq:4}
  \begin{split}
    I_1(\xi) &= \int\Big(\frac{\langle\eta\rangle^s}{(\la \eta \ra^m \la
      \eta'' \ra^a + \la \eta \ra^{\sfs} )
      \langle\xi-\eta\rangle^r}\Big)^2\,d\eta \qquad \mbox{ and }\\
    I_2(\xi) &=    \int\Big(\frac{\langle\xi - \eta\rangle^s}{(\la \eta \ra^m \la
      \eta'' \ra^a + \la \eta \ra^{\sfs} )
      \langle\xi-\eta\rangle^r}\Big)^2\,d\eta.
  \end{split}
\end{equation}
For any $\xi$, since $\langle\xi -
\eta\rangle^s/\langle\xi-\eta\rangle^r \le 1$, 
\begin{equation*}
  \begin{split}
    I_2(\xi) &\le  \int\Big(\frac{1}{\la \eta \ra^m \la
      \eta'' \ra^a + \la \eta \ra^{\sfs} }\Big)^2\,d\eta \le  \int\frac{1}{\la \eta' \ra^{2m} \la
      \eta'' \ra^{2a}}\,d\eta'd\eta'',
  \end{split}
\end{equation*}
which is bounded since $m > d/2$ and $a > (n - d)/2$.

Finally consider $I_1(\xi)$ for $\xi \in \sK$.  We break the integral
up into $\eta \in \sC$ and $\eta \in \sC^c$.  Using that, $s'(\xi)
\le s$, and that $\sfs = s$ on $\sC$ with $s \ge m + a$,
\begin{equation*}
  \begin{split}
    &\int_{\sC}\Big(\frac{\langle\eta\rangle^s}{(\la \eta \ra^m \la
      \eta'' \ra^a + \la \eta \ra^{\sfs} )
      \langle\xi-\eta\rangle^r}\Big)^2\,d\eta \le     \int_{\sC}\frac{1}{
      \langle\xi-\eta\rangle^{2r}} d\eta ,
  \end{split}
\end{equation*}
which is bounded uniformly.  On the other hand, since $\sK$ has
compact cross section and $\xi \in \sK$, we have $\la \eta \ra \lesssim \la
\xi - \eta \ra$ for $\eta \in \sC^c$, so as $s\leq r$,
\begin{equation*}
  \begin{split}
 \int_{\sC}\Big(\frac{\langle\eta\rangle^s}{(\la \eta \ra^m \la
     \eta'' \ra^a + \la \eta \ra^{\sfs} )
      \langle\xi-\eta\rangle^r}\Big)^2\,d\eta &\lesssim     \int_{\sC}\frac{1}{(\la \eta \ra^m \la
      \eta'' \ra^a + \la \eta \ra^{\sfs})^2} d\eta \\
    &  \le   \int_{\sC}\frac{1}{\la \eta' \ra^{2m} \la
      \eta'' \ra^{2a} } d\eta'd\eta'',
  \end{split}
\end{equation*}
and again the last integral is bounded.  Again, we leave the estimate
for $\xi \not \in \sK$ to the reader.
\end{proof}

Furthermore, we have
\begin{lemma}\label{thm:reallemma}
  Let $a,m,s \in \mathbb{R}$.  Let $u, v \in \cY_d^{m, a}$.  Then, provided $m >
  d/2$, and $a > (n - d)/2$ with $s \ge m +
  a$, we have
  \begin{equation}
    \label{eq:waveinsum}
    \WF^{s}(uv) \subset (\WF^{s}(u) + \WF^{s}(v))\cup\WF^s(u)\cup\WF^s(v).
  \end{equation}
  \begin{rem}\label{thm:interior-wave-sum-inclusion}
    Lemma \ref{thm:reallemma} gives the following elaboration on
    \eqref{eq:basicwavefrontproduct}--\eqref{eq:basicassumptionsforproduct}.
    If $u, v \in H^{s_0}$ for $s_0 \in \mathbb{R}$, $s_0 > n/2$, then
    for any $s \in \mathbb{R}$, \eqref{eq:waveinsum} holds.  Indeed,
    for $s_0 > n/2$, $s>s_0$ (the interesting case) and any $d$ one can find $m, a \in \RR$ such that $m > d
    /2$, $a > (n - d)/2$ and $s \ge m + a$, and for any such $H^{s_0} \subset \cY_d^{m,
      a}$, so the lemma applies to $u, v$.
  \end{rem}
\end{lemma}
\begin{proof}[Proof of Lemma \ref{thm:reallemma}]
  The idea behind the proof is the following.  Working above a fixed
  point $x_{0} \in \RR^{n}$, given 
  \begin{equation}\label{eq:notinsum}
\xi \not \in (\WF^{s}(u) +
  \WF^{s}(v))\cup\WF^s(u)\cup\WF^s(v),
\end{equation}
consider the integral
  \begin{equation}\label{eq:etapart}
\int\Big(\frac{\langle \xi \rangle^s}{(\la \eta \ra^m \la
     \eta'' \ra^a + \la \eta \ra^{s_{1}} )
     (\la \xi - \eta \ra^m \la
     \xi'' - \eta'' \ra^a + \la \xi - \eta \ra^{s_{2}} )}\Big)^2\,d\eta,
\end{equation}
where $s_{1}$ and $s_{2}$ are chosen so that $u \in H^{s_{1}}$ and $v
\in H^{s_{2}}$ (near $x_{0}$), and such that $s_{i} = s > n/2$ on sets
that are as large as possible.  If one could take $s_{1} = s$ on
$(\WF^{s}(u))^{c}$ and $s_{2} = s$ on
$(\WF^{s}(v))^{c}$ (one cannot) then the integral would be bounded
since by \eqref{eq:notinsum}, either $\eta \not \in \WF^s(u)$ or $\xi -
\eta \not \in \WF^s(v)$.  In the former case for example, the integral
is bounded by 
$$
\int \frac{1}{\la \xi - \eta\ra^{2m}\la \xi'' - \eta''\ra^{2s}} d\eta,
$$
which is bounded uniformly in $\xi$.

For the formal argument, note that for any open conic sets $\wt{\sC}_{1} \supset
\WF^{s}(u)$ and $\wt{\sC}_{2} \supset
\WF^{s}(v)$ there are functions $s_{1}, s_{2}$ so that $s_{i} \equiv s$ off
$\wt{\sC}_{i}$ and such that $\chi u \in H^{s_{1}}, \chi v \in
H^{s_{2}}$ for some cutoff $\chi$ with $\chi(x_{0}) = 1$.   Given
$\xi$ as in \eqref{eq:notinsum}, assume furthermore that $\xi \not \in
\wt{\sC}_1 + \wt{\sC}_2$, and let $\sC$ be a conic open set with $C
\subset (\wt{\sC}_1 + \wt{\sC}_2)^c$
For $\xi
\in \sC$ where $\sC$ is an arbitrary open subset with $\sC \subset
(\wt{\sC}_{1} + \wt{\sC}_{2})^{c}$.  Writing $\langle \xi \rangle^{2s}
\lesssim \la \eta \ra^{2s} + \la \xi - \eta \ra^{2s}$, we can bound the
integral above by two terms (one with the $\eta$ in the numerator and
the other with the $\xi - \eta$).  The argument to bound each of these
is symmetric so we consider only the $\eta$ term.  Then for $\xi \in
\sC$, the integral in \eqref{eq:etapart}
can be broken up as an integral over $\wt{\sC}_{1}$ and over
$\wt{\sC}_{1}^{c}$.  Over $\wt{\sC}_{1}^{c}$, we have that $\la \eta \ra^m \la
     \eta'' \ra^a + \la \eta \ra^{s_{1}} >(1/2) \la \eta \ra^{s}$, so
     that part of the integral is bounded by
\begin{equation}
  \label{eq:etapart1}
  \begin{split}
    \int\Big(\frac{1}{\la \xi - \eta \ra^m \la
     \xi'' - \eta'' \ra^a + \la \xi - \eta \ra^{s_{2}} }\Big)^2\,d\eta,
  \end{split}
\end{equation}
which is uniformly bounded.  On the other hand, for $\eta \in
\wt{\sC}_{1}$, we have that $s_{2}(\widehat{\xi - \eta}) = s$, since $\xi \in
\sC$, that is $\xi = \xi - \eta + \eta \not \in \wt{\sC}_{1} +
\wt{\sC}_{2}$, so $\xi - \eta \not \in \wt{\sC}_{2}$.  Therefore on
the $\wt{\sC}_{1}$ region, the integral in \eqref{eq:etapart} is
bounded by 
\begin{equation}
  \label{eq:etapart2}
      \int\Big(\frac{\langle \eta \rangle^s}{(\la \eta \ra^m \la
     \eta'' \ra^a + \la \eta \ra^{s_{1}} )
     \la \xi - \eta \ra^{s} }\Big)^2\,d\eta.
\end{equation}
But since $\xi \in \sC$ (so in particular $\xi\notin \wt{\sC}_{1}$), $\la \eta \ra  \lesssim \la \xi - \eta \ra$, so the
integral is uniformly bounded.

Thus we will be able to apply Lemma \ref{thm:algcondition} with $w_{i} = \la \eta \ra^m \la
     \eta'' \ra^a + \la \eta \ra^{s_{i}}$ and $w = \la \xi \ra^{\sfs}$
     with $\sfs = s$ on an arbitrary conic subset $\sK' \subset \sC$ by
     arguing exactly as in the previous lemmata, namely taking $\sfs$
     small but uniformly positive off of $\sK'$ so that both $w/w_{1}$
     and $w/w_{2}$ bounded off $\sC$.  In summary we have shown that
     for such $w, w_{1}, w_{2}$ that
     \begin{equation}
       \label{eq:tobeused}
       H^{(w_{1})} \cdot H^{(w_{2})} \subset H^{(w)},
     \end{equation}
and thus the lemma follows.
\end{proof}

\subsection{Microlocal multiplicative properties of b-Sobolev spaces and module
  regularity spaces}\label{sec:balgebraprops}

Recall that the b-Sobolev
space $\Hb^{m, 0}$ consists of distributions $u$ which are $H^m$
locally in the interior and whose behavior near infinity is as
follows.  We consider a tubular neighborhood of $\p M$ in $M$ (so a
tubular neighborhood of infinity) by $\{ \rho < \epsilon\}$ for some
small $\epsilon > 0$ and boundary defining function $\rho$.  Given a boundary point $p \in \p
M$, we can take coordinates near $p$ in this collar to be of the form
$(\rho, w)$ where $w$ form coordinates on $\p M$.  Defining the
following operation on functions $v$ of
$(\rho, w)$ by
\begin{equation}\label{eq:expchange}
\wt{v}(x, w) = v(e^x, w),
\end{equation}
then $u \in \Hb^{m, 0}$ if in addition to interior regularity, for some cutoff function $\chi$, $\wt{\chi u} \in H^m$.  Here if
$m \colon \Tb^*M \lra \RR$ is a homogeneous degree zero  function (at
least outside a compact set) i.e.\ a function on $\Sb^{*}M$, then we mean $H^m$ as defined in
\eqref{eq:realsobolevdef}, where $m = m(p, \xi)$ and $\xi$ is a
coordinate on $\Tb^*_pM$.
Note that the Fourier transform of $\wt{u}$ is
equal to the distribution obtained by taking the Fourier transform in
the $w$ variables and the Mellin transform (see \eqref{eq:mellindef}) in $\rho$.  The weighted
b-Sobolev spaces are defined by $\Hb^{m, l} = \rho^l \Hb^{m, 0}$.  Given $u \in \Hb^{-N, 0}$ for some $N$, a covector $(p,
\xi) \in \Tb^*_p(M)$ satisfies $(p, \xi) \not \in \WFb^{m,0}(u)$ if
$w(\xi) \widehat{\wt{\chi u}} \in L^2$ for some cutoff function, where
$w$ satisfies \eqref{eq:hsmicro}.  Just as stated above
\eqref{eq:hsmicro}, this is equivalent to having a $\chi$ for which $\chi u \in H^{\sfs}$ where $\sfs =
\sfs(\hat{\xi})$, $\sfs(\hat{\xi}_0) \equiv s$ in a neighborhood of
$\hat{\xi}_0$ and $\sfs \ll 0$ away from
$\hat{\xi}_0$.  Finally, for $u \in \Hb^{-N, l}$ with $l \in \RR$,
\begin{equation}
  \label{eq:weightedwavefront}
  \WFb^{m,l}(u) := \WFb^{m, 0}(\rho^{-l}u).
\end{equation}

Using the previous section we can, for example, easily prove
\begin{lemma}
 Given
$r,s_0,s \in \RR$, then
\begin{equation}\label{eq:basicwavefrontproduct2}
u\in \Hb^{r, l_1}, v\in \Hb^{s_0, l_2} \implies uv \in \Hb^{s_0, l_1 +
  l_2} \mbox{ and }
\WFb^s(uv)\subset \WFb^s(v),
\end{equation}
provided \eqref{eq:basicassumptionsforproduct} above holds, i.e.\
$ r \ge s \ge s_0 \ge 0$ and $r - s + s_0 > n/2$.
\end{lemma}
\begin{proof}
  The proof follows from the paragraph following Lemma
  \ref{thm:algcondition}, as we explain now.  

First let $l_1 = 0 = l_2$.  Given such $u$ and $v$, we want to show
first that $uv \in \Hb^{s_0, l_1 +
  l_2}$.  In the interior of $M$
this follows from \eqref{eq:basicwavefrontproduct} and
\eqref{eq:basicassumptionsforproduct} directly.  For $p \in \p M$, by definition, there is a cutoff function
$\chi$ so that $\wt{\chi u } \in H^{r}$, $\wt{\chi v} \in
H^{s_0}$, where the tilded functions are the functions on the cylinder
defined in \eqref{eq:expchange}. Then $\wt{\chi u } \wt{\chi v} \in
H^{s_0}$ by applying \eqref{eq:basicwavefrontproduct} with $s = s_0$.  

Now we show the wavefront set containment, which is almost identical to
the paragraph following \eqref{eq:basicassumptionsforproduct}.
Indeed, for $(p, \xi_0) \not \in \WFb^{s, 0}(v)$, let $\sC \subset \Tb^{*}_{p}M$ be an open
cone with $\xi_{0} \in \sC$ and $\sC \cap \WFb^{s, 0}(v) =
\varnothing$.  By definition there is
a cutoff $\chi$ supported near $p$ such that $\wt{\chi v} \in H^\sfs$
for some $\sfs$ with $\sfs \equiv s$ on $\sC$, $\sfs \ge s_0$ and such that $\wt{\chi u} \in H^r$.
Then let $\sK \subset \sC$ be a conic set with compact cross section and
$\xi_{0} \in \sK$, and let $s' = s'(\hat{\xi})$ be such that
$s'(\hat{\xi}) \equiv s$ for $\hat{\xi}$ near $\hat{\xi}_0$ and such
that $s'= s_0$ outside $\sK$.
It suffices to show that $\wt{\chi u} \wt{ \chi v} \in H^{s'}$, but
this is exactly \eqref{eq:basicmodel} above.

The statement for $l_{1}$ and $l_{2}$ follows by applying the above
paragraph to $\rho^{-l_1} u$ and $\rho^{-l_2}v$.
\end{proof}

We can now show that the module regularity spaces $\Hbpp$ have the
following algebra property which is closely related to \cite[Section~5.2]{HVsemi}:
\begin{lemma}\label{thm:firstmodulealg}
  Let $m\equiv m_0\in\mathbb{R}$ and let $k \in \Nat$.  Provided $m > 1/2$ and $k > (n
  - 1)/2$, if $u_1 \in \Hbpp^{m,l_1,k}, u_2 \in 
  \Hbpp^{m,l_2,k}$, then $u_1u_2 \in \Hbpp^{m,l_1 + l_2,k}.$
\end{lemma}
\begin{proof}
Away from the boundary this is just the statement that $H^{m + k}$
is an algebra, and at the boundary but away from $S_+$ that
$\Hb^{m+k,l_j}$ has the stated multiplicative property. Thus we assume that the $u_i$ are supported in a small
neighborhood of a point $x \in S_+$.  We begin by showing that
$$
\wt{u_i} \in
\cY^{m, k}_1
$$ 
where $\wt{u}_i$ is defined
as in \eqref{eq:expchange} and $\cY^{m, k}_{1}$ is the spaces
defined in \eqref{eq:basicmodulereg} with $d = 1$, $a = k$. 
Indeed, for our coordinates $(\rho, v, y)$ 
where $\rho$ is a boundary
defining function and $\rho = 0 = v$ defines $S_{+}$ (see Section
\ref{sec:geometry}), recall the
Mellin transform \eqref{eq:mellindef}, and consider the Mellin-Fourier
transform of test functions $\psi \in H^{\infty, \infty}_{b}$,
$\mathcal{M}\mathcal{F}_{v,y}(\psi)$, where $\mathcal{F}_{v,y}$
denotes the Fourier transform in the $v, y$ variables.  Concretely
\begin{equation}
  \label{eq:MellinFourier}
  \mathcal{M}\mathcal{F}_{v,y}(\psi) = \int \rho^{-i \zeta} e^{-i v
    \xi' - i y \cdot \eta} \psi(\rho, v, y) \rho^{-1}d\rho dv dy,
\end{equation}
and we write
\begin{equation}
  \label{eq:7}
  \xi := (\xi', \zeta, \eta), \qquad \xi'' := (\zeta, \eta),
\end{equation}
so $\xi$ is the total dual variable and $\xi'$ is dual to $v$.
Consider order $m$ elliptic b-pseudodifferential operator $A$ defined by \begin{equation}
  \label{eq:melliptic}
  A\psi = \mathcal{F}^{-1}\mathcal{M}^{-1}_{0} \la \xi \ra^{m} \mathcal{M}\mathcal{F}_{v,y} \psi,
\end{equation}
and the order $\leq k$ b-pseudodifferential operator $B_\alpha$,
$|\alpha|\leq k$, defined by
\begin{equation}
  \label{eq:k}
  B_\alpha\psi = \mathcal{F}^{-1}\mathcal{M}^{-1}_{0} (\xi'')^{\alpha} \mathcal{M}\mathcal{F}_{v,y} \psi.
\end{equation} 
By definition of $\Hbpp^{m,0,k}$ we have $AB_\alpha u_i
\in L^{2}_{b}$ for all $|\alpha|\leq k$, so
since the mellin transform of $u$ is the Fourier transform in $x =
\log \rho$ of $\wt{u}$ we have $\wt{u_i} \in \mathcal{Y}^{m,k}_{1}$
locally near $x$, as claimed.

To prove the lemma, we must show that if $a, b, c > 0$ integers,
$\alpha$ a multiindex,
and $a + b + c+|\alpha| \le k$, we have that $(\rho \p_\rho)^a(\rho \p_v)^b(v
\p_v)^c \pa_y^\alpha (u_1 u_2) \in \Hb^m$, but this distribution is equal to 
\begin{equation}\begin{aligned}
  \label{eq:leibnitz}
 \sum_{a'\leq a, b'\leq b,c'\leq c,\alpha'\leq\alpha} C_{a', b', c',\alpha'} &((\rho \p_\rho)^{a'}(\rho \p_v)^{b'}(v
\p_v)^{c'} \pa_y^{\alpha'}u_1)\\
&\qquad\qquad((\rho \p_\rho)^{a - a'}(\rho \p_v)^{b - b'}(v
\p_v)^{c - c'} \pa_y^{\alpha-\alpha'}u_2),
\end{aligned}\end{equation}
for some combinatorial constants $C_{a', b', c',\alpha'}$ (which depend on $a, b, c,\alpha$).  In
each of these terms we have the product of two elements $u_1, u_2$,
which $u_i \in \Hbpp^{m, 0, k - r_i}$ where $k - r_1 + k - r_2 \ge
k$.  But by the previous paragraph, locally near $S_+$, the $u_i$
satisfy that $\wt{u}_{i}$ lies in $\cY_d^{m, k -
  r_i}$.  Thus by the first
part of Lemma \ref{thm:basicalgebra}, $\wt{u_1} \wt{u_{2}} \in H^{m}$,
which is to say that
$u_{1} u_{2} \in \Hb^{m,0}$, locally near $S_+$, which is what we wanted in the case $l_1
= l_2 = 0$.  For general weights, i.e.\ $u_i, \in \Hbpp^{m, l_i, k}$,
$i = 1, 2$, apply the above arguements to $\rho^{- l_1 - l_2} u_1 u_2
= (\rho^{-l_1} u_1 )(\rho^{-l_2} u_2)$.
\end{proof}

Finally, we can prove Proposition \ref{thm:algebra} above.

\begin{proof}[Proof of Proposition \ref{thm:algebra}]
Using that multiplication is local, we will reduce in the end to
considering two compactly supported
distributions $u_i$, $i = 1, 2$, with $u_i  \in \Hb^{m, l_i, k}(M)$
supported in a neighborhood of a point $x \in M$.

While in fact the boundary case discussed below handles this as well,
we first treat interior points. So assume that the $u_i$ are supported in a
coordinate chart in $M^\circ$.
 Thus the $u_i \in H^{m + k}(\mathbb{R}^n)$, and since $m > 1/2, k > (n - 1)/2$,
 there is a constant $m_0 \in \mathbb{R}$ so that the $u_i \in H^{m_0 +
   k}(\mathbb{R}^n)$ and $m_0 + k > n/2$.  Since $H^{m_0 + k}$ is an
 algebra, $u_1 u_2 \in H^{m_0 + k}$.  We claim that $u_1 u_2 \in H^{m -
   \epsilon + k}$.  Indeed, for any $x$ and for any $s \in \mathbb{R}$ with $s - k < m_+  =
 \max_{S_x^*\RR^n} m$, the sets $\WF^s(u_i)$
 satisfy
\begin{equation}\label{eq:variable-to-const-order}
\WF^s(u_i) \cap T^*_x(\RR^n) \subset \{ (x, \xi ) : m(x,
\hat{\xi}) + k \le s\}.
\end{equation}
By Remark \ref{thm:interior-wave-sum-inclusion},
$$
\WF^s(u_1 u_2) \subset (\WF^s(u_1) + \WF^s(u_2)) \cup \WF^s(u_1) \cup \WF^s(u_2),
$$
and thus, \emph{by the assumption that the non-trivial sublevel sets
  are convex,} we conclude that $\WF^s(u_1 u_2) \cap T^*_x\RR^n$ is also a subset of $\{ (x, \xi ) : m(x,
\hat{\xi}) + k\le s\}$.  To see that $u_1 u_2 \in H^{m - \epsilon +
  k}$ then, for any $(x,\xi)$
  let $s = m(x, \xi) - \epsilon/2 + k$, and then note that since $(x, \xi) \not
  \in \WF^s(u_i)$ for $i = 1, 2$, by what we just said also $(x,
  \xi) \not \in \WF^s(u_1 u_2)$.

For $x \in \p M$ and the $u_i$ supported near $x$, we first assume
that $l_1 = 0 = l_2$.  Since such $u_i$ are also contained in
$\Hbpp^{m_-, 0, k}$, by Lemma \ref{thm:firstmodulealg} we know that $u_1 u_2\in \Hbpp^{m_-,0,k}$.
Due to the second assumption in \eqref{eq:reg-function-condition},
given $\ep>0$, microlocally near $\Sb N_+^*S_+$ (with the neighborhood
size depending on $\ep$), $\Hbpp^{m_-,0,k}$
is contained in $\Hbpm^{m-\ep,0,k}$, so microlocally near $\Sb
N_+^*S_+$ we have the conclusion of the proposition (if
$l_1=l_2=0$). Thus, it remains to consider points $q \in \SNb^*
\mathbb{R}^n$ away from $\Sb
N_+^*S_+$, but there the microlocal membership of  $\Hbpm^{m-\ep,0,k}$ is
equivalent to not being an element of $\WFb^{m+k-\ep}$.
Now $(x, \xi) \in \WFb^{s,
  0}(u_i)$ if and only if $\xi \in \WF^s(\chi \wt{u_i})$ for each
$\chi$ with $\chi \equiv 1$ near $x$.  But then by Lemma
\ref{thm:reallemma} we have \eqref{eq:waveinsum} with $\wt{u}_i$
replacing $u, v$.  Thus the same statements hold for this product as
for the interior case, and the same argument shows that their product
is in $\Hb^{m+k- \epsilon, 0}$, completing the proof of the
proposition in $l_1=l_2=0$.
The weights are multiplicative, establishing the proposition apart
from the constant $m$ case. In the case of constant $m$, we only need
to observe that in all the arguments we can take $\ep=0$.
\end{proof}

\subsection{A semilinear problem}\label{sec:semilinear}

Using the above, one has a complete analogue of the semilinear
results of \cite{HVsemi}. Concretely, one can conclude that the
Feynman problem for the equation
\begin{equation}\label{eq:semilinearequ}
\Box_g u+\lambda u^p=f
\end{equation}
is well-posed for appropriate $p \in \mathbb{N}$ (and small $f$), which includes $p\geq 3$ if $n\geq
4$, so in particular the not-yet-second-quantized $\varphi^4$ theory is
well-behaved on these curved space-times.

\begin{thm}\label{thm:semilinear}
Suppose $g$ is a perturbation of Minkowski space in the sense of
  Lorentzian scattering metrics (see Section~\ref{sec:geometry}) so
  that in particular Theorem~\ref{thm:modulereginverse} holds.
Let $p\in \mathbb{N},\lambda \in \RR$ with $p$ and the dimension $n$
satisfying
\begin{equation}
  \label{eq:weightcondition3-thm}
   \frac{2}{p-1} < \frac{n - 2}{2}.
\end{equation}
Let $\lprime < 0$ satisfy
\begin{equation}
  \label{eq:19-other}
\lprime \in \Big(\frac{2}{p-1} - \frac{n - 2}{2} , 0\Big).
\end{equation}
Note the interval is
non-empty. Let $k>\frac{n-1}{2}$ be an integer and $m$ a function on
$\Sb^*M$ given by Corollary~\ref{thm:algebracor}, or instead take
$m>1/2$ constant and
$k > \frac{n - 1}{2}$ integer with $m+l<1/2$, $m+l+k>3/2$ which exist if \eqref{eq:19} holds. Then there is a
constant $C>0$ such that the small-data Feynman problem for
\eqref{eq:semilinearequ}, i.e.\ given 
$$
f\in \Hbpm^{m-1,\lprime+(n-2)/2+2,k} \mbox{ with norm } < C
$$ 
finding $u\in\Hbpm^{m,\lprime+(n-2)/2,k}$  satisfying the equation, is
well-posed, and
$u$ can be calculated as the limit of a Picard iteration corresponding
to the perturbation series.

In particular the above holds for $p \ge 4$ and $n \ge 4$.
\end{thm}

\begin{rem}
  As mentioned, the condition on $p$ and $n$ in \eqref{eq:weightcondition3-thm} holds in
  particular if $p \ge 4$ and $n \ge 4$.  It holds also if $p = 3$ and
  $n \ge 5$ and when $n = 3, p \ge 6$, \emph{but fails} for $p = 3, n=4$.  We tackle this case in
  Theorem \ref{thm:semilinearthree} below.
\end{rem}

\begin{rem}\label{rem:semilinear-retarded}
Our argument also works for the retarded and advanced problems
considered in \cite[Section~5]{HVsemi}, but it gives a somewhat
different result since the multiplicative properties we use are
somewhat different, as necessitated by the microlocal nature of the
spaces that have to be used for the Feynman problems. For the retarded/advanced problems
\cite[Section~5.4]{HVsemi} considers general first order
semilinearities. In the case of no derivatives an analogous result is
shown there under the
constraint $p>1+\frac{3}{n-2}$ which is weaker than
\eqref{eq:weightcondition3-thm}. For $n=4$ the difference is whether
$p=3$ is admissible; it borderline fails our inequality. However, as noted above, in
Section~\ref{subsec:intricate-semilinear} we improve our result by
showing better multiplicative properties to handle $n=4$, $p=3$.
\end{rem}

\begin{rem}
By using $m$ constant as in the statement of the theorem we can
consider first order derivative nonlinearities as well, provided in
addition $m>3/2$ (so that for the first derivatives the b-order is
$m-1>1/2$, note that constant $m$ is crucial since we cannot afford
the $\ep>0$ losses in multiplication);
the natural assumption is that these are of the form of a finite sum
of products of vector fields $V_j\in\Vsc(M)$ applied to $u$, times $u^p$: $u^p(V_1
u)\ldots(V_{q}u)$, $p+q\geq 2$. Writing these as $V_j=\rho W_j$, $W_j\in\Vb(M)$, we
can proceed as in \cite[Section~5.4]{HVsemi} and as we proceed below
in the proof, provided that we can take some $l<-1$ (necessitated by
$m>3/2$, see Remark \ref{rem:constant-order}) for the Feynman
propagator invertibility considerations. This requires $n\geq 5$ so
that $\frac{n-2}{2}>1$. The numerology paralleling \eqref{eq:weightcondition1}-\eqref{eq:weightcondition2} is then that nonlinearities
satisfying
$$
(p-1)\frac{n-2}{2}+q\frac{n}{2}+(p+q)l-2\geq l
$$
can be handled by the same method. One can satisfy this with $l$ sufficiently close to $-1$,
$l<-1$, if
$$
(p-1)(n-4)+q(n-2)> 2,
$$
which is weaker than the requirement $(p-1)(n-4)+q(n-2)> 4$ of
\cite[Equation~(5.15)]{HVsemi}, but recall that here we need $n\geq
5$. In particular, as our results also apply for the advanced/retarded
problems, see Remark~\ref{rem:semilinear-retarded}, if $n\geq 5$, this slightly
improves the result of \cite{HVsemi}, allowing e.g.\ $q=1$ and $p=1$ for
$n=5$; $q=1,p=1$ is not allowed in \cite{HVsemi} even after
improvements discussed there in Remark~5.16 using analogues of our
improvements in Section~\ref{subsec:intricate-semilinear}.
\end{rem}

\begin{proof}[Proof of Theorem~\ref{thm:semilinear}]
  As in \cite[Section 5]{HVsemi}, moving the $\lambda
  u^p$ to the right hand side, we rewrite
  \eqref{eq:semilinearequ} as
  \begin{equation}
    \label{eq:Lsemilinearequ}
    L \wt{u} =  \wt{f}  -  \lambda \rho^{- 2 + (p - 1)(n - 2)/2}
    \wt{u}^p, 
  \end{equation}
where $\wt{u} := \rho^{-(n - 2)/2} u, \wt{f} := \rho^{- 2 -(n - 2)/2}
f$.  Assuming that $f \in \Hbpm^{m-1,\lprime+(n-2)/2+2,k}$, we have $\wt{f}
\in \Hbpm^{m-1,\lprime,k}.$  To apply a Picard iteration to
\eqref{eq:Lsemilinearequ}, we want the right hand side to be in the
domain of the forward Feynman inverse of $L$, 
$L_{+-}^{-1} \colon \cY^{m - 1, \lprime, k} \lra \cX^{m, \lprime, k}$
(where $\cY^{m, \lprime, k}, \cX^{m, \lprime, k}$ are defined as
in \eqref{eq:spaces} with the $\Hb^{m, l}$ replaced by $\Hb^{m, l, k}$),  so by Theorem
\ref{thm:modulereginverse} we want it in $\Hbpm^{m - 1, \lprime, k}$ where $m
+ \lprime$ now satisfies the defining properties of the Feynman
propagator and such that $|\lprime| < (n - 2)/2$.  Furthermore, we want to apply the algebra properties in
Proposition \ref{thm:algebra}; in particular we assume that $m
> 1/2$ everywhere.  Note
that
we have both
\begin{equation}\label{eq:conditions}
m > 1/2 \mbox{ and } m + \lprime < 1/2\ \text{near}\ \Sb N_+^*S_+\ \text{and}\ \Sb N_+^*S_-,
\end{equation}
i.e.\ in the low regularity
  regions.
Thus $\lprime$ (which is just a real number) \emph{must be negative}, and furthermore for any $\lprime <
0$ there is a function $m$ meeting all of the criteria of
Corollary~\ref{thm:algebracor}, in particular
both the criteria in
\eqref{eq:conditions} and the Feynman criteria (since $m$ increases
as one approaches the high regularity regions $\Sb N_-^*S_-$ and
$\Sb N_-^*S_+$), as well as the convexity/minima criteria
\eqref{eq:reg-function-condition}. (Indeed, note that by the remarks
preceding Corollary~\ref{thm:algebracor}, we could even take $m$
constant if we use Theorem~\ref{thm:modulereginverse-better} in place of
Theorem~\ref{thm:modulereginverse}. In this case we can take $\ep=0$ below.) Under these assumptions, by Proposition
\ref{thm:algebra}, for $k > (n - 1)/2$ we have that $\wt{u}^p$ lies in
$\Hbpm^{m - \epsilon, p\lprime, k}$ for any $\epsilon > 0$, and thus $\rho^{- 2 + (p - 1)(n - 2)/2}
    \wt{u}^p$ lies in $\Hbpm^{m - \epsilon, \llprime, k}$ where
\begin{equation}
  \label{eq:weightcondition1}
  \llprime = - 2 + (p - 1)(n - 2)/2 + p \lprime,
\end{equation}
and $\Hbpm^{m - \epsilon, \llprime, k} \subset \Hbpm^{m - 1, \lprime , k}$ if and
only if
\begin{equation}
  \label{eq:weightcondition2}
\llprime \ge \lprime \iff   \lprime \ge \frac{2}{p-1} -  \frac{n - 2}{2} ,
\end{equation}
where again $\lprime$ is an arbitrary negative number.  

For any $p, n$ such that
\begin{equation}
  \label{eq:weightcondition3}
   \frac{2}{p-1} < \frac{n - 2}{2},
\end{equation}
taking 
\begin{equation}
  \label{eq:19}
  \llprime = \lprime \in \Big(\frac{2}{p-1} - \frac{n - 2}{2} , 0\Big)
\end{equation}
sufficiently small, and $m$ picked correspondingly as above,
we claim that for every $\delta > 0$, there is an $R \ge 0$ such that
if both $\| \wt{u} \|_{\Hb^{m, \lprime, k}}$ and $\| \wt{v} \|_{\Hb^{m,
    \lprime, k}}$ are bounded by $R$ then
\begin{equation}
  \label{eq:17}
  \| \rho^{-2 + (p-1)(n-2)/2} \wt{u}^p - \rho^{-2 + (p-1)(n-2)/2}
  \wt{v}^p \|_{\Hb^{m-1, \lprime, k}} \le \delta   \|  \wt{u} -  \wt{v}
  \|_{\Hb^{m, \lprime, k}}.
\end{equation}
Assuming the claim for the moment, we see that the map 
\begin{equation*}
  \wt{u} \mapsto L_{+-}^{-1} (\wt{f}  + \lambda \wt{u}^p \rho^{-2 + (p-1)(n-2)/2})
\end{equation*}
is a contraction mapping on $\Hb^{m, \lprime, k}$ and thus the Picard
iteration $\wt{u}_{n + 1} =  L_{+-}^{-1} (\wt{f}  + \lambda \wt{u}_{n}^p \rho^{-2 + (p-1)(n-2)/2})$ with $\wt{u}_1 = 0$ converges if
$\wt{f}$ is sufficiently small in $\Hb^{m - 1, \lprime, k}$ (as assumed
in the theorem).

Thus it remains only to prove the claim.  For any $\lprime$ and for any
$\ep>0$ we have
\begin{equation}\label{eq:almostdone}
  \begin{split}
    \| \wt{u}^p - \wt{v}^p \|_{\Hb^{m - \epsilon, p\lprime, k}} &= \Big\| (\wt{u} - \wt{v})
    \sum_{j = 0}^{p-1} \wt{u}^j \wt{v}^{p-1-j} \Big\|_{\Hb^{m-\epsilon, p\lprime, k}} \\
    &\le C \| \wt{u} - \wt{v} \|_{\Hb^{m, \lprime, k}} \max(\| \wt{u}
    \|_{\Hb^{m, \lprime, k}}, \| \wt{u} \|_{\Hb^{m, \lprime, k}})^{p-1}
  \end{split}
\end{equation}
provided
$m-(p-2)\mu>1/2$.
Since $\lprime$ satisfies \eqref{eq:weightcondition2} with $\llprime$ as in
\eqref{eq:19}, by bounding the $\Hb^{m, \lprime, k}$ norm with the
$\Hb^{m, \llprime, k}$ norm,
\begin{equation*}
  \| \rho^{-2 + (p-1)(n-2)/2} \wt{u}^p - \rho^{-2 + (p-1)(n-2)/2}
  \wt{v}^p \|_{\Hb^{m-\epsilon, \lprime, k}}  \le     \| \wt{u}^p - 
  \wt{v}^p \|_{\Hb^{m-\epsilon, p\lprime, k}},
\end{equation*}
and combining with \eqref{eq:almostdone} gives the claim once $\mu>0$
is taken sufficiently small.
\end{proof}

\subsection{More intricate multiplicative properties and cubic
  semilinear problems for $n=4$}\label{subsec:intricate-semilinear}

To extend to $p = 3$, $n = 4$, we need improvements of the regularity
properties for products which allow us to take the weight $\llprime$ to be
greater than zero.  To do so and still have $m + \llprime < 1/2$ in the low
regularity zone, we need $m < 1/2$, which is below the regularity
threshold in the work in Section \ref{thm:basicalgebra}; thus we need
improvements of the results therein.  The necessary improvements are
based on the ideas in the following. 
\begin{lemma}
  Let $s, s' ,s_0\in \RR$.  Let $u, v \in H^{s_0}$, then $\WF^{s'}(uv) \subset \WF^s(v)$, provided
  $s \ge s_0 \ge s'$ and $s - s' + s_0 > n/2$. 
\end{lemma}
The point here is that one can take $s_0 < n/2$, and obtain a result
for $uv$ which says it is in a worse Sobolev space then $H^{s_0}$
microlocally provided $v$ is in a better one microlocally.

The proof in fact follows the first of the product regularity
arguments above, namely that \eqref{eq:basicassumptionsforproduct}
implies \eqref{eq:basicwavefrontproduct}.  Consider a point $\xi_0$
with $(x_0,\xi_0)
\not \in \WF^s(v)$, and a function $\sfs\geq s_0$ which equals $s$ on an open
cone $\sC \subset (\WF^s(v))^{comp}$ and take $\sfs' \equiv s_0$
outside some compact set $\sK \subset \sC$, $\sfs'\equiv s'$ near
$\xi_0$ with $\sfs'\leq\sfs$ everywhere; as we show this implies that $H^{\sfs} \cdot
H^{s_0} \subset H^{\sfs'}$. Indeed, this is analogous to
\eqref{eq:basicmodel} above, and we break the relevant integral
$I_\xi$ up in the same way as in \eqref{eq:twointegrals}, so we must bound integrals
$$
\sup_{\xi \in K}\int\frac{1}{\langle\eta\rangle^{2\sfs-2\sfs'}
  \langle\xi-\eta\rangle^{2s_0}}\,d\eta \quad 
\mbox{ and } \quad \sup_{\xi\in \sK}\int\frac{1}{\langle\eta\rangle^{2\sfs} \langle\xi-\eta\rangle^{2s_0-2\sfs'}}\,d\eta.
$$
The second integral is bounded by the arguments above, and for the
first integral, the only difference is that over the set $\sC$, using
that $\sfs - \sfs'\geq 0$ there, we have
$$
\int_{\sC}\frac{1}{\langle\eta\rangle^{2\sfs-2\sfs'}
  \langle\xi-\eta\rangle^{2s_0}}\,d\eta \le \int_{\sC} \frac{1}{
  \langle\xi-\eta\rangle^{2(\sfs - \sfs' +  s_0 )}}\,d\eta   \le \int_{\sC} \frac{1}{
  \langle\xi-\eta\rangle^{2(s - s' +  s_0 )}}\,d\eta ,
$$
which is finite since $s - s' + s_0 > n/2$.  The rest of the estimates
are exactly as in the previous case.

Applying this line of thinking to the model spaces $\cY^{m, a}_d$
defined in \eqref{eq:basicmodulereg}, we can obtain a regularity
result for products which allows us to dip under the threshold $d / 2$
above.
\begin{lemma}\label{thm:reallemmaimproved}
  For $m, m', m_0,  a \in \RR$ such that $m - m' + m_0 > d/2$, $a > (n
  - d) / 2$ and $m \ge m_0 \ge m'$, we have $\cY^{m_0, a}_d \cdot
  \cY_d^{m, a} \subset \cY_d^{m', a}$.  Furthermore,
  \begin{equation}
    \label{eq:10}
    \WF^{m' + a}(uv) \subset (\WF^{m + a}(u) + \WF^{m + a}(v))\cup \WF^{m + a}(u) \cup \WF^{m + a}(v).
  \end{equation}
\end{lemma}
\begin{proof}
  To see that the first conclusion holds, we argue as in Lemma
  \ref{thm:basicalgebra}, and thus use the inequality
    \begin{equation}
      \label{eq:notalgebra}
      \begin{split}
       & \int \lp \frac{\la \xi \ra^{m'}\la \xi'' \ra^a}{\la \xi - \eta \ra^{m_0} \la \xi'' - \eta'' \ra^a \la \eta  \ra^{m} \la \eta''
          \ra^a} \rp^2 d\eta \\
        &\le \sum_{i, j = 1}^2        \int \lp \frac{f_i g_j}{\la \xi - \eta \ra^m \la \xi'' - \eta'' \ra^a \la \eta''
          \ra^a} \rp^2 d\eta , 
      \end{split}
    \end{equation}
where $f_1 = \la \eta \ra^{m'}$, $f_2 = \la \xi -  \eta \ra^{m'}$ and $g_1 =
\la \eta'' \ra^a$, $g_2 = \la \xi'' -  \eta'' \ra^a$.  Replacing the
unprimed variable with primed variables and using that both
\begin{equation*}
  \int \lp \frac{\la \eta' \ra^{m'}}{\la \xi - \eta' \ra^{m_0} \la
    \eta'  \ra^{m} } \rp^2 d\eta', \qquad   \int \lp \frac{\la \xi' - \eta' \ra^{m'}}{\la \xi - \eta' \ra^{m_0} \la \eta'  \ra^{m} } \rp^2 d\eta'
\end{equation*}
are uniformly bounded under the stated assumptions on $m, m', m_0$
gives the statement.

The wavefront set containment is obtained by locating similar
improvements in the proof of Lemma \ref{thm:reallemma}.  Indeed, as
there, we have that $u \in H^{(w_1)}, v \in H^{(w_2)}$ where $w_i(\xi) =
\la \xi \ra^{m_0}\la \xi'' \ra^a + \la \xi \ra^{s_i}$ where
$s_i = s_i(\xi)$ where the $s_i \equiv m + a$ off open conic sets
$\wt{\sC}_i$ are arbitrary open sets containing, respectively, 
$\WF^{m + a}(u)$ and $\WF^{m + a}(v)$. We 
want to show that given
$$
\xi \not \in (\WF^{m + a}(u) + \WF^{m +
  a}(v))\cup \WF^{m + a}(u) \cup \WF^{m +
  a}(v)
$$
and a proper choice of function $s$ with $s = m' + a$ near $\xi$ that $uv \in H^{s}$,
which amounts to applying Lemma \ref{thm:algcondition} with $w = s$ and
$w_1, w_2$ exactly as in Lemma \ref{thm:reallemma}.

We thus want to bound an integral similar to \eqref{eq:etapart},
namely
\begin{equation*}
  \int\Big(\frac{\langle \xi \rangle^s}{(\la \eta \ra^{m_0} \la
     \eta'' \ra^k + \la \eta \ra^{s_{1}} )
     (\la \xi - \eta \ra^{m_0} \la
     \xi'' - \eta'' \ra^k + \la \xi - \eta \ra^{s_{2}} )}\Big)^2\,d\eta,
\end{equation*}
  We choose
$s$ so that $s \le m + a$ and bound
$\la \xi \ra^{2s} \lesssim \la \eta \ra^{2(m + a)} + \la \xi - \eta
\ra^{2(m + a)} $ and as usual break the integral into two parts
involving the two terms on the right of this bound.  Again we focus on
the $\la \eta \ra^{2s}$ term.  Integrating first over $\wt{\sC}_1$ and then
$(\wt{\sC}_1)^{comp}$; over $\wt{\sC}_{1}^{c}$, we have that $\la \eta \ra^m \la
     \eta'' \ra^k + \la \eta \ra^{s_{1}} >(1/2) \la \eta \ra^{s}$, so
     that part of the integral with $\la \eta \ra^{2s}$ in the
     numerator  over
     $\wt{\sC}_1$  is bounded by
    \begin{gather*}
      \int\Big(\frac{1}{\la \eta \ra^{s_1 - s} \la \xi - \eta \ra^m
        \la \xi'' - \eta'' \ra^k + \la \xi - \eta \ra^{s_{2}}
      }\Big)^2\,d\eta \\
    \le       \int\Big(\frac{1}{\la \eta' \ra^{m - m'} \la \xi' - \eta' \ra^{m_0}
        \la \xi'' - \eta'' \ra^k },
      \Big)^2\,d\eta .
    \end{gather*}
which, by separating into primed and double primed coordinates and
using the assumptions on $m, m_0, m'$  is uniformly bounded.  The rest
of the bounds proceed analogously and are left to the reader.
\end{proof}

By reducing locally and arguing exactly as in the proof of
Proposition~\ref{thm:algebra} and Corollary~\ref{thm:algebracor}, we obtain
\begin{prop}\label{thm:bwaveimproved}
  For $\lprime$ sufficiently small, there exists $m \colon \Sb^* M \lra \mathbb{R}$ satisfying: 1) that $m \ge 1/2 -
  \delta$ for some $\delta \in (0, 1/2)$, 2) $m, \lprime$ satisfy the forward
  Feynman condition in the strengthened form given in Theorem~\ref{thm:fredholm}, and 3) $m$ satisfies the condition on the
  sublevel sets and minima in \eqref{eq:reg-function-condition}.
  Moreover, for $k
  \in \Nat$ satisfying $k > (n - 1)/2$, 
$$
\Hbpm^{m,l_1,k}\Hbpm^{m, l_2,k}\subset
\Hbpm^{m - 2\delta-0, l_1+l_2, k}.
$$
In particular, for $\delta$ sufficiently small $(\Hbpm^{m,l_1,k})^3 \subset \Hbpm^{m - 4\delta-0, l_1+l_2, k}.$
\end{prop}

We can now finally prove
\begin{thm}[$p = 3, n = 4$]\label{thm:semilinearthree}
Suppose $g$ is a perturbation of Minkowski space in the sense of
  Lorentzian scattering metrics (see Section~\ref{sec:geometry}), in particular so that Theorem
\ref{thm:modulereginverse} holds.
In dimension $n = 4$, given $\lambda \in \RR$ and a weight
$\lprime \ge 0$ and $\lprime$ sufficiently small, and a regularity
function $m$ as in Proposition \ref{thm:bwaveimproved}, there
is $C>0$ such that the small-data Feynman problem, i.e.\ given 
$$
f\in \Hbpm^{m-1,\lprime+(n-2)/2+2,k} \mbox{ with norm } < C
$$ 
finding $u\in\Hbpm^{m,\lprime+(n-2)/2,k}$  satisfying 
\begin{equation}
  \label{eq:smalldata3}
  \Box_{g,+-} u + \lambda u^3 = f,
\end{equation} 
is
well-posed with $k > (n -1 )/2$, and
$u$ can again be calculated as the limit of a Picard iteration
corresponding
to the perturbation series.  
\end{thm}

\begin{rem}
  Again by Remark \ref{rem:constant-order}, there is a version of this
  theorem with $m$ constant, $m \ge 1/2 - \delta$, $l \ge 0$, $m + l
  < 1/2$ and $k > (n - 1)/2$ satisfying $m + l + k > 3/2$.  We leave
  the details of this straightforward modification of the following
  proof to the reader.
\end{rem}

\begin{rem}
Although we do not state it here explicitly, Proposition
  \ref{thm:bwaveimproved} also gives improvements to the statement of Theorem \ref{thm:semilinear} for
  other $n,p$ in terms of the spaces in which solvability holds (what
  $\lprime$ can be),
  though not for whether there {\em is} a space of the kind considered
  there in which solvability holds.
\end{rem}

\begin{proof}
 The proof is identical to that of Theorem \ref{thm:semilinear}
  incorporating the improvements given by Proposition
  \ref{thm:bwaveimproved}.   We
  take $\lprime \ge 0$ and find an $m$ such that $m + \lprime$ satisfies the
  Feynman condition and $m > 1/2 - \delta$ for some small  $\delta >
  0$.   Rewriting the equation as in
  \eqref{eq:Lsemilinearequ} with $\wt{f}$ and $\wt{u}$ defined in the
  same way, and assuming that $\wt{f} \in \Hbpm^{m, \lprime, k}$, the
  condition that 
$$
\rho^{2 - (p - 1)(n-2)/2}\wt{u}^p =  \rho^{4 - n}\wt{u}^3 \in
\Hbpm^{m - 5\delta, \lprime, k} \subset \Hbpm^{m - 1, \lprime, k}
$$
is now that $\delta$ be less than $1/5$ and that
\begin{equation}\label{eq:pthreeweightcondition}
  \lprime \le 4 - n + 3 \lprime \iff \lprime \ge  n/2 - 2.
\end{equation}
If $n = 4$, we can thus find $\lprime$ and $m$ satisfying the Feynman
conditions and \eqref{eq:pthreeweightcondition} simultaneously.  
From now on we assume that $n = 4$.  The existence of an $m$
satisfying the conditions in \eqref{eq:reg-function-condition} is a
trivial modification of the proof of Corollary \ref{thm:algebracor}.

The Picard iteration argument is now identical to that in Theorem
\ref{thm:semilinear} except incorporating the loss in Proposition \ref{thm:bwaveimproved}.  In this case, the claim in \ref{eq:17} is
substituted by the following: for every $\delta > 0$ sufficiently small, there is an $R \ge 0$ such that
if both $\| \wt{u} \|_{\Hb^{m - \delta, \lprime, k}}$ and $\| \wt{v}
\|_{\Hb^{m - \delta,
    \lprime, k}}$ are bounded by $R$ then
\begin{equation*}
  \|\wt{u}^3 -\wt{v}^3 \|_{\Hb^{m - 5\delta, \lprime, k}} \le \delta
  \|  \wt{u} -  \wt{v} \|_{\Hb^{m - \delta, \lprime, k}}.
\end{equation*}
This and the rest of the proof follow exactly as in Theorem
\ref{thm:semilinear} using the improvement in Proposition
\ref{thm:reallemmaimproved}.
\end{proof}

\bibliographystyle{plain}
\bibliography{feynman}

\end{document}